\documentclass[11pt]{compositio}
\usepackage{amssymb,amsmath,latexsym,lipsum}
\usepackage[inline]{enumitem}
\usepackage{tabulary}
\usepackage{booktabs}
\usepackage{upgreek}
\usepackage{charter}
\usepackage[title]{appendix}
\usepackage{longtable}

\newcommand{\RN}[1]{%
  \textup{\uppercase\expandafter{\romannumeral#1}}%
}

\usepackage{amsthm}
\usepackage{euscript}
\usepackage{color}
\usepackage[T1]{fontenc}
\usepackage[utf8]{inputenc}
\usepackage{enumitem}
\usepackage[singlespacing]{setspace}
\usepackage[colorinlistoftodos]{todonotes}
\usepackage{tikz}
\usetikzlibrary {positioning}
\definecolor {processblue}{cmyk}{0.96,0,0,0}
\usepackage{bookmark}
\usepackage{multirow,bigdelim}
\usepackage{stackengine}
\usepackage{mathrsfs}
\usepackage{mathtools}
\usepackage{tikz-cd}
\usepackage{pgfplots}
\usepackage{pgf}
\usetikzlibrary{arrows}

\let\emptyset\varnothing

\usepackage{subcaption}

\pgfplotsset{compat=1.15}

\newtheorem{theorem}{Theorem}[section]
\newtheorem{proposition}[theorem]{Proposition}
\newtheorem{lemma}[theorem]{Lemma}
\newtheorem{corollary}[theorem]{Corollary}
\newtheorem{definition}[theorem]{Definition}
\theoremstyle{remark}
\newtheorem{example}[theorem]{Example}
\theoremstyle{remark}
\newtheorem{remark}[theorem]{Remark}
\newtheorem{question}[theorem]{Question}
\newtheorem{conjecture}[theorem]{Conjecture}

\DeclareMathOperator{\spec}{Spec}

\DeclareMathOperator{\aut}{Aut}

\DeclareMathOperator{\ord}{ord}

\DeclareMathOperator{\et}{et}

\DeclareMathOperator{\an}{an}

\DeclareMathOperator{\sw}{sw}

\DeclareMathOperator{\Frac}{Frac}

\DeclareMathOperator{\sep}{sep}

\DeclareMathOperator{\Def}{Def}

\DeclareMathOperator{\Alg}{Alg}

\DeclareMathOperator{\gal}{Gal}

\DeclareMathOperator{\dsw}{dsw}

\DeclareMathOperator{\st}{st}

\DeclareMathOperator{\rsw}{rsw}

\DeclareMathOperator{\Proj}{Proj}

\DeclareMathOperator{\ind}{ind}

\DeclareMathOperator{\ab}{ab}

\DeclareMathOperator{\ness}{ness}

\DeclareMathOperator{\reduce}{red}

\DeclareMathOperator{\characteristic}{char}

\DeclareMathOperator{\Der}{Der}

\DeclareMathOperator{\cohom}{H}

\DeclareMathOperator{\ASW}{ASW}

\begin{document}

\title{Deforming cyclic covers in towers}

\author{Huy Dang}
\email{huydang1130@ncts.ntu.edu.tw}
\address{National Center for Theoretical Sciences, National Taiwan University, Taipei City 106, TAIWAN (R.O.C)} 

\thanks{The author is supported by the Vietnam Institute for Advanced Study in Mathematics and the Simons Foundation Grant Targeted for
Institute of Mathematics, Vietnam Academy of Science and Technology.}

\begin{abstract}
Obus-Wewers and Pop recently resolved a long-standing conjecture by Oort that says: every cyclic cover of a curve in characteristic $p>0$ lifts to characteristic zero. Sa{\"i}di further asks whether these covers are also ``liftable in towers''. We prove that the answer for the equal-characteristic version of this question is affirmative. Our proof employs the Hurwitz tree technique and the tools developed by Obus and Wewers.
\end{abstract}

\maketitle

\tableofcontents

\section{Introduction}
\label{secintroduction}

Throughout this paper, we assume that $k$ is an algebraically closed field of characteristic $p>0$. An \textit{Artin-Schreier-Witt $k$-curve} is a smooth projective connected $k$-curve $Y$ which is a $\mathbb{Z}/p^n$-cover of the projective line $\mathbb{P}^1_k$. When $n=1$, we call $Y$ an \textit{Artin-Schreier} curve. 

Classically, one may study an object in characteristic $p$ by finding a ``link'' of it with characteristic zero. For instance, Grothendieck showed that every (smooth, projective, connected) curve $Y$ over $k$ ``lifts'' to a curve $\mathcal{Y}$ over a finite extension of the ring of Witt vectors $W(k)$, hence in characteristic zero. In addition, the prime-to-$p$ part of the {\'e}tale fundamental group of $Y$ is equal to that of the generic fiber of $\mathcal{Y}$ \cite[XIII, Cor.2.12]{MR0217087}, which can be easily calculated from the topological fundamental group of its corresponding one over $\mathbb{C}$ due to the Riemann existence 
theorem (see, e.g., \cite{MR82175}). Note that the $p$-parts of the fundamental groups are not the same. For instance, when $Y=\mathbb{A}^1_k \cong \spec k[x]$, the geometric fundamental group of its lift's generic fiber $\mathbb{A}^1_{\overline{\Frac W(k)}} \cong \spec \overline{\Frac W(k)}[X]$ is trivial as $\pi_1(\mathbb{A}^1_{\mathbb{C}})$ is. However, $\pi_1^{\text{{\'e}t}}(\mathbb{A}^1_k)$ is non-trivial as there always exists an {\'e}tale  $\mathbb{Z}/p$-cover defined by the equation $y^p-y=x$. In fact, $\pi_1^{\text{{\'e}t}}(\mathbb{A}^1_k)$ is infinitely generated as there exists for each $n \in \mathbb{Z}_{>0}$ a $\mathbb{Z}/p^n$-cover of $\mathbb{A}^1_k$ (see \S \ref{seccycliccover}).

It is thus natural to ask whether one can lift a Galois cover of curves to charactersitic zero. Thanks to a local-to-global principle \cite[\S 3]{MR1424559}, one may restrict the study of Galois covers of curves to Galois extension of power series (see also \S \ref{seclocalglobal}). The answer, in general, is \emph{no} \cite[\S 1.1]{MR927980}. We call a group $G$ such that every local $G$-cover in characteristic $p$ lifts a \textit{local Oort} group for $p$. In fact, if a group $G$ is a local Oort group for $p$, then $G$ is either cyclic, dihedral of order $2p^n$, or the alternating group $A_4$ (with $p=2$) \cite{MR2919977}. One would naturally expect that the converse holds. When $G$ is cyclic, it is known as the \textit{Oort conjecture}, which first appeared in 1995 in a list of questions and conjecture published in \cite[Appendix 1]{MR3971540}, and was settled recently. The first result are due to Oort, Sekiguchi, and Suwa, who showed that the conjecture holds for all $\mathbb{Z}/pm$-covers, where $(m,p)=1$ \cite{MR1011987}. Green and Matignon then proved that for the case $G \cong \mathbb{Z}/p^2m$ \cite{MR1645000}. Finally, a combined effort of Obus-Wewers and Pop resolved the conjecture in \cite{MR3194815} and \cite{MR3194816}. Sa{\"i}di conjectured a more general version of this result \cite[Conj-0-Rev]{MR3051252}, which says that a lift of a sub-cover of a given $G$-cyclic cover can be extended to a lift of the cover itself. That conjecture holds given a definite answer to the following.
\begin{question}[(\textit{The refined local lifting problem})]
\label{questionrefinedlocallifting}
Let $k[[z]]/k[[x]]$ be a $G$-Galois extension where $G$ is cyclic. Suppose we are given a discrete valuation ring $R$ in characteristic zero and a lift $R[[S]]/R[[X]]$ of a subextension $k[[s]]/k[[x]]$. Does there exist a finite extension $R'$ of $R$ in characteristic zero with residue field $k$ and a $G$-Galois extension $R'[[Z]]/ \allowbreak R'[[X]]$ that lifts $k[[z]]/k[[x]]$ and contains $R'[[S]]/R'[[X]]$ as a sub-extension?
\end{question}  

As in the standard local lifting problem for cyclic groups, one may also assume that $G=\mathbb{Z}/p^n$ \cite[Proposition 6.3]{MR3051249}. We are tackling question \ref{questionrefinedlocallifting} by following the approach from \cite{MR3194815}.
Let us briefly describe how the Oort conjecture was proved. Obus and Wewers first proved a general result stating that a cyclic cover lifts if it has no \textit{essential ramification} (\cite[Theorem 1.4]{MR3194815}, see also Definition \ref{defessentialjump}). Pop completed the proof by showing that every cover that cannot be lifted by Obus and Wewers admits an equal-characteristic deformation whose generic fibers have no essential ramification and thus also lift to characteristic zero. The existence of these non-trivial equal-characteristic-deformations, which change the number of branch points but fix the genus, is also a unique aspect of wildly ramified covers. That gives us another way to investigate a cover in characteristic $p$: finding a connection of it with a slightly different one via equal-characteristic deformation. The main result of this paper is a positive answer to the global equal-characteristic analog of Question \ref{questionrefinedlocallifting}.

\begin{theorem}
\label{thmdeformtowers}
Suppose $\phi: Z \xrightarrow{} X$ is a cyclic $G$-Galois cover of curves over $k$, and $\psi: Y \xrightarrow{} X$ is its $H$-Galois sub-cover (where $H$ is a quotient of $G$). Suppose, moreover, that $\Psi: \mathcal{Y}_R \xrightarrow{} \mathcal{X}_R$ is a deformation of $\psi$ over a complete discrete valuation $R$ of characteristic $p$. Then there exists a finite extension $R'/R$, and a deformation $\Phi: \mathcal{Z}_{R'} \xrightarrow{} \mathcal{X}_{R'}$ of $\phi$ over $R'$ that contains $\Psi \otimes_R R': \mathcal{Y}_{R'} \xrightarrow{} \mathcal{X}_{R'}$ as a sub-cover. That is, one can always fill in the following commutative diagram of cyclic Galois covers,
\begin{equation*}
\begin{tikzcd}
\spec k \arrow[d] & X \arrow[d] \arrow[l] \arrow[d] & Y \arrow[l, "\psi"] \arrow[d] & Z \arrow[l, "\lambda"] \arrow[d, dotted] \\
\spec R   & \mathcal{X} \arrow[l] & \mathcal{Y} \arrow[l, "\Psi"] & \mathcal{Z} \arrow[l, dotted, "\Lambda"]
\end{tikzcd}
\end{equation*}
where $\lambda$ is the factor of $\phi$ through $Y$, after maybe a finite extension of $R$.
\end{theorem} 

To prove the theorem, we adapt the techniques from \cite{MR3194815}. One may first reduce the problem to the case where $\phi$ is a one-point cover. Its deformation $\Phi$ can be then regarded as a cover of a rigid disc over $R$ with good reduction (see \S \ref{secreduction}). To achieve the right reduction on the disc's boundary, we continuously control the degeneration of $\Phi$'s restrictions from inside to outside. These information are kept track of in $\Phi$'s \textit{Hurwitz tree}. It is a combinatorial-differential object which has the shape of the dual graph of a semi-stable model of the cover, together with the degeneration data of some restrictions of that cover at the corresponding vertices. The degeneration datum of a cover is derived from its \textit{refined Swan conductor}. The conductor was defined by Kato \cite{MR991978} in 1989. Since then, it has been studied by many other authors (Matsuda, Tsuzuki, Saito, Abbes, Kedlaya, Xiao, Chiarelotto, Pulita, Leal, Thatte, \ldots ) from various points of view (see \cite{MR1465067},\cite{MR1324636}, \cite{MR1617130},\cite{MR1925338},\cite{MR2904193},\cite{MR2567506}, \cite{MR3726102}, \cite{MR3484149}), and has plenty of applications.

\begin{remark}
The notion of Hurwitz tree was initially formulated for covers in mixed characteristic to tackle the lifting problem for $\mathbb{Z}/p$-covers by Henrio \cite{2000math.....11098H}. It was later improved by Bouw, Brewis, and Wewers \cite{MR2254623}, \cite{MR2534115}. In \cite{2020arXiv200203719D}, we characterize Hurwitz trees for $\mathbb{Z}/p$-covers of a rigid disc in \textit{equal} characteristic and use them to classify equal-characteristic-$\mathbb{Z}/p$-deformations \cite[Theorem 1.2]{2020arXiv200203719D}. 
\end{remark}


\begin{remark}
Understanding these deformations also equates to understanding the geometry of the moduli space of cyclic covers of fixed genus. Capitalizing on this fact, we show that the moduli space of Artin-Schreier covers of fixed genus $g$ is connected when the integer $g$ is sufficiently large by explicitly constructing some local $\mathbb{Z}/p$-equal-characteristic-deformations \cite[Theorem 1.1]{DANG2020398}. Furthermore, knowing the geometry of a moduli space, in turn, allows the study of invariants over flat families of objects parameterized by the space. The most well-known among them are $p$-rank, $a$-number, Ekedahl-Oort type, and Newton polygons (see, e.g., \cite[Chapter 6]{MR3971540}). Recently, there has been a consistent stream of papers about how these invariants behave for Galois covers (see, e.g., \cite{MR2049830} \cite{MR3466874} \cite{kato_leal_saito_2019}, \cite{MR4113776}), especially cyclic covers.  We will briefly discuss some of these applications along the way.
\end{remark}

\subsection{Outline of Theorem \ref{thmdeformtowers}'s proof}

Firstly, thanks to a local-to-global principle, we show that Theorem \ref{thmdeformtowers} holds if and only if its local version does (Theorem \ref{theoremmainlocal} and Proposition \ref{proplocalglobalintowers}). That means we can restrict ourselves to the case when $\phi$ is a cyclic extension of a power series over $k$. Furthermore, one may assume that $G \cong \mathbb{Z}/p^n$ for some $n \in \mathbb{Z}_{\ge 1}$, and $H \cong \mathbb{Z}/p^{n-1}$ (Theorem \ref{thminductionmain}). In that setting, the $G$-extensions are described by Artin-Schreier-Witt (ASW) theory, which is briefly discussed in \S \ref{secASW}. Finally, a construction by Katz and Gabber allows us to go back to the global case, i.e., when $\phi$ is a one-point-cover of the projective line (\S \ref{secdeformonepointcover}), which further simplifies some calculations.  

At this point, we can translate our problem to a finding-conditions-for-good-reduction of an $t$-adic projective line's cover as below.

\begin{proposition}[(Proposition \ref{propmainonepointcover})]
\label{propgoodreductiontowers}
Suppose $n \ge 2$ is an integer, $R=k[[t]]$, $\phi_n: \overline{Y}_n \rightarrow \mathbb{P}^1_k$ is a one-point-$\mathbb{Z}/p^n$-cover, and $\phi_{n-1}: \overline{Y}_{n-1} \rightarrow \mathbb{P}^1_k$ is $\phi_n$'s $\mathbb{Z}/p^{n-1}$-sub-cover. Suppose, moreover, that $\Phi_{n-1}: {Y}_n \rightarrow \mathbb{P}^1_R$ is a $\mathbb{Z}/p^{n-1}$ cover whose reduction (modulo $t$) is isomorphic to $\phi_{n-1}$. Then there exists a finite extension $R'/R$, and a $\mathbb{Z}/p^n$-cover $\Phi_n: {Y}_n \otimes_R R' \rightarrow \mathbb{P}^1_{R'}$ that extends $\Phi_{n-1}$ and has a special fiber isomorphic to $\phi_n$.
\end{proposition}

One may then assume that $\phi_n$'s generic branch locus lies inside a $t$-adic disc $D \subset (\mathbb{P}^{1}_{\Frac(R')})^{\an}$. We prove the above proposition in 4 steps as follows.

\begin{enumerate}[label=Step \arabic*]
    \item \label{step1main} We start by studying the refined Swan conductors of Artin-Schreier-Witt (ASW) covers of a $t$-adic disc over $K$ (recall that $K=\Frac R$) (\S \ref{secSwan}), which measure the degeneration of these covers. Those invariants usually have the form $(\delta, \omega)$, where $\delta$ is a non-negative rational number and $\omega$ is a differential form over $k(x)$. A key result is Theorem \ref{theoremCartierprediction}, which gives some rules on the degeneration of $\Phi_n$ when that of $\Phi_{n-1}$ are known. That is the equal-characteristic analog of \cite[Theorem 1.2]{MR3167623}.
    \item \label{step2main} The information from the previous step then allow us to define the Hurwitz tree $\mathcal{T}_{n-1}$ associated to $\Phi_{n-1}$ (\S \ref{secHurwitz}). That is a directed tree which encodes in its structure the geometry of $\Phi_{n-1}$'s branch points. Each vertex $v$ of $\mathcal{T}_{n-1}$ is equipped with a pair $(\delta_v, \omega_v)$, which is the refined Swan conductor of $\Phi_{n-1}$'s restriction to certain sub-disc of $D$. Each leaf has a conductor, which is determined by the ramification data of $\Phi_{n-1}$'s generic fiber. This generalizes the one for Artin-Schreier cover in \cite{2020arXiv200203719D}. We further show how the information from a Hurwitz tree can tell whether the corresponding cover has good reduction or not (Proposition \ref{propinverse} and Remark \ref{remarkhurwitzdetectgoodreduction}).
    \item \label{step3main} In this step, we show that one can always construct explicitly a $\mathbb{Z}/p^n$-tree $\mathcal{T}_n$ that ``extends'' $\mathcal{T}_{n-1}$ in the sense of the criteria from \ref{step1main} (Proposition \ref{propextendtree}). Moreover, the structure of $\mathcal{T}_n$ is specially designed to be used in the final step.
    \item \label{step4main} Finally, using $\mathcal{T}_n$ as a frame, we construct a cover $\Phi_n$ that proves Proposition \ref{propgoodreductiontowers}. Roughly speaking, we start from a candidate for $\Phi_n$ that verifies the degeneration data asserted by $\mathcal{T}_n$ on certain sub-discs of $D$ corresponding to the ``leaves'' of $\mathcal{T}_n$, where it is much easier to achieve (\S \ref{seccontrolfinal}). We then continuously modify that cover buy multiplying it to certain ``controlling characters'' (\S \ref{seccontrollingchars}) along $\mathcal{T}_n$ (\S \ref{seccontroledge}) until we get to its ``root'' (\S \ref{seccontrolroot}), which corresponds to the boundary of the $t$-adic disc. We then obtain a $\mathbb{Z}/p^n$-cover whose degeneration data coincide with that of $\mathcal{T}_n$. The fact that $\Phi_n$ has a good reduction isomorphic to $\phi_n$ is then immediate from what we learn in \ref{step2main}.
\end{enumerate}

\begin{remark}
The modifying process in \ref{step4main} is influenced by Obus and Wewers' \cite{MR3194815}, which in turn is inspired by \cite{MR1645000}. To who are familiar with that work, we will present a comparison between the two techniques as well as provide more details of this step in \S \ref{seccomparison}. An overview of the construction can also be found in \S \ref{sectinduction}.
\end{remark}

\begin{remark}
One major difference between this paper and \cite{MR3194815}, besides the characteristic of the ring $R$, is that the tree $\mathcal{T}_n$ of the latter has no ``branches'' that one need to control (that follows from the ramification breaks hypothesis of \cite[Theorem 1.4]{MR3194815}). Therefore, the Hurwitz tree technique is not utilized in the work, even they used it to acquire the intuition and the idea for the main strategy. We, however, prove Proposition \ref{proppartition}, which basically says that one can ``partition'' $\Phi_n$ at each vertex of $\mathcal{T}_n$ so that it is sufficient to modify the ``part'' of the cover corresponding to a sub-tree of $\mathcal{T}_n$ that has no branches at a time. Hence, we can alter many techniques from Obus and Wewers' to fit our situation.
\end{remark}

\begin{remark}
There are three main obstacles that prevent us to answer Question \ref{questionrefinedlocallifting} using the strategy from this manuscript. \begin{enumerate}
    \item Firstly, it is hard to compute the Swan conductors of a Kummer cover (with given equation) of a disc in mixed characteristic, unlike the equal-characteristic case as discussed in \S \ref{sectionrefinedSwanASWleal}. \item Adapting \ref{step3main} is also an issue, as finding the differential forms to fit in the tree $\mathcal{T}_n$ is no longer as natural as shown in \S \ref{secsolutioncartierequation}. We are able to do so when $p=2$ and $n=2$, though.
    \item Finally, \ref{step4main} also becomes much more complicated because the  controlling characters of \S \ref{seccontrollingchars} are no longer straightforward to build. Our current method, which is developed from \cite{MR3194815}, requires showing that certain square matrices, whose sizes can be arbitrarily large, are invertible to prove the existences of such characters. 
\end{enumerate}
\end{remark}

\subsubsection{\ref{step4main}}
\label{seccomparison}
As discussed above, we may restrict ourselves to the situation of Proposition \ref{propgoodreductiontowers}. By Artin-Schreier-Witt theory (\S \ref{seccycliccover}), the cover $\phi_n$ (resp. $\Phi_{n-1}$) could be presented by a length-$n$ (resp. length-$(n-1)$) Witt-vector $\underline{g}_n:=(g^1, g^2, \ldots, g^n) \in W_n\big(k(x)\big)$ (resp. $\underline{G}_{n-1}=(G^1, \ldots, G^{n-1}) \in W_{n-1}\big(K(X)\big)$, where $K:=\Frac R$). In addition, one may assume that a cover $\Phi_n$ verifying the proposition should have the form $\underline{G}_n:=(G^1, G^2, \ldots , \allowbreak G^{n-1}, G^n)\in W_{n}\big(K(X)\big)$. Therefore, it suffices to show the existence of such a rational $G^n \in K(X)$. Corollary \ref{corgood} then asserts that it suffices to construct a $G^n$ so that $\Phi_n$
\begin{enumerate}[label=(C\arabic*)]
    \item \label{condphi_n1} has an {\'e}tale reduction (Definition \ref{defnetalegoodreduction}), and
    \item \label{condphi_n2} its generic fiber's covering space has the right genus, and
    \item \label{condphi_n3} the Witt-vector representing the special fiber is in the same \textit{ASW class} with $\underline{g}_n$.
\end{enumerate}
The construction of such a $G^n$ can be summarized as follows. 
\begin{enumerate}[label=Step 4.\arabic*]
    \item \label{step4.1main} We first assume that $\underline{G}_{n-1}$ is best (Definition \ref{defnbest}), and set $G^n=0$. That makes the genus of $\Phi_n$'s generic fiber ``minimal''. This fact allows us to easier manage the generic ramification.
    \item \label{step4.2main} Recall from \ref{step3main} that $\mathcal{T}_{n-1}$ is a Hurwitz tree arising from $\Phi_{n-1}$, and $\mathcal{T}_n$ is one that ``extends'' $\mathcal{T}_{n-1}$. Suppose $v$ is a vertex of $\mathcal{T}_{n}$. Then, by construction, there is a matching one in $\mathcal{T}_{n-1}$, which we also call $v$. That vertex corresponds to a closed sub-disc $\mathcal{D}_v$ of the $t$-adic disc $\spec R[[X]]$. Suppose moreover that there are $m$ edges or leaves $e_1, e_2, \ldots, e_m$ pointing out from $v$ in $\mathcal{T}_{n-1}$. Each $e_i$ represents an open annulus or an open disc $A_{e_i}$ contained in $\mathcal{D}_v$ and has the same outer radius. We then break down
    \begin{equation}
    \label{eqnpartition}
        \underline{G}_{n-1}=:\underline{G}_{n-1,v, \infty}+ \sum_{i=1}^m \underline{G}_{n-1, v,i},
    \end{equation}
    where $\underline{G}_{n-1,v, \infty}$ (resp. $\underline{G}_{n-1,v,i}$) has no poles inside $\mathcal{D}_v$ (resp. outside the disc formed by the outer radius of $A_{e_i}$). In addition, the data on each vertex of $\mathcal{T}_{n-1}$ starting from $e_i$ coincides with that of one formed by $\underline{G}_{n-1, v, i}$ (Proposition \ref{proppartition}). This fact allows us to modify $G^n$, hence $\underline{G}_n$, inductively along the tree $\mathcal{T}_n$ starting from the leaves and ending at the root. That process equates to controlling the degeneration of $\Phi_n$ from particular subdiscs to the whole disc. The four following steps will construct $G^n$ by doing inductions on the vertices and edges of $\mathcal{T}_n$.
    \item \label{step4.3main} Consider a ``final vertex'' $v$, which is adjacent to the leaves $\{b_1, \ldots, b_m\}$ with conductors $\iota_{1,n}, \ldots, \iota_{m,n}$ respectively of $\mathcal{T}_n$. One can modify the part of $G^n$ at $v$ by adding polynomials whose poles are corresponding to the $b_i \in B_v$ and with degree at most $\iota_{i,n}-1$. The result is a cover whose degeneration on $\mathcal{D}_v$ coincide with that of $\mathcal{T}_n$ at $v$. That is \ref{step 1} in \S \ref{sectinduction}. Note that \ref{condphi_n2} is achieved in this step. The below items correspond to \ref{step 2} to \ref{step 5}.
    \item \label{step4.4main} Look at an edge $e$ adjacent to the final vertex $v$. Note that the previous step gives an existence of a $G^n$ giving rise to the right degeneration at $v$. Furthermore, \ref{step4.2main} allows us to assume that $G^n$ has only poles inside $\mathcal{D}_v$. It hence can be represented by a power series that converges outside $\mathcal{D}_v$. Using an analytical technique and the comparison tool developed in \S \ref{secdetect}, we show that one can modify $G^n$ in such a way that the right degeneration occurs at the starting vertex of $e$, settle the base case. The method is parallel to one in \cite[\S 6.4]{MR3194815}.
    \item \label{step4.5main} Let $v$ be a non-final vertex, which is the target of $e$ and the start of $e_1, \ldots, e_m$. Recall that we may apply (\ref{eqnpartition}) to partition $\underline{G}_{n-1}$ into a sum, where each summand $\underline{G}_{n-1, v,i}$ correspond to a $\mathbb{Z}/p^{n-1}$-cover whose branch points lie inside the disc $D_{e_i}$ formed by the outer radius of $A_{e_i}$. By induction, we obtain, for each $i$, a $G_{v,i}^n \in K(X)$ that extends $\underline{G}_{n-1, v,i}$, and whose sum gives the right degeneration at $v$.
    \item \label{step4.6main} This step is just a repetition of \ref{step4.3main} to an edge $e$ that is adjacent to the vertex $v$ above, completing the inductive step. We hence obtain, by induction, the right degeneration everywhere but the root of $\mathcal{T}_n$. Condition \ref{condphi_n1} is hence fulfilled.
    \item \label{step4.7main} Finally, we adapt \cite[\S 6.5]{MR3194815} to attain the right reduction at the root of $\mathcal{T}_n$ hence on the whole disc. Condition \ref{condphi_n3} is thus satisfied.
\end{enumerate}

\subsection{Acknowledgements}
The author would like to thank Andrew Obus, his Ph.D. advisor, for introducing to him the lifting problem, which is the primary motivation of this research, and for the constant guidance throughout his early career. He thanks Frans Oort and David Harbater for helpful comments on an earlier version of this article. He thanks Takeshi Saito, Haoyu Hu, and Vaidehee Thatte for introducing many important articles that help him to better understand the refined Swan conductors. He thanks Nathan Kaplan, Joe Krammer-Miller, and the University of California Irvine's Department of Mathematics for their hospitality during his stay in Orange County, California, where most of the writing was done, amid the Covid-19 pandemic. He appreciates the valuable suggestions from the anonymous referees. Finally, the author is grateful to the University of Virginia, the Institute of Mathematics of Vietnam Academy of Science and Technology, the Vietnam Institute for Advanced Study in Mathematics, and the National Center for Theoretical Sciences (Taipei) for the excellent working conditions. This research is funded by the Simons Foundation Grant Targeted for the Institute of Mathematics, Vietnam Academy of Science and Technology.


\subsection{Notation and conventions}
The letter $K$ will always be a field of characteristic $p$ that is complete with respect to a discrete valuation $\nu: K^{\times} \rightarrow \mathbb{Q}$. The residue field $k$ of $K$ is algebraically closed of characteristic $p$. One example to keep in mind is $K=k((t))$, the field of Laurent series over $k$, and $\nu(t)=1$ defines the discrete valuation. The ring of integers of $K$ will be denoted by $R$.

We fix an algebraically closure $\overline{K}$ of $K$, and whenever necessary, we will replace $K$ by a suitable finite extension within $\overline{K}$, without changing the above notation. The symbol $\mathbb{K}$ usually denotes a function field over $K$ (e.g., $\mathbb{K}=K(X)$).

Below are some unusual notation used in this paper and their locations.
\begin{longtable}{ |p{1.6cm}|p{10.3cm}|p{2.8cm}| }
\hline
Notation & Description & Location \\
 \hline
 \hline
$\mathfrak{K}_n(\underline{a})$ & The character defined by the Witt-vector $\underline{a}$ of length $n$ & \S \ref{seccharacter} \\
 \hline
$\mathcal{D}[r,z]$ & Closed disc of radius $p^{-r}$, centered at $z$ & \S \ref{secdiscannuli} \\
 \hline
$D[r,z]$ & Open disc of radius $p^{-r}$, centered at $z$ & \S \ref{secdiscannuli} \\
 \hline
$W_n(\mathbb{K})$ & Witt-vector of length $n$ over $\mathbb{K}$ & \S \ref{secASW} \\
\hline
$(\--)_r$ & $\-- \cdot \pi^{-pr}$, where $\-- \in \mathbb{K}$, and $\pi$ is a uniformizer of a DVR & \S \ref{secdiscannuli}  \\
\hline
$[\hspace{1mm} \--\hspace{1mm}]_{r,z}$ & Reduction of $((\--) - z)_r$ & \S \ref{secdiscannuli} \\
\hline
$\mathcal{G}_v$ & A collection of extensions with right depth at $v$  & Definiton \ref{defnrightbranchdepth} \\
\hline
$\mathcal{H}_v$ & A collection of extensions with right branching datum at $v$ & Definition \ref{defnrightbranchdata} \\
\hline
$\mathcal{W}_v$ & A collection of extensions giving rise to right differentials at $v$ & Definition \ref{defextraexact} \\
\hline
$\lambda_e(G)$ & The kink of the character that $G$ gives rise to on an edge $e$ & \S \ref{seccontroledge} \\
\hline
$\lambda_m(\chi)$ & The largest kink of $\chi$ that verify the conditions in \S \ref{secdetect} & Proposition \ref{propdetect} \\
\hline
$\mu_m(\chi)$ & A function that detects the kink of $\chi$ of \S \ref{secdetect}  & Proposition \ref{propdetect} \\
\hline
$\delta_{\chi}(r,z)$ & The depth conductor of $\chi\lvert_{D[pr,z]}$ & \S \ref{secboundaryswanrateofchange}\\
\hline
$\omega_{\chi}(r,z)$ & The differential conductor of $\chi\lvert_{D[pr,z]}$ & \S \ref{secboundaryswanrateofchange} \\
\hline
$\sw_{\chi}(\overline{x})$ & The boundary conductor at the boundary associated to $\overline{x}$ of $\chi$ & \S \ref{secSwan} \\
\hline
$\mathfrak{C}_{\chi}(r,z, \overline{w})$ & Sum of the conductors on the $\overline{w}$ direction in $\mathcal{D}[pr,z]$ & Definition \ref{defconductoratplace} \\
\hline
$\mathfrak{C}_{\mathcal{T}}(e)$ & Sum of conductors of the leaves succeeding an edge $e$ of $\mathcal{T}$ & Definition \ref{defsumcondsprec} \\
\hline
$\nu_{s,z}$ & The Gauss valuation associated to $\mathcal{D}[s,z]$ & \S \ref{secdiscannuli} \\
\hline
$\mathbb{B}(\chi)$ & The branch locus of the character $\chi$ & \S \ref{secdiscannuli} \\
\hline
$\mathbb{B}_{\mathcal{T}}(e)$ & The leaves of $\mathcal{T}$ succeeding an edge $e$  & Definition \ref{defconductoratplace} \\
\hline
$\mathcal{T}(e)$ & The sub-tree of $\mathcal{T}$ that contains only the data after $e$  & \S \ref{secsubtree} \\
\hline
\end{longtable}
\smallskip
\section{Artin-Schreier-Witt covers}
\label{secASW}
Recall that, the prime-to-$p$ part of the fundamental group of a curve in characteristic $p$ is equal to that of its lift to characteristic $0$. However, the $p$-part of $\pi_1^{\et}(\mathbb{A}^1_k)$ is no longer trivial as there always exists an {\'e}tale  $\mathbb{Z}/p$-cover defined by the equation $y^p-y=x$. That cover, also known as an Artin-Schreier cover, is the simplest example of a \textit{wildly ramified} Galois cover ($p$ divides the order of an inertia group). Furthermore, $\mathbb{Z}/p^n$-covers of a projective line can be described using Witt vectors \cite[p. 330]{MR1878556}, hence it is easy to construct examples and study them using explicit methods. Therefore, understanding these covers is usually the first step in developing a theory for all wildly ramified covers.

\subsection{Artin-Schreier-Witt theory}
\label{secASWtheory}
In this section, we give a quick overview of Artin-Schreier-Witt theory. For more details, see, e.g., \cite[\S 26]{MR2371763} or  \cite[\RN{4}]{MR1878556}. Suppose $M/K$ is a $\mathbb{Z}/p^n$-extension of fields in characteristic $p>0$. Then $M=K(\alpha^1, \ldots, \alpha^{n})$, where $\alpha^i \in K^{\sep}$ is a solution of an \textit{Artin-Schreier-Witt equation}
\begin{equation}
    \label{eqnasw}
    \wp(\alpha^1, \ldots, \alpha^{n})=(f^1, \ldots, f^{n}),
\end{equation}
\noindent where $(f^1, \ldots, f^{n})$ lies in the ring $W_n(K)$ of truncated Witt vector of length $n$, and $\wp(\alpha:=(\alpha^1, \ldots, \alpha^{n}))=F(\alpha)-\alpha$ is the Artin-Schreier-Witt isogeny where $F$ is the Frobenius morphism of $W_n(K)$. Moreover, the extension is unique up to adding an element of the form $\wp(b^1, \ldots, b^{n})$ where $(b^1, \ldots, b^n) \in  W_n(K)$ to $(f^1, \ldots, f^{n})$ \cite[\S 26 Theorem 5]{MR2371763}. In other words, we have the following bijection, which is an application of Hilbert 90
\begin{equation*}
    \textrm{H}^1(K, \mathbb{Z}/p^n\mathbb{Z}) \xrightarrow{\simeq} W_n(K)/\wp(W_n(K)).
\end{equation*}
We call $\underline{f}:=(f^1, \ldots, f^n)$ the \textit{defining Witt vector} of the extension $M/K$ and the process of adding $\wp(b^1,\ldots,b^n)$ to the right-hand-side of (\ref{eqnasw})
an \textit{Artin-Schreier-Witt operation}. If $\underline{f}$ and $\underline{g}$ are different by an Artin-Schreier-Witt operation, we say they are in the same \textit{Artin-Schreier-Witt (ASW) class}.

\subsection{Cyclic covers of curves}
\label{seccycliccover}
An \textit{Artin-Schreier-Witt curve of level $n$} is a smooth, projective, connected curve that is a $G \cong \mathbb{Z}/p^n$-cover of the projective line over $k$. Recall that the category of normal projective $k$-curves and nonconstant morphisms is equivalent to that of finitely generated field extensions $K/k$ of transcendence degree $1$ \cite[Chapter I, Corollary 6.12]{MR0463157}. Therefore, it follows from the previous section that an arbitrary $\mathbb{Z}/p^n$-cover $\phi_n: Y_n \xrightarrow{} \mathbb{P}_k^1=\Proj k[x,v]$ can be represented by an ASW equation as follows       \begin{equation*}
            \label{ASWeqn}
            \wp(y^1, \ldots, y^n)=\big(f^1(x), \ldots, f^n(x)\big),
        \end{equation*}
\noindent        where $\underline{f}:=(f^1, \ldots, f^n)$ lies in the ring $W_n(k(x))$. For the rest of section \S \ref{seccycliccover}, we set $K:=k(x)$.

\begin{example}
\label{exAS}
Suppose $p=5$. The cover of $Y$ of $\mathbb{P}^1_K$ defined by
\begin{equation}
\label{eqnASex}
    y^5-y=\frac{1}{x^5}+\frac{1}{(x-1)^7},
\end{equation}
\noindent is an Artin-Schreier curve. Note that the term $1/x^5$ is a $5$th-power. Hence, by the above discussion, adding $(-1/x)^5-(-1/x)$ to the right-hand-side of (\ref{eqnASex}) does not change the cover it defines. The result is an Artin-Schreier equation of the form
\begin{equation}
\label{eqnASexreduced}
    y^5-y=-\frac{1}{x}+\frac{1}{(x-1)^7}.
\end{equation}
\end{example}

\begin{example}
\label{exasw}
Suppose $p=2$. The following equations defines a $\mathbb{Z}/4$-cover of $\mathbb{P}^1_k$.
\begin{equation}
    \label{eqnexasw}
     \wp(y^1, y^2)=\bigg(\frac{1}{x^2}, \frac{1}{x^4} +x^2 \bigg).
\end{equation}
By adding $\wp\big(\frac{1}{x}, x+\frac{1}{x} \big)$ to right-hand-side of (\ref{eqnexasw}), one obtains an alternative representation
\begin{equation}
    \label{eqnexaswreduced}
    \wp(z^1, z^2)=\bigg(\frac{1}{x}, \frac{1}{x}+x \bigg).
\end{equation}
\end{example}

\begin{remark}
We say the ASW equations (\ref{eqnASexreduced}) and (\ref{eqnexaswreduced}) have \textit{reduced forms}, or the defining Witt vectors are reduced. That means the partial fraction decomposition of each entry of the defining Witt vector only consist of prime-to-$p$-degree terms. One can show that every ASW cover can be represented by a unique Witt vector of reduced form.
\end{remark}

\subsubsection{Branching datum}
\label{secbranchingdatum}
A reduced defining Witt vector tells us everything about the ramification data of the cover it defines. Particularly, suppose $\underline{f} \in W_n(K)$ (recall that $K=k(x)$) is \emph{reduced}, and $\mathscr{P}:= \{P_1, \ldots, P_r\}$ is the set of poles of the $f^i$'s. Then $\mathscr{P}$ is also the branch locus of $\phi_n$. Furthermore, for each ramified point $Q_j$ above $P_j$, $\phi_n$ induces an exponent $p$ cyclic extension of complete local ring $\hat{\mathcal{O}}_{Y_n,Q_j}/\hat{\mathcal{O}}_{\mathbb{P}^1,P_j}$ with perfect residue field. Hence, it makes sense to talk about the ramification filtration of $\phi_n$ at a branch point $P_j$ (see \cite[Chapter IV, \S 3]{MR554237}). Suppose the inertia group of $Q_j$ is $\mathbb{Z}/p^m$ (where $n \ge m$). We say the \textit{$i$-th ramification break} of $\phi_n$ at $P_j$ is $-1$ for $i \le n-m$. When $i >n-m$, the $i$-th ramification break of $\phi_n$ at $P_j$ is the $(i-n+m)$-th one of $\hat{\mathcal{O}}_{Y_n,Q_j}/\hat{\mathcal{O}}_{\mathbb{P}^1,P_j}$. We denote by $u_{j,i}$ the $i$-th upper ramification break of $\phi_n$ at $P_j$. We call $e_{j,i}=u_{j,i}+1$ the \textit{$i$-th conductor of $\phi$ at $P_j$}. The following formula explicitly computes the ramification filtration of $\phi_n$ in terms of $\underline{f}$.

\begin{theorem}[{\cite[Theorem 1]{MR1935414}}]
\label{theoremcaljumpirred}
With the assumptions and the notation as above, we have
\begin{equation}
\label{eqnformulalowerjumpasw}
    u_{j,i}=\max\{ p^{i-l} \deg_{(x-P_j)^{-1}} (f^{l}) \mid l=1, \ldots, i\}, 
\end{equation}
for $i>n-m$.
\end{theorem}

\begin{remark}
\label{remarkupperbreaks}
From equation (\ref{eqnformulalowerjumpasw}), one deduces that, if the inertia group of $Q_j$ is $\mathbb{Z}/p^i$, then $i=\min \{ l \mid  \deg_{(x-P_j)^{-1}} (f^l) \neq 0  \}$. In addition, we have $p \nmid u_{j,n-m+1} \neq 0$, that $u_{j,i} \ge pu_{j,i-1}$ for $n-m+2 \le i \le n$, and that if $p \mid u_{j,i}$, then $u_{j,i}=pu_{j,i-1}$. In particular, when $n=1$, i.e., when $\phi_1$ is an Artin-Schreier cover, the unique ramification break at $P_i$ is equal to the order of the pole of $f^1$ at $P_i$.
\end{remark}

\begin{example}
The cover from Example \ref{exasw} has breaks $(1, 2)$ at $x=0$ and $(-1, 1)$ at $x=\infty$.
\end{example}

\begin{remark}
One easily derives from the above discussion that the length-$n$ Witt vector $(x, 0, \ldots, 0)$ defines an {\'e}tale $\mathbb{Z}/p^n$-cover of the affine line over $k$. That proves the infinitely-generatedness of $\mathbb{A}^1_k$'s fundamental group discussed in \S \ref{secintroduction}.
\end{remark}

Furthermore, it follows from \cite[Fact 2.3]{MR3194816} that the degree of the different at the branch point $P_j$ is
\begin{equation}
\label{eqndifferent}
\delta_{P_j}=\sum_{i=1}^n (u_{j,i}+1)(p^i-p^{i-1}).
\end{equation}
Set $e_{j,i}:=u_{j,i}+1$ . We denote by $\phi_l$ the $\mathbb{Z}/p^l$-subcover of $\phi_n$. Then $\phi_l: Y_l \xrightarrow{} \mathbb{P}^1_k$ corresponds to the length $l$ Witt vector $(f^1, \ldots, f^l)\in W_l(K)$. We denote by $\delta_{P_j}^l$ the degree of the different at $P_j$ of $\phi_l$. The following result gives the genus of $Y_l$ in terms of the ramification breaks.

\begin{proposition}
\label{propgenusleveli}
In the above notation, the genus of the $\mathbb{Z}/p^l$-covering-curve $Y_l$ is
\begin{equation}
    \label{eqngenuslvi}
    g_{Y_l}=\frac{2(1-p^l)+ \sum_{i=1}^l(\sum_{j=1}^r e_{j,i})(p^i-p^{i-1})}{2}.
\end{equation}
\end{proposition}

\begin{proof}
Apply the Riemann-Hurwitz formula \cite[IV, Corollary 2.4]{MR0463157}, with the degree of the different at each branch point $P_j$ is computed in (\ref{eqndifferent}), we obtain
\[ g_{Y_l}=\frac{2(1-p^l) + \sum_{j=1}^r \delta^l_{P_j}}{2}. \]
Using (\ref{eqndifferent}) to calculate the different, we immediately realize (\ref{eqngenuslvi}).
\end{proof}

Therefore,  $\mathbb{Z}/p^n$-covers of the same genus on each of its sub-cover have the same $e_i:=\sum_{j=1}^r e_{j,i}$ for $1 \le i \le n$. We thus use an $r \times n$ matrix as follows to record the ramification data of the cover
\[M= \begin{bmatrix}
   e_{1,1} & e_{1,2} & \ldots & e_{1,n} \\
   e_{2,1} & e_{2,2} & \ldots & e_{n,n} \\
   \vdots & \vdots & \ddots & \vdots \\
   e_{r,1} & e_{r,2} & \ldots & e_{r,n} \\
   \end{bmatrix}.\]
We call the above matrix $M$ the \textit{branching datum} of $\phi_n$. We call the divisor $\sum_{j=1}^r e_{j,l} P_j$ the \textit{level $l$-th branching divisor} or, when $l=n$, just the \textit{branching divisor} of $\phi_n$.

The forward direction of the following corollary is immediate from Remark \ref{remarkupperbreaks}.

\begin{corollary}
\label{corconductorconditions}
An $r \times n$ matrix $M=(e_{j,i})$ with positive integer entries is a branching datum of a $\mathbb{Z}/p^n$-cover if and only if the followings hold
\begin{enumerate}
    \item \label{corconductorconditions1} For $J=\min \{{i} \mid e_{j,i} \neq 0\}$,  $e_{j,J} \not\equiv 1 \pmod{p}$ for $j=1, \ldots, r$.
    \item \label{corconductorconditions2} For $J<i\le n$, we have  $e_{j,i} \ge pe_{j,i-1}-p+1$. The ``$=$'' holds if and only if $e_{j,i} \equiv 1 \pmod{p}$. If $e_{j,i}>pe_{i,j-1}-p+1$, then $e_{j,i}=pe_{i,j-1}-p+a_{j,i}+1$ for an integer $a_{j,i}$ prime to $p$.
\end{enumerate}
\end{corollary}

\begin{proof}
Suppose $M=(e_{j,i})$ is an $r \times n$ matrix whoses entries verifying (\ref{corconductorconditions1}) and (\ref{corconductorconditions2}). Let $x_i, \ldots, x_r$ be distinct points on $k$. Consider the length-$n$-Witt-vector
\[ F=(f^1, \ldots, f^n) \in W_n(k(x)), \]
where $f^i=\sum_{j=1}^r \frac{a_{j,i}}{(x-x_j)^{e_{j,i}-1}}$ and the $a_{j,i}$'s are defined as follows
\[a_{j,i}:= \begin{cases} 
      0  & e_{j,i} \equiv 1 \pmod{p} \\
      1 & e_{j,i} \not\equiv 1 \pmod{p}  \\
   \end{cases}.
\]
It then follows from Theorem \ref{theoremcaljumpirred} that the cyclic cover defined by $F$ has branching datum coinciding with $M$, proving the necessity.
\end{proof}

\begin{definition}
\label{defnconductormatrices}
Suppose $(e_1, \ldots, e_n)$, as a $1 \times n$ matrix, verifies Corollary \ref{corconductorconditions}. We define $\Omega_{e_1, \ldots, e_n}$ to be the collection of $r \times n$ matrices that partition $(e_1, \ldots, e_n)$ and such that each of its rows satisfies the conditions asserted by the same corollary.
\end{definition}

\begin{definition}
\label{defessentialjump}
With the notation above and set $e_{0,i}=0$ for $1 \le j \le r$, we say $\phi_n$ has no \textit{essential jump} from level $i-1$ to $i$ at $P_j$ if $e_{j,i}=pe_{j,i-1}-p+1$ or $e_{j,i}=pe_{j,i-1}-p+a_{j,i}+1$ where $1 \le a_{j,i} \le p-1$. If that holds for all $P_j$'s, we say $\phi_n$ has no essential jumps from level $i-1$ to $i$. We say $\phi_n$ has no essential , jump if that is true for all $1\le i \le n$.
\end{definition}

\begin{example}
The cover in Example \ref{exAS} has branching datum $[2,8]^{\top}$, hence has as essential jump at $x=1$. One from Example \ref{exasw} has branching datum  \[ \begin{bmatrix}
   2 & 3  \\
   0 & 2  \\
\end{bmatrix},\]
thus has no essential jump.
\end{example}

\begin{remark}
\label{remarksumofconductors}
When $n=1$, the genus of an Artin-Schreier cover $Y_1 \xrightarrow{\phi} \mathbb{P}^1_k$ is
\[ g:=g_{Y_1}= \frac{(\sum_{i=1}^r (e_{i,1}+1)-2)(p-1)}{2}.  \] Hence, all the Artin-Schreier $k$-curves with the same genus $g$ have the same sum of conductors $d+2$, where $d:=2g/(p-1)$. Thus, it is natural to classify Artin-Schreier covers of the same genus by their branching data. This idea is utilized by Pries and Zhu to partition the moduli space $\mathcal{AS}_g$ of Artin-Schreier curves of genus $g$ into irreducible strata \cite[Theorem 1.1]{MR2985514}. By explicitly constructing some local deformations, we show that $\mathcal{AS}_g$ is connected when $g$ is sufficiently large \cite[Theorem 1.1]{DANG2020398}. In a manuscript in preparation, we carry the idea from \cite{MR2985514} and \cite{DANG2020398} further to all cyclic covers. In particular, we show that the moduli space $\mathcal{ASW}\text{cov}_{(g^1, \ldots, g^n)}$ of $\mathbb{Z}/p^n$-covers whose $\mathbb{Z}/p^i$-sub-cover has genus $g^i$ can be partitioned into irreducible strata, where each stratum is represented by a suitable matrix of branching datum.
\end{remark}

\subsection{Deformations of cyclic covers}

\subsubsection{Deformation of Galois covers}
Suppose $Y \xrightarrow{\phi} C$ a $G$-Galois cover over $k$, where $C$ and $Y$ are smooth, projective, connected $k$-curves. Suppose, moreover, that $A$ is a local, Noetherian, complete $k$-algebra with residue field $k$, which is also a domain. Let 
\[ \Def_{\phi} : \Alg/k \xrightarrow{} \text{Set} \]
\noindent be the functor which, to any $A \in \Alg/k$, associates classes of $G$-Galois cover $\mathscr{Y}_A \xrightarrow{\Phi} \mathscr{C}_A$ of smooth proper curves that makes the following cartesian diagram commutes
\begin{equation}
\label{defndef}
     \begin{tikzcd}
Y \arrow{d}{\phi} \arrow{r}{}
& \mathscr{Y}_A \arrow{d}{\Phi} \\
C \arrow{r}[black]{} \arrow[black]{d}{} & \mathscr{C}_A \arrow[black]{d}{}\\
\spec k \arrow[black]{r}{f} & \spec A,
\end{tikzcd}
\end{equation}
\noindent and so that the $G$-action on $\mathscr{Y}_A$ induces the original action on $Y$. More precisely,
\begin{itemize}
    \item The special fiber of $\Phi$ is isomorphic to $\phi$, and
    \item The isomorphism $Y \cong \mathscr{Y}_A \otimes _A k$ is $G$-equivariant.
\end{itemize}
We call $\Phi$ a \textit{deformation} of $\phi$ over $A$. For more details, see \cite[\S 2]{MR1767273}. In this paper, we focus on the case where $G$ is cyclic and $A$ is a finite extension of a power series $k[[t]]$, which is a complete discrete valuation ring of characteristic $p$ with residue $k$.

\begin{remark}
One application of deformation is reducing hard, general problems to ones that are easier to study. For instance, Norman and Oort showed that every abelian variety in characteristic $p>0$ can deforms to an ordinary one \cite{MR595202}, hence lifts to characteristic zero by Serre-Tate \cite{MR638600}, \cite{MR638599}. A similar strategy was applied to prove that every cyclic cover in characteristic $p$ lifts to characteristic zero. Namely, Pop showed that every cyclic cover equal-characteristically deforms to one with no essential ramification jumps. This cover always lifts by Obus and Wewers \cite{MR3194815}, hence so is the original one. This approach was for the first time discussed in \cite{MR228505}.
\end{remark}

\subsubsection{A local-to-global principle}
\label{seclocalglobal}
Let $R$ be a complete discrete valuation ring as before. The \textit{local-to-global principle} below allows us to study the deformation problem by way of the local deformation problem.

\begin{theorem}
\label{thmlocalglobaldeformation}
Let $Y$ be a smooth projective curve over $k$, with a faithful action of $G$ by $k$-automorphisms. Let $y_1, \ldots, y_r \in Y$ be the points where $G$ acts with non-trivial inertia. For each $1 \le j \le r$, let $G_j$ be the inertia group of $y_j$ in $G$, and let $\iota_j: G_j \xrightarrow{} \aut_k k[[u_j]]$ be the induced local action on the complete local ring of $y_j$. Then the deformation of $Y$ with $G$-action over $R$ is determined by the deformation over $R$ of each of the local $G_j$-action.
\end{theorem}

\begin{proof}
See \cite[\S 1.2]{MR3051252} or \cite[\S 3]{MR1424559}.
\end{proof}

We thus can reduce the study of deformation to the case where $\phi$ is local, namely, where $\phi$ is a $G$-Galois extension of power series $k[[x]] \xrightarrow{} k[[z]]$. In that case, Bertin and M{\'e}zard showed that the functor $\Def_{\phi}$ is represented by (the spectrum of) a versal deformation ring \cite[Theorem 2.1]{MR1767273}, which we denote by $R_{\phi}$ or $R_G$. The \textit{tangent space} of $\Def_{\phi}$ is $\textrm{H}^1(G, \Der_k k[[z]])$. Its \textit{obstruction space} is $\textrm{H}^2(G, \Der_k k[[z]])$. Note that, when $G$ is cyclic, the dimension of $\textrm{H}^2(G, \Der_k k[[z]])$ is usually positive \cite[Th{\'e}or{\`e}m 4.1.1]{MR1767273}. The deformation problem is hence non-trivial.


Little is known about $R_G$ when $\phi$ is wildly ramified, i.e., when $p \mid \lvert G \rvert$. When $G \cong \mathbb{Z}/p$ and $\phi$ has ``conductor'' $1$ (that is when the corresponding HKG-cover \cite{MR579791} (see also \S \ref{secdeformonepointcover}) of the projective line over $k$ has genus $0$), $R_{G}$ is completely described by \cite[Th{\'e}or{\`e}me 4.2.8]{MR1767273}. It is also known for all conductors when $G=\mathbb{Z}/2$ \cite[Th{\'e}or{\`e}m 4.3.7]{MR1767273}. To generalize these results, one would naturally ask the following Galois theory type of question.

\begin{question}
\label{questionGalois} Suppose $n \ge m \in \mathbb{Z}_{\ge 0}$, and we are given a tower of Galois extensions $$\begin{tikzpicture}
\node at (0,0) (a) {$k[[x]]$};
\node[right =1cm of a]  (b){$k[[y_m]]$};
\node[right =1cm of b]  (c){$k[[y_n]].$};
\draw[->,>=stealth] (a) --node[above]{$\phi_m$} (b);
\draw[->,>=stealth] (b) --node[above]{$\phi_{n/m}$} (c);
\draw[->,>=stealth] (a) edge[bend right=12]node[below]{$\phi_n$} (c);
\end{tikzpicture}$$
How do the rings $R_{\phi_m}$, $R_{\phi_{n/m}}$, and $R_{\phi_n}$ relate?
\end{question}

There is a natural map $\spec R_{\phi_n} \xrightarrow{\ind} \spec R_{\phi_m}$ \cite{2011arXiv1112.0352B}. In particular, when $\phi_n$ is a $\mathbb{Z}/p^n$-cover defined by a length-$n$-Witt-vector $\underline{f}_n:=(f^1, \ldots, f^n)$ and $\phi_m$ is its $\mathbb{Z}/p^m$-sub-cover, the morphism $\ind$ maps $\underline{f}_n$ to $\underline{f}_m:=(f^1, \ldots, \allowbreak f^m)$. Let $\phi_m$ denotes the $m$-th \textit{level} of $\phi_n$. We conjecture the following.
\begin{conjecture}
\label{conjcyclicdeformationtowers}
In the notation of Question \ref{questionGalois}, if $G:=\gal(k[[y_n]]/k[[x]])$ is cyclic then $\ind: \spec R_{\phi_n} \xrightarrow{} \spec R_{\phi_m}$ is surjective. That is, given an arbitrary complete discrete valuation ring $R$ with residue field $k$, one can always fill in the following commutative diagram
\begin{equation*}
\begin{tikzcd}
\spec k \arrow[d] & \spec k[[x]] \arrow[d] \arrow[l] \arrow[d] & \spec k[[y_m]] \arrow[l, "\phi_m"] \arrow[d] & \spec k[[y_n]] \arrow[l, "\phi_{n/m}"] \arrow[d, dotted] \\
\spec R'   & \spec R'[[X]] \arrow[l] & \spec R'[[Y_m]] \arrow[l, "\Phi_m"] & \spec R'[[Y_n]] \arrow[l, dotted, "\Phi_{n/m}"],
\end{tikzcd}
\end{equation*} 
after a finite extension $R'$ of $R$. If that holds, we say $\phi_n$ is \textit{deformable in towers}.
\end{conjecture}

\begin{remark}
When $R$ has characteristic zero, the conjecture is precisely the refined local lifting problem (Question \ref{questionrefinedlocallifting}).
\end{remark}

\begin{remark}
One may want to approach Question \ref{questionGalois} on the level of tangent space, i.e., investigating the relations between the groups $\textrm{H}^1(\gal(k[[y_n]]/ \allowbreak k[[x]]), \Der_k k[[y_n]])$ and $\textrm{H}^1(\gal(k[[y_m]]/ \allowbreak k[[x]]), \Der_k k[[y_m]])$. So far, by generalizing some results from \cite{MR1767273}, we have computed the dimensions of $R_{\mathbb{Z}/p^n}$, generalizing \cite[Theorem 5.3.3]{MR1767273}. 
\end{remark}

Note that a cyclic extension in characteristic zero is described by Kummer theory, which is ``multiplicative'' compared with the ``additive'' nature of ASW theory. It is hence natural to first study the case when $R$ has equal-characteristic, as both the generic and the special fiber are additive in that situation. That is exactly the local version of Theorem \ref{thmdeformtowers}! 

\begin{theorem}[(Local deformation of Artin-Schreier-Witt covers)]
\label{theoremmainlocal}
Suppose $H$ is a (nontrivial) quotient of a finite cyclic group $G$ and $\phi_n$ is a $G$-extension $k[[y_n]]/k[[x]]$, hence branches at exactly one point. Suppose, moreover, that $\Phi_{m}$ is a $H$-extension $R[[Y_m]]/R[[X]]$ that deforms the unique sub-$H$-cover $\phi_m: k[[x]] \xrightarrow{} k[[y_m]]$ of $\phi_n$ over $R:=k[[t]]$. Then there exist a finite extension $R'/R$ and a deformation $R'[[Y_n]]/R'[[X]]$ of $\phi_n$ over $R'$ that contains $R'[[Y_{n-1}]]/R'[[X]]$ as a sub-cover.
\end{theorem}

In the following proposition, we show that the local result implies the global one. The proof of ``$\Rightarrow$'' direction is postponed to \S \ref{secdeformonepointcover}.

\begin{proposition}
\label{proplocalglobalintowers}
Theorem \ref{thmdeformtowers} holds if and only if Theoreom \ref{theoremmainlocal} does.
\end{proposition}

\begin{proof}[Sketch of the ``$\Leftarrow$'' direction's proof]
This proof is developed from one for \cite[Theorem 3.1]{MR3051249}. Let us assume that each of the local covers deforms in towers. With the assumptions of Theorem \ref{thmdeformtowers}, let $\mathfrak{X}$ (resp., $\mathfrak{Y}$) be the formal completion of $\mathcal{X}_R$ (resp., $\mathcal{Y}_R$) at $X$ (resp., at $Y$). Let $B \subsetneq X$ be the branch locus of $\psi$, set $U:= X \setminus B$, $V:=\psi^{-1} (U)$, and $W:=\phi^{-1} (U)= Y \setminus \{ y_1, \ldots, y_s \}$. Let $\mathfrak{U} \subseteq \mathfrak{X}$ be the formal subscheme associated to $U \subseteq X$. By Grothendieck's theory of {\'e}tale lifting \cite[IX, 1.10]{MR0217087}, the $G$-cover  $\phi\lvert_{W}: W \rightarrow U$ deforms over $R$ to an {\'e}tale $G$-cover of formal scheme $\mathfrak{W} \rightarrow \mathfrak{U}$ with the deformation $\mathfrak{V} \rightarrow \mathfrak{U}$ of $\psi\lvert_V: V \rightarrow U$ as the $H$-sub-cover. The boundary of $\mathfrak{W}$ is isomorphic to disjoint union $\bigsqcup_{j=1}^s \mathcal{B}_j$, where each $\mathcal{B}_j$, which corresponds to the point $y_j$, is isomorphic to the boundary of a disc (see \S \ref{secdiscannuli}). For each $j$, there exists a canonical action of the inertia group $G_j \le G$ of $y_j$ on $\mathcal{B}_j$.

By assumption, each local $G_j$-extension $\hat{\mathcal{O}}_{Z, y_j} / \hat{\mathcal{O}}_{X, \phi(y_j)}$ deforms over $R$ in towers to a $G_j$-cover of an open disc $D_j \cong \spec R[[W_j]] \rightarrow \spec R[[V_j]] \rightarrow \spec R[[U_j]]$. The action of $G_j$ on $D_j$ induces an action on its boundary $\partial D_j$. In addition, the theory of {\'e}tale lifting asserts that the $G_j$-action on $\partial D_j$ is isomorphic to the action on $\mathcal{B}_j$. Thus, by identifying $\mathcal{B}_j$ and $\partial D_j$, we can use formal patching to ``glue'' each of these discs $D_j$ to $\mathfrak{W}$ in a $G_j$-equivariant way. This yields a formal curve with $G$-action and projective special fiber. By Grothendieck's Existence Theorem \cite[\S 5]{MR217085}, this formal curve is the projective completion of a smooth projective curve $\mathcal{Z}_R$ with $G$-action such that $\mathcal{Z}_R/G \cong \mathcal{X}_R$ and $\mathcal{Z}_R/H \cong \mathcal{Y}_R$. See \cite{2000math.....11098H} or \cite{2020arXiv200203719D} for concrete examples of formal patching. This is the cover we are looking for.
\end{proof}

On the level of tangent space, the following result follows immediately from Theorem \ref{theoremmainlocal}.

\begin{corollary}
In the notation of Question \ref{questionGalois}, when $\phi_n$ is cyclic, the following group homomorphism is surjective
\[ \ind: \cohom^1(\gal k[[y_n]]/k[[x]], \Der_k k[[y_n]]) \xrightarrow{} \cohom^1(\gal k[[y_m]]/k[[x]], \Der_k k[[y_m]]). \]
\end{corollary}

We first explore the case where $G$ is a cyclic $p$-group, i.e., $G \cong \mathbb{Z}/p^n$ for some $n>1$.

 \subsubsection{Deformation of Artin-Schreier-Witt covers}
 
In the notation of Theorem \ref{theoremmainlocal}, suppose that $G=\mathbb{Z}/p^n$. We assume that $R$ is a complete discrete valuation ring of equal-characteristic. Thanks to the local-to-global principle (Theorem \ref{thmlocalglobaldeformation}), one might assume that $\phi_n$ branches at exactly one point while studying its deformations. Thus, a $\Phi_n$ deformation of $\phi_n$ over $R$ can be described by the following expression (see \S \ref{secbranchingdatum})
\begin{equation}
    \label{eqndeformationtype}
[e_1, e_2, \ldots, e_n] \xrightarrow{}   \begin{bmatrix}
   e_{1,1} & e_{1,2} & \ldots & e_{1,n} \\
   e_{2,1} & e_{2,2} & \ldots & e_{2,n} \\
   \vdots & \vdots & \ddots & \vdots \\
   e_{r,1} & e_{r,2} & \ldots & e_{r,n} \\
   \end{bmatrix}
\end{equation}
\noindent We call the expression in (\ref{eqndeformationtype}) the \textit{type} of the deformation $\Phi_n$. The second matrix indicates that $\Phi_n$'s generic fiber branches at $r$ points $Q_1, \ldots, Q_r$, which correspond to rows $1, \ldots, r$, with upper ramification jumps $(e_{i,1}-1, e_{i,2}-1, \ldots, e_{i,n}-1)$ at $Q_i$. 

 \begin{proposition}
Suppose we are given a deformation of a $\mathbb{Z}/p^n$-cover over $R$ of the following type (\ref{eqndeformationtype}). Then $\sum_{j=1}^r e_{j,i}=e_i$ for $i=1,2, \ldots, n$.
\end{proposition}

\begin{proof}
A deformation of a $\mathbb{Z}/p^n$-cover induces for each $1 \le i <n$ a $\mathbb{Z}/p^i$-deformation whose fibers have the same genus. The rest follows immediately from Proposition \ref{propgenusleveli}.
\end{proof}

In the following section, we present a well-known tool to realize a deformation and some explicit examples.

\subsubsection{Birational deformations}
Usually, when dealing with $k[[x]]$, we will often want to think of its Galois extension in terms of the associated extension of fraction field. 

\begin{definition}
\label{defbirationaldeformatiton}
Suppose $A/k[[x]]$ is a $G$-extension. Suppose, moreover, that $M / \Frac (R[[X]])$ where $R/k[[t]]$ finite, is a $G$-extension, and $A_R$ is the integral closure of $R[[X]]$ in $M$. We say $M / \Frac (R[[X]])$ is a \textit{birational deformation} of $A/k[[x]]$ if
\begin{enumerate}
    \item The integral closure of $A_R \otimes_R k$ is isomorphic to $A$, and
    \item The $G$-action on $\Frac (A)= \Frac (A_R \otimes_R k)$ induced from that on $A_R$ restricts to the given $G$-action on $A$.
\end{enumerate}
\end{definition}

Adapting the strategy from \cite{MR1424559}, which proves that any $G$-local-extension lifts, one can show that any $G$-local-extension can also be birationally deformed in towers! The following criterion is extremely useful for seeing when a birational deformation is actually a deformation (i.e., when $A_R \otimes_R k$ is already integrally closed, thus isomorphic to $A$).

\begin{proposition}[(The different criterion) {\cite[I,3.4]{MR1645000}}]
\label{propdifferentcriterion} Suppose $A_R/R[[X]]$ is a birational deformation of the $G$-Galois extension $A/k[[x]]$. Let $K:=\Frac R$, let $\delta_{\eta}$ be the degree of the different of $(A_R \otimes_R K)/(R[[X]] \otimes_R K)$, and let $\delta_s$ be the degree of the different of $A/k[[x]]$. Then $\delta_s \le \delta_{\eta}$, and equality holds if and only if $A_R/R[[X]]$ is a deformation of $A/k[[x]]$.
\end{proposition}

\begin{example}
\label{exampleorder5threebranches}
One may use the above criterion to check that the $\mathbb{Z}/5$-extension $\Phi$ given by
\[ Y^5-Y=\frac{X+2t^{10}}{X^5(X-t^{10})^2(X-t^5)^5} \]
\noindent is a deformation of $y^5-y=\frac{1}{x^{11}}$ over $k[[t]]$. An easy computation shows that the generic branch points of $\Phi$ are $0, t^{10}$, and $t^{5}$, which have conductors $4, 3$, and $5$, respectively. Hence, it is a deformation of type $[12] \xrightarrow{} [4,3,5]^{\top}$. See \cite{2020arXiv200203719D} and \cite{DANG2020398} for more explicit examples of $\mathbb{Z}/p$-deformations.
\end{example}

\begin{example}
\label{exmain}
Let us consider the $\mathbb{Z}/4$-extension $\chi_2$ of $k[[t,X]]$ (where $\characteristic k=2$) defined by
     \begin{equation}
     \label{order4deformation}
         \wp(Y_1,Y_2)=\bigg(\frac{1}{X^2(X-t^8)}, \frac{1}{X^3(X-t^8)^2(X-t^2)^4} \bigg).
     \end{equation}
Its fiber $\overline{\chi}_2$ at $t=0$ is birationally equivalent to the following $\mathbb{Z}/4$-extension of $k[[x]]$
\[ \wp(y_1,y_2)=\bigg( \frac{1}{x^3}, \frac{1}{x^9} \bigg). \]
\noindent By converting the Witt vector in (\ref{order4deformation}) to a reduced form as in Example \ref{exasw}, one sees that the generic fiber has upper jumps $(1,2)$ at $0$, $(1,2)$ at $t^8$, and $(-1,3)$ at $t^2$. Therefore, the deformation $\chi_2$ has type
\[ \begin{bmatrix}
4 & 10 \\
\end{bmatrix} \xrightarrow{}  \begin{bmatrix}
   2 & 2& 0  \\
   3 & 3 & 4 \\
\end{bmatrix}^\intercal. \] 
\noindent Note that the $\mathbb{Z}/2$-sub-extension of $\chi_2$ is a deformation of type $[4] \xrightarrow{} [2,2]^{\top}$ with generic branch points $0$ and $t^8$. 
\end{example}

\subsubsection{A family of non-trivial deformations}
\label{secessentialcover}

Suppose $\phi_n$ is a $\mathbb{Z}/p^n$-extension $k[[z_n]]/k[[x]]]$ with upper ramification jumps $\iota_1<\iota_2<\ldots<\iota_n$. Set $\iota_0=0$. Then, for each $1 \le j \le n$, we may write
\[ \iota_j -p\iota_{j-1}=pq_j+\epsilon_j, \]
with $0 \le q_j$, $0 \le \epsilon_j <p$, and $q_j=0$ if $\epsilon_j=0$, as asserted by Corollary \ref{corconductorconditions}. Thus $0< \epsilon_j$ if and only if $(p, \iota_j)=1$ if and only if $p \iota_{j-1}<\iota_j$. We called $q_j$ the \textit{essential part} of the upper jump at level $j$, and if $0<q_j$ we say that $\iota_j$ is an \textit{essential upper jump}. These terminologies were introduced in \cite{MR3194816}. Let $r:=1+\sum_{j \text{ essential}} q_j$. According to Pop, one may partition $(\iota_1+1, \ldots, \iota_{n}+1)$ into an $r \times n$ matrix $M:=(e_{j,i})$ as follows. 
\begin{enumerate}
    \item $e_{1,i}=pe_{1,i-1}+\epsilon_i$ for $1 \le i \le n$, $e_{1,0}=0$,
    \item Add $q_j$ number of rows of the following form to $M$ for each essential place $j$
    \[(0, \ldots, 0, p-1, p^2-1, \ldots, p^{n-j+1}-1). \]
\end{enumerate}
Furthermore, Pop shows that there always exists a deformation of type $M$ as above. We restate his result using the conventions in this paper as follows.

\begin{lemma}[{\cite[Key Lemma 3.2]{MR3194816}}]
Let $A=k[[x]] \xhookrightarrow{} k[[z]]:=B $ be a cyclic $\mathbb{Z}/p^n$-extension with upper ramification jumps $\iota_1, \ldots, \iota_n$. In the above notation, let $x_1, \ldots, x_r \in tk[[t]]$ be distinct elements. Then there exists a $\mathbb{Z}/p^n$-deformation of $A \xhookrightarrow{} B$ over $k[[t]]$ of types $M$ that has $x_1, \ldots, x_r$ as branch points.
\end{lemma}

\begin{example}
\label{examplepopdeformation}
Suppose $\overline{\chi}$ is a $\mathbb{Z}/5^2$-extension of $k[[x]]$ that is defined by
\[ \wp(y_1,y_2)=\bigg( \frac{1}{x^8}, \frac{1}{x^{52}}+\frac{1}{x^{46}} \bigg)=\bigg( \frac{1}{x^8}, \frac{1+x^6}{x^{52}}\bigg). \]
It branches at $x=0$ with jumps $(\iota_1, \iota_2)=(8, 52)$, thus branching datum $[9,53]$. Observe that
\begin{equation}
\begin{cases}
\iota_1-5\iota_0 & = 5 \cdot \textcolor{black}{1} + \textcolor{black}{3} \\
\iota_2-5\iota_1 & = 5 \cdot \textcolor{black}{2}+\textcolor{black}{2}
\end{cases}.
\end{equation}
Based on the above data one can ``split'' $[9,53]$ into the following $4 \times 2$ matrix
\begin{equation}
\label{eqnpopsplit}
\begin{bmatrix}
9 & 53 \\
\end{bmatrix} \xrightarrow{}  \begin{bmatrix}
   \textcolor{black}{3}+1 & \textcolor{black}{3}\cdot 5+\textcolor{black}{2}+1 \\
   \textcolor{black}{5} & \textcolor{black}{25} \\
   \textcolor{black}{0} & \textcolor{black}{5}\\
   \textcolor{black}{0} & \textcolor{black}{5}\\
   \end{bmatrix}=\begin{bmatrix}
   4 & 18 \\
   5 & 25 \\
   0 & 5\\
   0 & 5 \\
\end{bmatrix}.  
\end{equation}
One can show, using the strategy from Example \ref{exmain}, that the extension $\chi$ defined by
\begin{equation}
    \label{expopdeformationeqn}
    \wp(Y_1,Y_2)= \bigg(\frac{1}{x^3(x-t_1)^5}, \frac{1+x^6}{x^{17}(x-t_1)^{25}(x-t_2)^{5}(x-t_3)^{5}}  \bigg),
\end{equation}
where $t_1, t_2, t_3 \in tk[[t]]$ are distinct, is a deformation of $\overline{\chi}$ over $k[[t]]$ with type (\ref{eqnpopsplit}). Equation (\ref{expopdeformationeqn}) is derived from \cite[Key Lemma 3.2]{MR3194816}'s proof.
\end{example}

\begin{remark}
The splitting in (\ref{eqnpopsplit}) and the explicit equation (\ref{expopdeformationeqn}) easily generalize to arbitrary $\mathbb{Z}/p^n$-cover of $\mathbb{P}^1_k$. We call these type of deformations \textit{Oort-Sekiguchi-Suwa} (OSS) deformations. They are generalizations of (the $p$-fiber of) ones for Artin-Schreier covers introduced in \cite[\S 4.3]{MR1767273}. See \cite[\S 3.1.1]{DANG2020398} for a detail discussion regarding how the $p$-fibers of the deformations in \cite[\S 4.3]{MR1767273} are OSS deformations of Artin-Schreier covers. In addition, when the cover is of order $p$, Bertin and M{\'e}zard show that these deformations form a dominant component of the local deformation ring's spectrum. One thus would expect this also holds for the versal deformation rings of $\mathbb{Z}/p^n$-covers
when $n>1$.
\end{remark}

In the next section, we show that it suffices to answer Theorem \ref{theoremmainlocal} for the case $G \cong \mathbb{Z}/p^n$.

\subsubsection{Reduction to the case of cyclic \texorpdfstring{$p$}{p}-groups}
\label{sectionreductioncyclic}
We first state a well-known result, which suggests that the deformations of local-cyclic-tamely-ramified extensions are not very interesting.

\begin{proposition}
\label{propdeformtameintowers}
Tamely-ramified-cyclic-covers are deformable over $k[[t]]$ in towers. 
\end{proposition}

\begin{proof}
Suppose $\phi$ is a tamely-ramified cyclic covers of $k[[x]]$ with Galois group $\mathbb{Z}/m$, where $m$ is prime to $p$. It then follows from Kummer theory that, after a change of variables, $\phi$ is given by $z^m=x$. Suppose $m=nr$, where $n, r \neq 1$. Then the unique $\mathbb{Z}/n$-subcover $\tau$ of $\phi$ is defined by $y^n=x$. Furthermore, a deformation $\mathscr{T}$ of $\tau$ over $k[[t]]$, after a change of variables, can be defined generically by $Y^n=X-h(t)$, where $h(t) \in  t \cdot k[[t]]$ and $X$ is a lift of $x$ to $k[[t]]$. It is then easy to verify, using the different criterion, that $Y^m=X-h(t)$ defines the deformation of $\phi$ that we want. 
\end{proof}

\begin{proposition}
\label{propreducecyclicpgroup}
Let $G \cong \mathbb{Z}/mp^n$, where $p \nmid m$. If $\mathbb{Z}/p^n$ is deformable in towers, then so is $G$. In particular, it suffices to prove Theorem \ref{theoremmainlocal} for the case $G \cong \mathbb{Z}/p^n$.
\end{proposition}

\begin{proof}
Suppose $H$ is a subgroup of $G$. Then $H$ is isomorphic to $\mathbb{Z}/lp^r$, where $l \vert m$ and $0 \le r \le n$. Given a $G$-cover $f: Y \xrightarrow{} \spec k[[x]]$, and let $g: Z \xrightarrow{} \spec k[[x]]$ (resp. $h: X \xrightarrow{} \spec k[[x]]$) be the unique sub-$\mathbb{Z}/m$-cover (resp. sub-$\mathbb{Z}/p^n$-cover). Then the normalization $Z \times_{\spec k[[x]]} X$ is isomorphic to $Y$. Similarly, the $H$-sub-cover of $f$ can be identified with $f': Y'\cong Z' \times_{\spec k[[x]]} X'  \xrightarrow{} \spec k[[x]]$ where $g': Z' \xrightarrow{} \spec k[[x]]$ (resp. $h': X' \xrightarrow{} \spec  k[[x]]$) is the sub-$\mathbb{Z}/l$-cover over $g$ (resp. the sub-$\mathbb{Z}/p^r$-cover over $h$). Suppose $F'_R: \mathcal{Y}' \xrightarrow{} \spec R[[X]]$ is a deformation over $f'$ over $R$. Then $\mathcal{Y}' \cong \mathcal{X}' \times_{\spec R[[X]]} \mathcal{Z}'$ where $G'_R: \mathcal{Z}' \xrightarrow{} \spec R[[X]]$ deforms $g'$ and $H'_R: \mathcal{X}' \xrightarrow{} \spec  R[[X]]$ deforms $h'$. Moreover, the unique branch point of the generic fiber $G'_K$ of $G'_R$ has index $p^l$ in the generic fiber $H'_K$ of $H'_R$. 
By assumption, $H'_R$ extends to $H_R: \mathcal{X} \xrightarrow{} \spec R[[X]]$. Furthermore, by Proposition \ref{propdeformtameintowers}, one can extend $G'_R$ to $G_R: \mathcal{Z} \xrightarrow{} \spec R[[X]]$.

Finally, let $\mathcal{Y}''_R$ be the normalization of $\mathcal{X}_R \times_{\mathcal{D}} \mathcal{Z}_R$. Then the canonical map $F_R: \mathcal{Y}''_R \xrightarrow{} \mathcal{D}$ is a birational deformation of $f$. The degree of the different of $g$ (and of $G_K$) is $m-1$. Let $\delta$ be the degree of the different of $h$ (and of $H_K$). Using our assumption on the branch loci of $F'_K$ and $G'_K$, one shows that the degree of the different of $F_R$ and $f$'s generic fibers are both $m \delta +m-1$. It thus follows from Proposition \ref{propdifferentcriterion} that $F_R$ is a deformation of $f$ over $R$. 
\end{proof}

Let $L_n=k[[y_n]]/k[[x]]$ be a $\mathbb{Z}/p^n$-extension. From the above discussion, Theorem \ref{theoremmainlocal} is an immediate result of the following.

\begin{corollary}
\label{cordeformcyclicintowers}
Suppose $k[[y_n]]/k[[x]]$ is cyclic Galois of order $p^n$, and $R[[Y_{m}]]/ \allowbreak R[[X]]$ is a deformation of the $\mathbb{Z}/p^{m}$-subextension $k[[y_{m}]]/k[[x]]$ over a finite extension $R$ of $k[[t]]$. Then, there exists a finite extension $R'/R$, and a deformation $R'[[Y_n]]/R'[[X]]$ of $k[[y_n]]/k[[x]]$ over $R'$ that extends $R'[[Y_{m}]]/R'[[X]]$.
\end{corollary}

We will prove Corollary \ref{cordeformcyclicintowers} inductively using the below result.

\begin{theorem}
\label{thminductionmain}
Suppose $k[[y_n]]/k[[x]]$ is cyclic Galois of order $p^n$, and $R[[Y_{n-1}]]/ \allowbreak R[[X]]$ is a deformation of the $\mathbb{Z}/p^{n-1}$-subextension $k[[y_{n-1}]]/k[[x]]$ over a finite extension $R$ of $k[[t]]$. Then, there exists a finite extension $R'/R$, and a deformation $R'[[Y_n]]/R'[[X]]$ of $k[[y_n]]/k[[x]]$ over $R'$ that extends $R'[[Y_{n-1}]]/R'[[X]]$.
\end{theorem}
Assuming Theorem \ref{thminductionmain}, Corollary \ref{cordeformcyclicintowers} easily follows.

\subsection{Deformation of one-point-covers}
\label{secdeformonepointcover}
Let us once more consider the $\mathbb{Z}/p^n$-extension $L_n/k[[x]]$. Harbater, Katz, and Gabber shows that there exists a unique $\mathbb{Z}/p^n$-cover, which is usually known as the HKG cover of $L_n/k[[x]]$, $\overline{Y}_n \xrightarrow{} \overline{C}:=\mathbb{P}^1_k$ that is {\'e}tale outside $x=0$, totally ramified at $x=0$, and the formal completion of $\overline{Y} \xrightarrow{} \overline{C}$ at $x=0$ yields the extension $L_n/k[[x]]$ (see \cite{MR579791} and \cite{MR867916}). That allows one to go from a local back to a global situation. Specially, Theorem \ref{thminductionmain} is equivalent to the following version for one-point-covers (of curves), which is compatible with the language used in \cite{MR3194815} and allows us to only deal with rational functions instead of Laurent series. The proof of the below proposition, hence Theorem \ref{thminductionmain} is deferred to \S \ref{secproofmain}.

\begin{proposition}
\label{propmainonepointcover}
Suppose $k[[y_n]]/k[[x]]$ is a $G\cong \mathbb{Z}/p^n$-Galois extension, and $\psi_{n-1}: Y_{n-1}\xrightarrow{} C:=\mathbb{P}^1_K$ is a $\mathbb{Z}/p^{n-1}$-cover with the following properties:
\begin{enumerate}[label=(\arabic*)]
    \item \label{propmainonepointcover1} The cover $\psi_{n-1}$ has good reduction with respect to the standard model $\mathbb{P}^1_R$ of $C$ and reduces to a $\mathbb{Z}/p^{n-1}$-cover $\overline{\psi}_{n-1}: \overline{Y}_n \xrightarrow{} \overline{C} \cong \mathbb{P}^1_k$ that is totally ramified above $x=0$ and {\'e}tale everywhere else.
    \item \label{propmainonepointcover2} The completion of $\overline{\psi}_{n-1}$ at $x=0$ yields $k[[y_{n-1}]]/k[[x]]$, the unique $\mathbb{Z}/p^{n-1}$-sub-extension of $k[[y_n]]/k[[x]]$. We thus may assume that $k[[y_n]]/k[[x]]$ is given by $\underline{g}_n=(g^1, \ldots, g^n) \in W_n(k(x))$, i.e., all the entries are rational.
\end{enumerate}
Then, after possibly a finite extension of $R$, there exists a $G$-Galois cover $\psi_n: Y_n \xrightarrow{} C$ with good reduction that extends $\psi_{n-1}$ and verifies the followings.
\begin{enumerate}
    \item \label{propmainonepointcover3} Its reduction $\overline{\psi}_n: \overline{Y}_n \xrightarrow{} \overline{C}$ is totally ramified above $\overline{x}=0$ and {\'e}tale elsewhere.
    \item \label{propmainonepointcover4} The completion of $\overline{\psi}_n$ at $x=0$ yields $k[[y_n]]/k[[x]]$.
\end{enumerate}
\end{proposition}

\begin{remark}
Items \ref{propmainonepointcover3} and \ref{propmainonepointcover4} of Proposition \ref{propmainonepointcover} can be reformulated as follows:
\begin{enumerate}[label=(\alph*)]
    \item The cover $Y_n \xrightarrow{} C$ is {\'e}tale outside the open disc
    \[ D:=\{ X \in K \mid \lvert X\rvert <1\}. \]
    \item The inverse image of $D$ in $Y_n$ is an open disc.
    \item If $A=R[[X]]\{X^{-1}\}$ is the ring 
    \[ \bigg\{ \sum_{j \in \mathbb{Z}} a_j X^j \mid a_j \in R, a_j \rightarrow 0 \text{ as } j \rightarrow -\infty \bigg\} \hspace{5mm} \text{ (see \S \ref{secdiscannuli})} \]
    the cover $Y_n \xrightarrow{} C$ is unramified when base changed to $\spec A$, which can be thought of as the ``boundary'' of the disc $D$. The extension of residue fields is isomorphic to the extension of fraction fields coming from $k[[y_n]]/k[[x]]$.
\end{enumerate}
\end{remark}

\subsubsection{Proof of Proposition \ref{proplocalglobalintowers}'s ``$\Rightarrow$''}
With the assumptions of Theorem \ref{theoremmainlocal}, there exists an $H$-cover $\psi_m: Y_m \rightarrow \mathbb{P}^1_K=\Proj K[X,V]$ that
\begin{enumerate}[label=(\arabic*)]
    \item verifies \ref{propmainonepointcover1} and \ref{propmainonepointcover2} of Proposition \ref{propmainonepointcover}, where $H$ is in place of $\mathbb{Z}/p^{n-1}$ ($G$ is in place of $\mathbb{Z}/p^n$), and
    \item its standard model's completion at $X=0$ is isomorphic to $R[[Y_m]]/R[[X]]$,
\end{enumerate}
as discussed in \S \ref{secdeformonepointcover}. It thus follows from Theorem \ref{thmdeformtowers} that one can extends $\psi_m$ to a $G$-cover $\psi_n$ over $K$ whose standard model's completion at $X=0$ is a deformation $\Phi_n$ of $\phi_n$ over $R$ that contains $R[[Y_{m}]]/R[[X]]$ as the sub-$H$-cover. The extension $\Phi_n$ is exactly what we seek. \qed

\section{Degeneration of cyclic covers}
\label{sectcharandswan}
In this section, we study the degeneration of $G:=\mathbb{Z}/p^n$-covers of a curve $C$ as in \S \ref{secdeformonepointcover}.

\subsection{Setup}
Throughout the section, $R$ is a complete DVR of characteristic $p$ with uniformizer $\pi$, valuation $\nu$, and residue field $k$ (e.g., $R=k[[t]]$ and $\pi=t$). Normalize the valuation on $R$ so that $\nu(\pi)=1$. Let $K$ be $R$'s fraction field. Let $C$ be a smooth, projective, irreducible curve over $K$. Note that, for the purpose of this paper, we only need to consider the case $C \cong \mathbb{P}^1_K$. We denote by $\mathbb{K}$ the function field of $C$ (e.g., $\mathbb{K}=\Frac K[X]$). We may fix a smooth $R$-model $C_R$ of $C$ over $R$. Fix a rational point $x_0$ on $C$ and write $\overline{x}_0 \in \overline{C}$ for its specialization. 

We denote by $C^{\an}$ the rigid analytic space associated to $C$. The \textit{residue class} of $x_0$ with respect to the model $C_R$, denoted by
\[ D:=]\overline{x}_0[_{C_R} \subset C^{\an}, \]
is the set of points of $C^{\an}$ specializing to $\overline{x}_0$. It is an open subspace of $C^{\an}$ and is isomorphic to the open unit disc. 

The central intention of this section is examining the degeneration of cyclic covers of $C^{\an}$ that are {\'e}tale outside $D$. In the next sub-section, we briefly recall some non-archimedean geometry notions, which are crucial for the study of these degeneration.

\subsection{Discs and annuli}
\label{secdiscannuli}
Suppose $\epsilon \in \mathbb{Q}_{\ge 0}$, $r=p^{-\epsilon}$, $z \in R$, and $a \in K$ be such that $\nu(a)=\epsilon$. In Table \ref{tabnonarchimedean}, we list down some usual rigid geometry conventions. For more details, see \cite{MR1202394}, \cite{MR774362}, \cite{MR767194}.

\begin{table}[ht]
\small
    \centering
\begin{tabular}{ |p{4.8cm}|p{3.8cm}|p{5.8cm}| }
\hline
Notion & Algebra & Geometry  \\
 \hline
 \hline
Open unit disc & $\spec R[[X]]$  & $D=\{ u \in (\mathbb{A}^1_K)^{\an} \mid \nu(u)>0 \}$ \\
\hline
Closed unit disc & $\spec R\{X\}$ & $\mathcal{D}=\{ u \in (\mathbb{A}^1_K)^{\an} \mid \nu(u) \ge 0 \}$ \\
\hline
Boundary of a unit disc & $\spec R[[X]]\{X^{-1}\}$ & $\{ u \in (\mathbb{A}^1_K)^{\an} \mid \nu(u) = 0 \}$ \\
\hline
Open disc of radius $r$ center $z$ & $\spec R[[a^{-1}(X-z)]]$  & $D[\epsilon,z]:=\{ u \in (\mathbb{A}^1_K)^{\an} \mid \nu(u-z) > \epsilon \}$ \\
\hline
Open disc of radius $r$ center $0$ & $\spec R[[a^{-1}X]]$ & $D[\epsilon]:=D[\epsilon,0]$ \\
\hline
Open annulus of thickness $\epsilon$ & $\spec R[[X,U]]/(XU-a)$ & $\{ u \in (\mathbb{A}^1_K)^{\an} \mid 0< \nu(u) < \epsilon \}$ \\
\hline
\end{tabular}
\vspace{3mm}
    \caption{Non-archimedean geometry notions}
    \label{tabnonarchimedean}
\end{table}
Recall that $R\{X\} = \big\{ \sum_{i \ge 0} a_i X^i \in R[[X]] \mid \lim_{i \to \infty} \nu(a_i)=\infty  \big\}$. Let $z \in R$ and $s \in \mathbb{Q}_{\ge 0}$. One can associate with $\mathcal{D}[s,z]$ the ``Gauss valuation'' $\nu_{s,z}$ defined by
\begin{equation*}
    \nu_{s,z}(f)=\inf_{a \in \mathcal{D}[s,z]} (\nu(f(a)),
\end{equation*}
for each $f \in \mathbb{K}^{\times}$. This is a discrete valuation on $\mathbb{K}$ which extends the valuation $\nu$ on $K$, and has the property that $\nu_{s,z}(X-z)=s$. We denote by $\kappa_s$ the function field of $\mathbb{K}$ with respect to the valuation $\nu_{s,z}$. That is the function field of the canonical reduction $\overline{\mathcal{D}}[s,z]$ of $\mathcal{D}[s,z]$. In fact, $\overline{\mathcal{D}}[s,z]$ is isomorphic to the affine line over $k$ with function field $\kappa_{s,z}:=k(x_{s,z})$, where $x_{s,z}$ is the image of $\pi^{-s}(X-z)$ in $\kappa_{s,z}$. For a closed point $\overline{x} \in \overline{\mathcal{D}}[s,z]$, we let $\ord_{\overline{x}}: \kappa_{s,z}^{\times} \xrightarrow{} \mathbb{Z}$ denote the normalized discrete valuation corresponding to the specialization of $\overline{x}$ on $\overline{\mathcal{D}}[s,z]$. We let $\ord_{\infty}$ denote the unique normalized discrete valuation on $\kappa_{s,z}$ associated with the ``point at infinity.''

For $b \in \mathbb{K}$ and $r \in \mathbb{Q}_{\ge 0}$, we define $(b)_r:=b\pi^{-pr}$. We will usually write $X_r:=X\pi^{-pr}$. For $F \in \mathbb{K}$, $z \in (\mathbb{A}^1_K)^{\an}$, and $s \in \mathbb{Q}_{\ge 0}$, we let $[F]_{s,z}$ stand for the image of $\pi^{-\nu_{s,z}(F)}F$ in $\kappa_{s,z}$.

\subsection{Semi-stable models and a partition of a disc}
\label{secsemistablemodel}

Consider the open unit disc $D \subset C^{\an} \cong (\mathbb{P}^1_K)^{\an}$, which we may associate with $\spec R[[X]]$ for some $X \in \mathbb{K}$. Suppose we are given $x_{1,K}, \ldots, x_{r,K}$ in $D(K)$, with $r \ge 2$. We can think of $x_{1,K}, \ldots, x_{r,K}$ as elements of the maximal ideal of $R$. Let $C^{\st}$ be a blow-up of $C_R$ such that
\begin{itemize}
    \item the exceptional divisor $\overline{C}$ of the blow-up is a semi-stable curve over $k$,
    \item the fixed points $x_{b,K}$ specialize to pairwise distinct smooth points $x_b$ on $\overline{C}$, and
    \item if $x_0$ (which we usually denote by $\overline{\infty}$) denotes the unique point on $C$ which lies in the closure of $C^{\st} \otimes k \setminus \overline{C}$, then $(\overline{C},(x_b),x_0)$ is stably marked in the sense of \cite{MR702953}. 
\end{itemize}
\noindent Then we call $C^{\st}$ the \emph{stable model of $C$ corresponding to the marked disc $(D; x_1, \ldots, x_r)$}. See, e.g., \cite[\S 2.5.3]{Akeyrthesis} for more details. Note that the set of points of $(\mathbb{P}^1_K)^{\an}$ that specialize away from $x_0$ form the closed unit disc $\mathcal{D}$. The dual graph of $\overline{C}$ is a tree whose
\begin{itemize}
    \item \textit{leaves} correspond to the marked points, whose
    \item \textit{root} corresponds to $x_0$, whose
    \item \textit{vertices} correspond to the punctured discs, and whose 
    \item \textit{edges} correspond to the annuli
\end{itemize}
that partition $D$. For each edge $e$ of the tree, the \emph{source} (resp. the \emph{target}) of $e$ is the unique vertex  $s(e) \in V$ (resp. $t(e) \in V$)  adjacent to $e$ which lies in the direction toward the root (resp. in the direction away from the root). For each vertex $v$ of the dual graph, we denote by $U_v$ the corresponding closed punctured disc.

\begin{example}
Let us consider the $\mathbb{Z}/4$-cover in Example \ref{exmain}. Recall that $\chi_2$
branches at four points $0, t^8, t^2$, and $t^2(1+t^2)$, all are contained in the open disc $D$ associated to $\spec R[[X]]$. The left graph in Figure \ref{figspecialfiberdual} represents a semi-stable model of $\mathbb{P}^1_K$ corresponding to the disc $D$ marked by the branch points of the cover. The model is obtained by first blowing-up $\spec R[[X]]$ with respect to the ideal $(t^2, X)$, which separates the reductions of $t^2$ and $t^2+t^4$ from ones of $0$ and $t^{8}$. We then can distinguish $t^2$ and $t^2+t^4$ by doing the same for $\spec R[[X_1^{-1}]]$, where $X_1:=X/t^2$ with respect to $(t^2, X_1^{-1}-1)$.

The tree on the right is the special fiber's dual graph together with the leaves (labelled by $[\overline{t^2}], [\overline{t^2+t^4}], [\overline{t^{16}}]$, and $[\overline{0}]$) that are associated with the corresponding marked points. See \cite[Example 3.1]{2020arXiv200203719D} for more explanation.

The vertex $v_1$ represents the punctured disc associated with $\spec R\{ X_1^{-1}, \allowbreak X_1, \allowbreak (X_1^{-1}-1)^{-1} \}$. The points $t^2$ and $t^2+t^4$ (resp. $0$ and $t^{16}$) of $D$ lie inside $\spec R(X_1^{-1}-1)$ (resp. $\spec R(X_1)$), and reduce to $1$ (resp. $0$) on its special fiber. Table \ref{tabpartitions} illustrates where the $\overline{K}$-points of $D$ specialize.

\begin{remark}
\label{remarkdirection}
In the above example. we say the directions from $v_1$ toward $e_2, e_1$, and $e_0$, are the $0,1$, and $\infty$ \textit{directions}, respectively.
\end{remark}

\begin{center}
\begin{figure}[ht]
\centering
\begin{tikzpicture}[line cap=round,line join=round,>=triangle 45,x=1cm,y=1cm]
\clip(-5.913951741732328,-0.42740404202494114) rectangle (10.236822781996308,4.043296939568755);
\draw [line width=2pt] (1,3.5)-- (1,0);
\draw [line width=2pt] (-4.900693696805673,3.5)-- (-4.900693696805673,2.5);
\draw [line width=2pt] (0.5,2)-- (1.52,2);
\draw [line width=2pt] (0.5,1.2)-- (1.5,1.2);
\draw (-4.846846996414543,2.92063948653854) node[anchor=north west] {$\overline{\infty}$};
\draw (1.901869501000637,2.2829169096693054) node[anchor=north west] {$\overline{0}$};
\draw (1.8009271602273331,1.518978406128035) node[anchor=north west] {$\overline{t^{16}}$};
\draw (1.0510697716256465,0.8812558292588005) node[anchor=north west] {$\overline{C}_3$};
\draw [line width=2pt] (2.497760136302129,3)-- (-5.502239863697875,3);
\draw [line width=2pt] (-3.1869684357952335,3.461122757720566)-- (-3.1869684357952335,-0.03887724227943412);
\draw [line width=2pt] (-3.686968435795234,1.9611227577205659)-- (-2.666968435795234,1.9611227577205659);
\draw [line width=2pt] (-3.686968435795234,1.1611227577205658)-- (-2.686968435795234,1.1611227577205658);
\draw (-2.366549480270502,2.1832727570334876) node[anchor=north west] {$\overline{t^2}$};
\draw (-2.45307148664762,1.419334253492217) node[anchor=north west] {$\overline{t^2+t^4}$};
\draw (-2.856840849740836,0.8214693376773098) node[anchor=north west] {$\overline{C}_2$};
\draw (-1.4580684133107664,3.7975080297337374) node[anchor=north west] {$\overline{C}_1$};
\draw [line width=2pt] (5.5,1.5)-- (3.879577221035335,2.587879622356322);
\draw [line width=2pt] (5.5,1.5)-- (6.879577221035335,2.587879622356322);
\draw [line width=2pt] (5.5,1.5)-- (5.5,0);
\draw [line width=2pt] (3.879577221035335,2.587879622356322)-- (3.379577221035335,3.58787962235632);
\draw [line width=2pt] (3.879577221035335,2.587879622356322)-- (4.379577221035335,3.58787962235632);
\draw [line width=2pt] (6.879577221035335,2.587879622356322)-- (6.379577221035335,3.58787962235632);
\draw [line width=2pt] (6.879577221035335,2.587879622356322)-- (7.379577221035339,3.58787962235632);
\begin{scriptsize}
\draw [fill=black] (5.5,1.5) circle (2.5pt);
\draw[color=black] (6.15586814787559,1.4890851603372897) node {$v_1$};
\draw [fill=black] (3.879577221035335,2.587879622356322) circle (2.5pt);
\draw[color=black] (3.3439029406192655,2.565242008804123) node {$v_2$};
\draw[color=black] (4.944560058595942,2.3061672119509966) node {$e_1$};
\draw [fill=black] (6.879577221035335,2.587879622356322) circle (2.5pt);
\draw[color=black] (7.367176237155237,2.631671443894668) node {$v_3$};
\draw[color=black] (5.896302128744237,2.2596666073876146) node {$e_2$};
\draw [fill=black] (5.5,0) circle (2.5pt);
\draw[color=black] (6.213549485460335,0.0010658143090758049) node {$v_0$};
\draw[color=black] (5.146444740142551,0.7450754873231826) node {$e_0$};
\draw [fill=black] (3.379577221035335,3.58787962235632) circle (2.5pt);
\draw[color=black] (3.7620926381086672,3.727757122888665) node {$[\overline{t^2}]$};
\draw [fill=black] (4.379577221035335,3.58787962235632) circle (2.5pt);
\draw[color=black] (5.002241396180688,3.727757122888665) node {$[\overline{t^2+t^{4}}]$};
\draw [fill=black] (6.379577221035335,3.58787962235632) circle (2.5pt);
\draw[color=black] (6.720145870873463,3.727757122888665) node {$[\overline{0}]$};
\draw [fill=black] (7.379577221035339,3.58787962235632) circle (2.5pt);
\draw[color=black] (7.801670950587435,3.7211141793796103) node {$[\overline{t^{16}}]$};
\end{scriptsize}
\end{tikzpicture}
    \caption{The special fiber $\overline{C}$ and its dual graph}
    \label{figspecialfiberdual}
\end{figure}
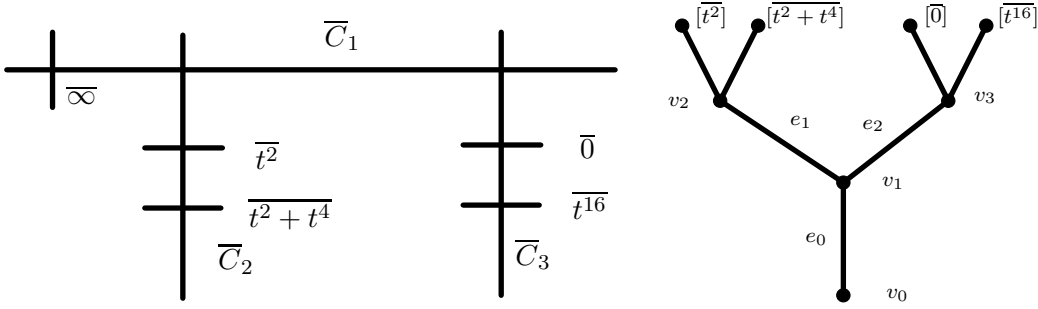
\end{center}
\begin{table}[ht]
\small
    \centering
\begin{tabular}{ |p{3.4cm}|p{5.6cm}|p{5.2cm}|  }
\hline
 Subscheme $\overline{V}$ of $\overline{C}$ & Points of $C(K)$ that specialize to $\overline{V}$ & Associated algebraic object \\
 \hline
 \hline
$\overline{\infty}$ & $ \{Y \mid 0< \nu(Y)<2  \}$  & $ R[[X,X_1]]/(XX_1-t^2)$ \\
\hline
$\overline{C}_3 \cap \overline{C}_1$ & $\{Y \mid 2  < \nu(Y) < 8\}$ & $ R[[X_1^{-1},X_2]]/(X_1^{-1}X_2-t^6)$ \\
\hline
$\overline{C}_3 \setminus \overline{C}_1$ & $\{Y \mid \nu(Y) \ge 8 \}$ & $ R\{X_2^{-1}\}$   \\
\hline
$\overline{C}_2 \cap \overline{C}_1$ & $\{Y \mid 2 < \nu(Y-t^5) <8 \}$ & $ R[[X_1^{-1}-1,V]](V(X_1^{-1}-1)-t^2)$ \\
\hline
$\overline{C}_2 \setminus \overline{C}_1$ & $\{Y \mid \nu(Y-t^5) \ge 2 \wedge \nu(Y)=2 \}$  & $ R\{V^{-1}\}$  \\ 
\hline
$\overline{C}_1 \setminus ( \overline{C}_3 \cup \overline{C}_2 \cup \{\overline{\infty}\})$ & $\{Y \mid \nu(Y)=2 \wedge \nu(Y-t^2)=2 \}$ & $ R\{X_1^{-1},X_1,(X_1^{-1}-1)^{-1}\}$ \\
     \hline
\end{tabular}
    \caption{Partitions of $C(K)$}
    \label{tabpartitions}
\end{table}
\end{example}

\subsubsection{A coordinate system for a marked disc} With the notation of \S \ref{secsemistablemodel}, we define a  one-dimensional ``coordinate system'' for each edge $e$ of the special fiber $\overline{C}$'s dual graph as follows. One may assume that $e$ corresponds to the annulus $\{Y \in K \mid pr_1 < \nu(Y-x_j) <pr_2\}$. We identify $e$ with a rational line segment $[r_1, r_2] \cap \mathbb{Q}$, and assign the value $r_1$ to $s(e)$ and the value $r_2$ to $t(e)$. We then naturally associate a rational number $r \in [r_1, r_2]$ with the circle $\{ Y \mid \nu(Y-x_j)= pr \}$. We call $r$ a \emph{rational place} on $e$. Suppose $r$ is a rational place on an edge $e$, and $r'$ is a rational place on a succeeding edge of $e$. Then we say $r<r'$. By abuse of notation, we usually write $t(e)$ in place of $r_2$ and $s(e)$ in place of $r_1$. The reason we scale the coordinate by $1/p$ is to make it compatible with the calculation of the depth Swan conductor introduced in \S \ref{secSwan}, which, in turn, is also modified to be consistent with \cite{MR2377173} and \cite{2020arXiv200203719D}.

\subsection{Characters}
\label{seccharacter}
Fix $n \ge 1$. We set
\begin{equation}
    \label{eqnaswmap}
    \textrm{H}^1_{p^n}(\mathbb{K}):=\textrm{H}^1(\mathbb{K}, \mathbb{Z}/p^n \mathbb{Z}) \overset{\ASW}{\cong} W_n(\mathbb{K})/\wp(W_n(\mathbb{K})). 
\end{equation}
The ``$\overset{\ASW}{\cong}$'' in (\ref{eqnaswmap}) is just Artin-Schreier-Witt theory (\S \ref{secASWtheory}). We call the identification $\ASW$ the \textit{Artin-Schreier-Witt map}. The elements of $\textrm{H}^1_{p^n}(\mathbb{K})$ are called the \textit{characters} on $C$. Given an element $\underline{F}_n \in W_n(\mathbb{K})$, we let $\mathfrak{K}_n(\underline{F}_n) \in \textrm{H}^1_{p^n}(\mathbb{K})$ denote the character corresponding to the class of $\underline{F}_n$ in $W_n(\mathbb{K})/\wp(W_n(\mathbb{K}))$. 

For $i=1, \ldots, n$ the homomorphism
\[ W_{i}(\mathbb{F}_p) \cong  \mathbb{Z}/p^i \mathbb{Z} \xhookrightarrow{V^{n-i}} \mathbb{Z}/p^n \mathbb{Z} \cong W_n (\mathbb{F}_p), \hspace{3mm} (a^1, \ldots, a^i) \mapsto (0, \ldots, 0,a^1, \ldots, a^i). \]
\noindent induces an injective homomorphism
\begin{equation}
    \label{embedchar}
    \textrm{H}^1_{p^i}(\mathbb{K}) \xrightarrow{\phi_{i,n}} \textrm{H}^1_{p^n}(\mathbb{K}).
\end{equation}
\noindent Its image consists of all characters killed by $p^i$. We consider $\textrm{H}^1_{p^i}(\mathbb{K})$ as a subgroup of $\textrm{H}^1_{p^n}(\mathbb{K})$ via this embedding.

A character $\chi \in \textrm{H}^1_{p^n}(\mathbb{K})$ gives rise to a branched Galois cover $Y \rightarrow C$. In particular, if $\chi=\mathfrak{K}_n(\underline{F})$ for some $\underline{F} \in W_n(\mathbb{K})$, then $Y$ is a connected component of the smooth projective curve given generically by the ASW equation $\wp(y)=\underline{F}$. If $\chi$ has order $p^i$ as an element of $\textrm{H}^1_{p^n}(\mathbb{K})$, then the Galois group of $Y \rightarrow C$ is the unique quotient of $\mathbb{Z}/p^n\mathbb{Z}$ of order $p^i$.

\begin{remark}
If one identifies $\textrm{H}^1_{p^i}(\mathbb{K})$ with $W_i(\mathbb{K})/\wp(W_i(\mathbb{K}))$, then the map $\phi_{i,n}$ assigns to each class of length $i$-Witt vector $(a^1, \ldots, a^i)$ a class of length $n$-Witt vector $\big(0, \ldots, 0, a^1,\allowbreak \ldots, a^i\big)$. 
\end{remark}

\begin{remark}
If $\chi:=\mathfrak{K}((a^1, \ldots, a^n)) \in \textrm{H}^1_{p^n}(\mathbb{K})$, then $\chi^{p^{n-i}} \cong \mathfrak{K}((a^1, \ldots, a^i)) \in \textrm{H}^1_{p^i}(\mathbb{K})$.
\end{remark}

A point $x \in C$ is called a \textit{branch point} for the character $\chi \in \textrm{H}^1_{p^n}(\mathbb{K})$ if it is a branch point for the cover $Y \rightarrow C$. The \textit{branching index} of $x$ is the order of the inertia group for some point $y \in Y$ above $x$. The set of all branch points is called the \textit{branch locus} of $\chi$ and is denoted by $\mathbb{B}(\chi)$.

\begin{definition}
\label{defbranchlocus}
A character $\chi \in \textrm{H}^1_{p^n}(\mathbb{K})$ is called \textit{admissible} if its branch locus $\mathbb{B}(\chi)$ is contained in an open unit disc $D \subset C^{\an}$. We call its associated cover an \textit{admissible cover}.
\end{definition}

\begin{definition}
\label{defngeometryofbranchpoints}
Suppose $\chi:=\mathfrak{K}((f^1, \ldots, f^n))  \in \textrm{H}^1_{p^n}(\mathbb{K})$ is an admissible character and $\mathbb{B}(\chi)=\{b_1, \ldots, b_l\}$. We call the dual graph of the semi-stable model of $\mathbb{P}^1_K$ corresponding to the marked disc $(D;b_1, \ldots, b_l)$ (discussed in \S \ref{secsemistablemodel}) the \textit{branching geometry} of $\chi$, or the \textit{geometry of the branch points} of $\chi$, or the \textit{geometry of the poles} of $(f^1, \ldots, f^n)$.
\end{definition}

\subsubsection{Reduction of characters}
\label{secreduction}

Let $\chi \in \textrm{H}^1_{p^n}(\mathbb{K})$ be an admissible character of order $p^n$, and let $Y \rightarrow C$ be the corresponding cyclic Galois cover. Let $Y_R$ be the normalization of $C_R$ in $Y$. Then $Y_R$ is a normal $R$-model of $Y$ and we have $C_R=Y_R/(\mathbb{Z}/p^n)$.

After enlarging our ground field $K$, we may assume that the character $\chi$ is weakly \textit{unramified} with respect to the valuation $\nu_0$ (where $X$ has valuation $0$) \cite{MR0321929} (there is an error, corrected in \cite{MR2000472}). By definition, this means for all extension $\omega$ of $\nu_0$ to the function field of $Y$ the ramification index $e(\omega/\nu_0)$ is equal to $1$. It then follows that the special fiber $\overline{Y}=Y_R \otimes_R k$ is reduced \cite[\S 2.2]{2012arXiv1211.4624A}.

\begin{definition}
\label{defnetalegoodreduction}
We say that the character $\chi$ has \textit{étale reduction} if the map $\overline{Y} \rightarrow \overline{C}$ is generically étale. It has \textit{good reduction} if, in addition, $\overline{Y}$ is smooth.
\end{definition}

\begin{remark}
\label{remarketalereduction}
In terms of Galois cohomology, the definition can be rephrased as follows. The character $\chi$ has {\'e}tale reduction if and only if the restriction of $\chi$ to the completion $\hat{\mathbb{K}}_0$ of $\mathbb{K}$ with respect to $\nu_0$ is \textit{unramified}. The latter means that $\chi\lvert_{\hat{\mathbb{K}}_0}$ lies in the image of the cospecialization morphism
\[ \textrm{H}^1_{p^n}(\kappa_0) \xrightarrow{} \textrm{H}^1_{p^n}(\hat{\mathbb{K}}_0) \]
\noindent (which is simply the restriction morphism induced by the projection of absolute Galois groups $\gal_{\hat{\mathbb{K}}_0} \xrightarrow{} \gal_{\kappa_0}$). Since the cospecialization morphism is injective, there exists a unique character $\overline{\chi} \in \textrm{H}^1_{p^n}(\kappa_0)$ whose image in $\textrm{H}^1_{p^n}(\hat{\mathbb{K}}_0)$ is $\chi\lvert_{\hat{\mathbb{K}}_0}$. By construction, the Galois cover of $\overline{C}$ corresponding to $\overline{\chi}$ is isomorphic to an irreducible component of the normalization of $\overline{Y}$.
\end{remark}

\begin{definition}
\label{defnreductiondeformation}
If $\chi$ has étale reduction, we call $\overline{\chi}$ the \textit{reduction} of $\chi$, and $\chi$ a \textit{deformation} of $\overline{\chi}$ over $R$.
\end{definition}

With the new definition of good reduction, we can reformulate Proposition \ref{propmainonepointcover} as follows.

\begin{proposition}
\label{propinductionmainreduction}
Proposition \ref{propmainonepointcover} holds if there exists an admissible extension $\psi_n: Y_n \xrightarrow{} C$ of $\psi_{n-1}: Y_{n-1} \xrightarrow{} C$ with the following properties.
\begin{enumerate}
    \item $\psi_n$ has good reduction, and
    \item The completion of the reduction $\overline{\psi}_n$ at $x=0$ is birationally equivalent to $k[[y_n]]/k[[x]]$.
\end{enumerate}
\end{proposition}

\subsection{Swan conductors}
\label{secSwan}

Suppose $\chi \in \textrm{H}^1_{p^n}(\mathbb{K})$ is a character associated with a cyclic $p^n$-exponent cover of $C$. Suppose, moreover, that $\mathcal{D} \subset C^{\an}$ is a closed disc equipped with the topology corresponding to the canonical valuation $\nu_0$ (\S \ref{secdiscannuli}). After enlarging $R$, we may assume that the restriction $\chi\lvert_{\mathcal{D}}$ is weakly unramified with respect to $\nu_0$. As usual, the residue field of $\Frac (\mathcal{D})$ is denoted by $\kappa$.

As in \cite[\S 3.4]{2020arXiv200203719D}, we define two invariants that measure the degeneration of $\chi\lvert_{\mathcal{D}}$ with respect to the valuation $\nu_0$. The \textit{depth} is
\[ \delta_{\chi\lvert_{\mathcal{D}}}:=\sw({\chi\lvert_{\mathcal{D}}})/p \in \mathbb{Q}_{\ge 0}, \]
\noindent where $\sw({\chi\lvert_{\mathcal{D}}})$ is the classical Swan conductor \cite[Definition 3.3]{MR991978} of $\chi\lvert_{\mathcal{D}}$. Note that the rational number $\delta_{\chi\lvert_{\mathcal{D}}}$ is equal to $0$ if and only if $\chi\lvert_{\mathcal{D}}$ is unramified. If this is the case, then its reduction $\overline{\chi}\lvert_{\mathcal{D}}$ is well-defined. In particular, if $\chi\lvert_{\mathcal{D}}$ is of order $p^n$ and $\delta_{\chi\lvert_{\mathcal{D}}}=0$, then there exists $\underline{f} \in W_n(\kappa)$ such that $\overline{\chi}\lvert_{\mathcal{D}}$ is defined by $\wp(y)=\underline{f}$. We call $\underline{f}$ a \textit{reduction} of $\chi\lvert_{\mathcal{D}}$. Note also that, as discussed in \S \ref{secASWtheory}, a reduction $\underline{f}$ is unique up to adding an element of the form $\wp(a)$, where $a \in W_n(\kappa)$, and there exists a unique $\underline{f}^{\reduce}=\underline{f}+\wp(b)$ for some $b \in W_n(\kappa)$, which we called the $\chi\lvert_{\mathcal{D}}$'s \textit{reduced reduction}. We say $\chi\lvert_{\mathcal{D}}$ is \textit{radical} if $\delta_{\chi\lvert_{\mathcal{D}}}>0$.

Suppose $\delta_{\chi\lvert_{\mathcal{D}}}>0$. Then we can define the \textit{differential Swan conductor} or \textit{differential conductor} 
\[ \omega_{\chi\lvert_{\mathcal{D}}}:=\dsw({\chi\lvert_{\mathcal{D}}}) \in \Omega^1_{\kappa}  \]
identically to the mixed characteristic case in \cite[Definition 3.9]{MR991978} (see also \cite[\S 3.2]{MR3167623}). It is derived from the \textit{refined Swan conductor} $\rsw^{\ab}(\chi\lvert_{\mathcal{D}})$ \cite{MR991978} and depends on the choice of the uniformizer $\pi$ for $R$. In particular, we have the relation
\begin{equation}
    \label{eqnrswab}
    \rsw^{\ab}(\chi\lvert_{\mathcal{D}})= \pi^{-sw(\chi\lvert_{\mathcal{D}})} \otimes \dsw(\chi\lvert_{\mathcal{D}}) \in \mathfrak{m}^{-\sw(\chi\lvert_{\mathcal{D}})} \otimes_{\mathcal{O}_{\mathbb{K}}} \Omega^1_{\kappa},
\end{equation}
where $\mathfrak{m}$ is the maximal ideal of $R$. Note that $\rsw^{\ab}(\chi\lvert_{\mathcal{D}})$ does not depend on the choice of $\pi$.

We call the pair $(\delta_{\chi\lvert_{\mathcal{D}}}, \omega_{\chi\lvert_{\mathcal{D}}})$ when $\delta_{\chi\lvert_{\mathcal{D}}}>0$, or $(\delta_{\chi\lvert_{\mathcal{D}}}, \underline{f})$ when $\delta_{\chi\lvert_{\mathcal{D}}}=0$ and $\underline{f}$ is the reduced reduction of $\chi\lvert_{\mathcal{D}}$, the \textit{degeneration type} or the \textit{reduction type} of the restriction $\chi\lvert_{\mathcal{D}}$. 

Suppose $\overline{x}$ is a point on the canonical reduction of $\mathcal{D}$ or the point at infinity, which we write $\overline{x}=\infty$, and let $\ord_{\overline{x}}: \kappa^{\times} \xrightarrow{} \mathbb{Z}$ be the normalized discrete valuation (whose restriction to $k$ is trivial) corresponding to $\overline{x}$. Then the composite of $\nu_0$ with $\ord_{\overline{x}}$ is a valuation on $K$ of rank $2$, which we denote by $\mathbb{K}^{\times} \xrightarrow{} \mathbb{Q} \times \mathbb{Z}$. In \cite{MR904945}, Kato defined a Swan conductor $\sw^K_{\chi\lvert_{\mathcal{D}}} (\overline{x}) \in \mathbb{Q}_{\ge 0} \times \mathbb{Z}$. Its first component is equal to $p\delta_{\chi\lvert_{\mathcal{D}}}$.
We define the \textit{boundary Swan conductor with respect to} $\overline{x}$,
$$\sw_{\chi\lvert_{\mathcal{D}}}(\overline{x}) \in \mathbb{Z},$$ as the second component of $\sw^K_{\chi\lvert_{\mathcal{D}}}(\overline{x})$. Geometrically, it gives the instantaneous rate of change of the depth in the direction (see Remark \ref{remarkdirection}) corresponding to $\overline{x}$. We will discuss in details this phenomenon in \S \ref{secboundaryswanrateofchange}.

\begin{remark}
\label{remarkboundaryswan}
The invariant $\sw_{\chi\lvert_{\mathcal{D}}}(\overline{x})$ is determined by $\delta_{\chi\lvert_{\mathcal{D}}}$ and $\omega_{\chi\lvert_{\mathcal{D}}}$ as follows:

\begin{enumerate}
    \item \label{remarkboundaryswan1} If $\delta_{\chi\lvert_{\mathcal{D}}}=0$, then
    \[ \sw_{\chi\lvert_{\mathcal{D}}}(\overline{x})=\sw_{\overline{\chi}\lvert_{\mathcal{D}}}(\overline{x}), \]
    where $\overline{\chi}\lvert_{\mathcal{D}}$ is the reduction of $\chi\lvert_{\mathcal{D}}$ (remark \ref{remarketalereduction}) and $\sw_{\overline{\chi}\lvert_{\mathcal{D}}}(\overline{x})$ is the usual Swan conductor of $\overline{\chi}\lvert_{\mathcal{D}}$ with respect to the valuation $\ord_{\overline{x}}$ \cite[IV,\S 2]{MR554237}. That follows immediately from the definitions of the conductors. We thus have $\sw_{\chi\lvert_{\mathcal{D}}}(\overline{x})\ge 0$ and $\sw_{\chi\lvert_{\mathcal{D}}}(\overline{x})=0$ if and only if $\overline{\chi}\lvert_{\mathcal{D}}$ is unramified with respect to $\ord_{\overline{x}}$.
    \item \label{remarkboundaryswan2} If $\delta_{\chi\lvert_{\mathcal{D}}}>0$, then \cite[Corollary 4.6]{MR904945} asserts that
    \[ \sw_{\chi\lvert_{\mathcal{D}}}(\overline{x})=-\ord_{\overline{x}}(\omega_{\chi\lvert_{\mathcal{D}}})-1. \]
    This fact will be used intensively later in this paper.
\end{enumerate}
\end{remark}

\subsubsection{Refined Swan conductors of a product of characters}

The Swan conductors behave somewhat like valuations.

\begin{lemma}
\label{lemmacombination}
Let $\chi_i$, with $i=1,2,3$, be abelian characters on a disc satisfying the relation $\chi_3=\chi_1 \cdot \chi_2$. Set $\delta_i:=\sw(\chi_i)/p$ (with the canonical valuation) and $\omega_i:=\dsw(\chi_i)$, for $i=1,2,3$. Then the following holds.
\begin{enumerate}[label=(\arabic*)]
    \item \label{lemmacombination1} If $\delta_1 \neq \delta_2$ then $\delta_3=\max \{\delta_1,\delta_2 \}$. Furthermore, we have $\omega_3=\omega_1$ if $\delta_1>\delta_2$ and $\omega_3=\omega_2$ otherwise.
    \item \label{lemmacombination2} If $\delta_1=\delta_2>0$ and $\omega_1+\omega_2 \neq 0$ then $\delta_1=\delta_2=\delta_3$ and $\omega_3=\omega_1+\omega_2$.
    \item \label{lemmacombination3} If $\delta_1=\delta_2>0$ and $\omega_1+\omega_2=0$ then $\delta_3 < \delta_1$.
    \item \label{lemmacombination4} if $\delta_1=\delta_2=0$, then $\delta_3=0$ and $\overline{\chi}_3=\overline{\chi}_1 \cdot \overline{\chi}_2$. Hence, if the reduction type of $\chi_1$ (resp., $\chi_2$) is $\underline{f}$ (resp., $\underline{g}$), then $\chi_3$'s reduction type is $\underline{f}+\underline{g}$.
\end{enumerate}
\end{lemma}

\begin{proof}
Parallel to the proof for \cite[Proposition 5.6]{MR3194815}.
\end{proof}

\begin{remark}
The above lemma allows one to compute the conductors of one character by breaking it down into easier-to-calculate-ones. That philosophy will be employed frequently later in this paper.
\end{remark}

\begin{remark}
With the notation of Lemma \ref{lemmacombination}, suppose $m \ge n \in \mathbb{Z}_{\ge 0}$, and $\chi_1 \in \textrm{H}^1_{p^n}(\mathbb{K})$ (resp., $\chi_2 \in \textrm{H}^1_{p^m}(\mathbb{K})$) is defined by the Witt vector $\underline{G} \in W_n(\mathbb{K})$ (resp. $\underline{H} \in W_m(\mathbb{K})$). Then $\chi_1 \cdot \chi_2$ correspond to the Witt vector $\underline{H}+ \phi_{n,m}(\underline{G}) \in W_m(\mathbb{K})$ ($\phi_{n,m}$ is defined in (\ref{embedchar})). 
\end{remark}

\subsubsection{}
\label{secboundaryswanrateofchange}
Fix a closed disc $\mathcal{D}$, which may be associated with the ring $R\{X\}$. Let $z$ be a $K$-point of $\mathcal{D}$ and $r \in \mathbb{Q}_{\ge 0}$. We denote by $\delta_{\chi}(r,z)$ (resp. $\delta_{\chi}(r)$) and $\omega_{\chi}(r,z)$ (resp. $\omega_{\chi}(r)$) the depth and the differential conductors of the restriction of $\chi$ to $\mathcal{D}[pr,z]$ (resp. $\mathcal{D}[pr]$). Following the notation in \cite{MR3194815}, we define $\nu_r: \mathbb{K}^{\times} \xrightarrow{} \mathbb{Q}$ as the Gauss valuation associated with $D[pr]$, such that $\nu_r(X) = pr$. Furthermore, let $\kappa_r$ represent the residue of $\mathbb{K}$ with respect to the valuation $\nu_r$.

When the point $z$ is fixed, we may regard $\delta_{\chi}(r,z)$ and $\omega_{\chi}(r,z)$ as  functions of $r$. Suppose that $\overline{y}$ is a point on the reduction of $\mathcal{D}[pr,z]$, or a point at infinity $\overline{y}=\overline{\infty}$. Set $\sw_{\chi}(r,z,\overline{y}):=\sw_{\chi\lvert_{\mathcal{D}[pr,z]}}(\overline{y})$. The following results show how understanding a character's differential conductor can give a lot of information about its depth. They are the exact analogs of ones from \cite[\S 5.3.2]{MR3194815} in mixed characteristic.

\begin{proposition}
\label{propdeltasw}
Suppose $z \in R$ is fixed. Then $\delta_{\chi}(\--, z)$ extends to a continuous, piecewise linear function 
\[ \delta_{\chi}(\--, z): \mathbb{R}_{\ge 0} \xrightarrow{} \mathbb{R}_{\ge 0}. \]
\noindent Furthermore,
\begin{enumerate}
    \item \label{propdeltasw1} For $r \in \mathbb{Q}_{>0}$, the left (resp. right) derivative of $\delta_{\chi}(\--, z)$ at $r$ is $-\sw_{\chi}(r,z,\overline{\infty})$ (resp. $\sw_{\chi}(r,z,\overline{0})$).
    \item \label{propdeltasw2} If $r$ is a kink of $\delta_{\chi}(\--, z)$ (meaning the left and the right derivaties do not agree), then $r \in \mathbb{Q}$.
\end{enumerate}
\end{proposition}

\begin{proof}
See, e.g., \cite[Proposition 2.9]{2005math.....11434W} or \cite[Theorem 1.9]{2020arXiv201014843B}. Note that, in the language of \cite{2020arXiv201014843B}, a lisse {\'e}tale sheaf of $\mathcal{F}_{\ell}$-module on a rigid disc $\mathcal{D}$ coincides with a connected {\'e}tale cover of $\mathcal{D}$ \cite[Theorem 2.10]{CM_1995__97_1-2_89_0}.
\end{proof}

\begin{corollary}
\label{corleftrightderivative}
If $r \in \mathbb{Q}_{\ge 0}$ and $\delta_{\chi}(r,z)>0$, then the left and right derivatives of $\delta_{\chi}$ at $r$ are given by $\ord_{\overline{\infty}}(\omega_{\chi}(r))+1$ and $-\ord_{\overline{0}}(\omega_{\chi}(r))-1$, respectively. In particular, $-\ord_{\overline{x}}(\omega_{\chi}(r,z))-1$ is the instantaneous rate of change of $\delta_{\chi}(r,z)$ on the direction with respect to $\overline{x}$.
\end{corollary}

\begin{proof}
Immediate from Proposition \ref{propdeltasw}(\ref{propdeltasw1}) and Remark \ref{remarkboundaryswan} (\ref{remarkboundaryswan2}).
\end{proof}

\begin{definition}
\label{defncondopendisc}
Due to the above results, we can define the refined Swan conductors of the restriction of $\chi$ to an open disc $D[pr,z]$ as $\delta(\chi\lvert_{D[pr,z]}):=\delta(\chi\lvert_{\mathcal{D}[pr,z]})$ and $\omega(\chi\lvert_{D[pr,z]}):=\omega(\chi\lvert_{\mathcal{D}[pr,z]})$. We also set the boundary conductor of $\chi\lvert_{D[pr,z]}$ to be that of $\chi\lvert_{\mathcal{D}[pr,z]}$.
\end{definition}

The below result describes how the refined Swan conductors vary under a change of coordinate. In particular, it shows that the depth Swan conductor of a cover (of a disc) does not depend on the choice of the parameter.

\begin{proposition}
Suppose that $z_1$ and $z_2$ are two $K$-point on the interior of a closed unit disc $\mathcal{D}$. Suppose, moreover, that $\nu(z_1-z_2)\ge r$ for some non-negative rational $r$, and $(z)_r:=(z_1-z_2)\pi^{-r}$. Then
\[ \delta_{\chi}(r,z_1)=\delta_{\chi}(r,z_2). \]
Moreover, if $\omega_{\chi}(r,z_1)=f(x)dx$, then $\omega_{\chi}(r,z_2)=f(x+\overline{(z)}_r)dx$. In particular, if $\nu(z_1-z_2)>r$ then $\omega_{\chi}(r,z_1)$ coincides with $\omega_{\chi}(r,z_2)$.
\end{proposition}

\begin{proof}
We first note that $\mathcal{D}[pr,z_1]=\mathcal{D}[pr,z_2]$ as $\nu(z_1-z_2) \ge r$. The first part is then immediate from the definition of the depth Swan conductor. Suppose $X$ is a parameter of $\mathcal{D}[pr,z_1]$. Then $X+(z_1-z_2)$ is one of $\mathcal{D}[pr,z_2]$. The second part follows from the change of variables $X \mapsto X+(z_1-z_2)$.
\end{proof}

\subsection{A vanishing cycle formula}
\label{secvanishingcycleformula}
In this section, we discuss how the refined Swan conductors determine whether a character has good reduction. Most of the details are almost identical to \cite[\S 2.4]{MR3552986}, which examine the case of mixed characteristic. 

We fix an admissible character $\chi \in \textrm{H}^1_{p^n}(\mathbb{K})$ of order $p^n$, which may be identified with a cover $\phi: Y \xrightarrow{} C$. Let $b_1, \ldots, b_r \in K$ be the branch points of $\phi$, hence contained in an open rigid disc $D \subset  C^{\an}$ (as $\chi$ is admissible), and let $\iota_{i,n}$ be the $n$-th conductor at $b_i$ (\S \ref{secbranchingdatum}). Let us also fix $r \in \mathbb{Q}_{\ge 0}$.

\begin{definition}
\label{defconductoratplace}
With the notation as above, let $z$ be a point inside $D$, and let $\overline{w}$ is a closed pont of $\overline{\mathcal{D}}[pr,z]$. We denote by $U(r,z, \overline{w}):=]\overline{w}[_{\mathcal{D}[pr, z]}$ the residue class of $\overline{w}$ on the affinoid $\mathcal{D}[pr, z]$. Set
\[ \mathfrak{C}_{\chi}(r,z, \overline{w}):= \sum_{i \in I(r, z, \overline{w})} \iota_{i,n} \in \mathbb{Z}_{\ge 0}   ,\]
where $I(r, z, \overline{w})=\{i \in \{1, 2, \ldots, r\}  \mid b_i \in \mathbb{B}(\chi) \cap U(r,z, \overline{w})\}$.
\end{definition}

\begin{proposition}
\label{propvanishingcycle}
With the notation introduced above, suppose $\overline{w} \neq \overline{\infty} \in \kappa$, considered as a closed point of $\overline{D}[pr,z]$, we have
\begin{equation}
\label{eqnvanishingcyclenoninfty}
    \ord_{\overline{w}} (\omega_{\chi}(r,z)) \ge - \mathfrak{C}_{\chi}(r,z, \overline{w}).
\end{equation}
Moreover, the equality holds if $\chi$ has good reduction.
When $\overline{w}=\overline{\infty}$, let $g_C$ be the genus of $C$. The following holds
\begin{equation}
\label{eqnvanishingcycleinfty}
    \ord_{\overline{\infty}}(\omega_{\chi}(r,z)) \ge  -\frac{2pg_C}{p-1} -\mathfrak{C}_{\chi}(r,z, \overline{\infty}).
\end{equation}
\end{proposition}

\begin{proof}
The proof is parallel to the mixed characteristic one in \cite[Proposition 2.14]{MR3552986}, where $\mathfrak{C}_{\chi}(r,z, \overline{w})$ (resp. $\mathfrak{C}_{\chi}(r,z, \overline{\infty})$) plays the role of $\lvert \mathbb{B}(\chi) \cap U(r,z,\overline{w}) \rvert$ (resp. $\lvert \mathbb{B}(\chi) \cap U(r,z,\overline{\infty}) \rvert$). See \cite[Proposition 3.8]{2020arXiv200203719D} for the explanation regarding the case $n=1$ and $\overline{w} \neq \overline{\infty}$. The case $n>1$ or $\overline{w}=\overline{\infty}$ is similar.
\end{proof}

\begin{remark}
We call (\ref{eqnvanishingcyclenoninfty}) and (\ref{eqnvanishingcycleinfty}) the local vanishing cycle formulas because they can be derived from Kato's vanishing cycle formula \cite{MR904945} (see the second version of \cite{MR3552986}).
\end{remark}

One easily derives result below from the proposition above.

\begin{corollary}
\label{corswgood}
With the notation as in the Proposition \ref{propvanishingcycle} and assuming $\overline{w} \neq \overline{\infty}$, we have the following inequality
\begin{equation}
    \sw_{\chi}(r,z,\overline{w}) \le \mathfrak{C}_{\chi}(r, z,\overline{w})-1.
\end{equation}
Moreover, if $\chi$ has good reduction and $\overline{w} \neq \overline{\infty}$ then the equality holds.
\end{corollary}

\begin{remark}
\label{remarkmindepth}
Corollary \ref{corswgood} and Proposition \ref{propvanishingcycle} indicate that, for $\chi$ to have good reduction, $\delta_{\chi}(r,z)$ is the smallest it could be, and its instantaneous rate of change at $r$ depends only on the sum of the $n$-th conductors of the branch points whose differences with $z$ have valuations greater than $r$.
\end{remark}

\begin{corollary}\label{corgood}
\begin{enumerate}
    \item \label{corgooditem1} Let $\chi \in \textrm{H}^1_{p^n}(\mathbb{K})$ be an admissible character of order $p^n$. Then
    \begin{equation}
        \label{eqngoodconductorswan}
        \iota_n:= \sum_{x \in \mathbb{B}(\chi)} \iota_{x,n}= \mathfrak{C}_{\chi}(0,0,\overline{0} ) \ge \sw_{\chi}(0,0,\overline{0})+1.
    \end{equation}
    \noindent Also, $\chi$ has good reduction if and only if $\delta_{\chi}(0)=0$ and the equality above holds. We call $\iota_n$ the \textit{conductor} of the character $\chi$.
    \item \label{corgooditem2} Suppose $\chi$ has good reduction with upper breaks $(m_1, \ldots, m_n)$. Let $\chi_i:=\chi^{p^{n-i}}$. If $1 \le i \le n$, then the conductor of $\chi_i$ is
    \begin{equation}
        \label{eqnupperbreakconductor}
        \iota_i:=\sum_{x \in \mathbb{B}(\chi)} \iota_{x,i}=m_i+1
    \end{equation}
    In particular, $\iota_i \ge p\iota_{i-1}-p+1$, and if $\iota_i \equiv 1 \pmod{p}$ then $\iota_i=p\iota_{i-1}-p+1$.
\end{enumerate}
\end{corollary}

\begin{proof}
Item (\ref{corgooditem1}) is immediate from Proposition \ref{propvanishingcycle}.

In the situation of part (\ref{corgooditem2}), as $\chi_i$ also has good reduction, it has conductor $\mathfrak{C}_{\chi_i}(0,0,\overline{0})=\sum_{x \in \mathbb{B}(\chi)} \iota_{x,i}$ by part (\ref{corgooditem1}). In addition, its reduction $\overline{\chi}_i \in \textrm{H}^1_{p^i}(\kappa)$ is a well defined character with upper ramification breaks $(m_1, \ldots, m_n)$. Hence, the Swan conductor of $\overline{\chi}_i$, which coincides with $\sw_{\chi_i}(0,0,\overline{0})$ as discussed in Remark \ref{remarkboundaryswan} (\ref{remarkboundaryswan1}), is $m_i$ \cite[Cor. 2 to Th. 1]{MR554237}. Compare with (\ref{eqngoodconductorswan}), we obtain (\ref{eqnupperbreakconductor}). The remaining assertions follow from Remark \ref{remarkupperbreaks}. 
\end{proof}

\begin{remark}
Corollary \ref{corgood} (\ref{corgooditem1}) indicates that, for a $\mathbb{Z}/p^n$-cover with {\'e}tale reduction to have good reduction, it is necessary and sufficient that the conductor of the reduction is equal to the sum of conductors of the generic branch points. It is thus equivalent to the different criterion (Proposition \ref{propdifferentcriterion}).
\end{remark}

\subsection{Characters of order \texorpdfstring{$p$}{p}}
Fix an admissible character $\chi \in \textrm{H}^1_p(\mathbb{K})$ and $z \in D$. We will now provide an algorithm to calculate $\delta_{\chi}(r,z)$ and $\omega_{\chi}(r,z)$ for $r \in \mathbb{Q}_{\ge 0}$. We may assume that $z=0$. The proofs in this section can easily translate to the case $z \neq 0$. The proposition below is the equal-characteristic analog of \cite[Proposition 5.16]{MR3194815}.

\begin{proposition}
\label{propcomputeordp}
Suppose $\chi=\mathfrak{K}_1(F)$, where $F \in  \mathbb{K}^{\times}/ \wp(\mathbb{K}^{\times})$, and $r \in \mathbb{Q}_{\ge 0}$. Without loss of generality, one may assume that $\nu_r(F)\le 0$ and $\chi$ is weakly unramified with respect to $\nu_r$.
 \begin{enumerate}
     \item \label{propcomputeorderpfirst} We have 
   \[  \delta_{\chi}(r)= -\frac{1}{p} \max\limits_{a} \nu_r(F+a^p-a), \]
   \noindent where $a$ ranges over all elements of $\mathbb{K}$.
     \item \label{propcomputeorderpsecond} The maximum of $\nu_r(F+a^p-a)$ in (\ref{propcomputeorderpfirst}) is achieved if and only if
     \[ g:=[F+a^p-a]_r \not\in \kappa^p_r. \]
     If this is the case, and $\delta_{\chi}(r)>0$, then
     \[ \omega_{\chi}(r) = dg. \]
     If, instead, $\delta_{\chi}(r)=0$, then $\overline{\chi}$ corresponds to the Artin-Schreier extension given by the equation $y^p-y=g$.
 \end{enumerate}
\end{proposition}

\begin{proof}
Part (\ref{propcomputeorderpfirst}) is immediate from the definition of the depth for abelian characters.

Let us now prove part (\ref{propcomputeorderpsecond}). Suppose $g \not \in \kappa_r^p$ and $\delta:=-\nu_r(F+a^p-a)/p \ge 0$. Then
\[ F+a^p-a=\pi^{-p\delta}u, \]
where $u$ is an element of $\mathbb{K}$ such that $\nu_r(u)=0$ and specializes to $g$. By \cite[Proposition 3.16]{2020arXiv200203719D}, we immediately obtain $\delta_{\chi}(r)=\delta$. If $\delta_{\chi}(r)>0$, then it follows from the same result that $\omega_{\chi}(r)=dg$. The case $\delta_{\chi}(r)=0$ is similar.
\end{proof}

\begin{remark}
\label{remarkboundaryorderp}
Sa{\"i}di shows in \cite[Proposition 2.3.1]{MR2377173} that a $\mathbb{Z}/p$-cover of a boundary is completely determined by its depth and its boundary Swan conductor. That result is motivated by Henrio's \cite[Corollaire 1.8]{2000math.....11098H} for the case $R$ is of mixed characteristic. However, just as the case of mixed characteristic examined in \cite[Chapter 5]{brewisthesis}, it is no longer true that a $\mathbb{Z}/p^n$-cover ($n>1$) of a boundary $\spec R[[X^{-1}]]\{X\}$ is determined by its Swan and boundary conductor. 
\end{remark}

\subsection{Computing the slope of particular order-\texorpdfstring{$p$}{p}-characters}
\label{secdetect}
This section is parallel to \cite[\S 5.5]{MR3194815}. Let $\chi \in \textrm{H}^1_p(\mathbb{K})$ be an admissible character of order $p$, hence giving rise to a branched cover of a disc $D[0,z]$, which we may associate with $\spec R[[X]]$. As before, we may assume that $z=0$. Let $m>1$ be a prime-to-$p$ integer. We assume further that the following conditions hold:

\begin{enumerate}[label=(\text{D}{{\arabic*}})]
    \item \label{firstconddetect} The branch locus of $\chi$ is contained in the closed disc $D[pr_0]$, for some $r_0>0$, and $X=0$ is one of the branch points,
    \item \label{secondconddetect} For all $r \in (0,pr_0]$, the left derivative of $\delta_{\chi}$ at $r$ is less than or equal to $m$, and
    \item \label{thirdconddetect}For all $r \in (0,pr_0]$, we have $\delta_{\chi}(r)>0$, and $\delta_{\chi}(pr_0)=\delta$. 
\end{enumerate}

\begin{proposition}
\label{propASclassdetect}
Suppose $\chi \in \textrm{H}^1_p(\mathbb{K})$ verifies \ref{firstconddetect} and \ref{thirdconddetect}. Then $\chi$ can be represented by the Artin-Schreier class of
\begin{equation}
    \label{eqnpropASclassdetect}
     F= \sum_{i=0}^{\infty} a_i X^{-i} \in \mathbb{K},
\end{equation}
with $a_i \in R$ and $\nu(a_i) \ge p(r_0i-\delta)$ for all $i$.
\end{proposition}

\begin{proof}
Condition \ref{firstconddetect} tells us that the function $F$ is holomorphic on the annulus $\big\{ a \in K \mid 0<\nu(A)<pr_0 \big\}$. Hence, as there is also no pole at infinity, it follows from the classical theory of analytic function that $F$ is represented by a power series of the form (\ref{eqnpropASclassdetect}) with
\begin{equation}
    \label{eqnannulusconds}
    \liminf\limits_{i \to \infty} \frac{\nu(a_i)}{i} \ge  pr_0 \text{, or } \liminf\limits_{i \to \infty} \frac{\nu(a_i)}{pr_0i} \ge 1.
\end{equation}
Let us now consider the place $pr_0$. Substitute $X$ by $X\pi^{-pr_0}=:X_{r_0}$ in (\ref{eqnpropASclassdetect}), we obtain
\[ F(X_{r_0})=\sum_{i=0}^{\infty} \frac{a_i}{\pi^{pr_0i}X_{r_0}^i}.\]
One can easily derive from (\ref{eqnannulusconds}) that there exists a positive integer $M$ such that, for all $i \ge M$, $\nu(a_i)-pr_0i> -p\delta=-p\delta(r_0)$. Thus, we may assume further, after replacing $F$ by another one in the same Artin-Schreier class, that $a_j=0$ for all $j < M$, $j \equiv 0 \pmod{p}$. It then follows from Proposition \ref{propcomputeordp} and \ref{thirdconddetect} that 
\[ \nu(F(X_{r_0}))=\nu\bigg(\sum_{i=0}^M  \frac{a_i}{\pi^{pr_0i}X_{r_0}^i}   \bigg)=\min_{i} (\nu(a_i)-pr_0i) = -p \delta. \]
Hence, we show $\nu(a_i) \ge p(r_0i-\delta)$, completing the proof.
\end{proof}

Let us consider the function $F$ in Proposition \ref{propASclassdetect}. We wish to find a polynomial $a$ in $T^{-1}$ such that $a^p-a$ approximates $F$ well enough to use Proposition \ref{propcomputeordp} simultaneously for all $r$ in an interval $(0,s] \cap \mathbb{Q}$ for some $0<s<pr_0$. We will then get explicit expressions for the slopes of $\delta_{\chi}$ on the interval $[0,pr_0]$, which will be useful later in this article (\S \ref{secminimal}).

For any $N \ge 1$, set
\[ a:= \sum_{j=1}^N b_j x^{-j} \in R[X].\]
\noindent Here we consider $b_j$ for the moment as indeterminates. Write
\[ F+a^p-a= \sum_{k=1}^{\infty} c_k x^{-k}, \]
\noindent where $c_k$ is a polynomial in $b_1, \ldots, b_{\min (k,m)}$. Note that $c_k=a_k \in R$ for any $k > pN$. 
\begin{lemma}
\label{lemmaapprox}
Assuming condition \ref{firstconddetect}, after replacing $K$ by some finite extension, there exists $b_1, \ldots, b_N \in R$ such that
\begin{enumerate}[label=\arabic*.]
    \item \label{lemmaapprox1st} $\nu(c_k) \ge p(r_0 k -\delta)$ for all $k$, and
    \item \label{lemmaapprox2nd} $c_{kp}=0$ for all $k \le N$.
\end{enumerate}
\end{lemma}

\begin{proof}
In order to achieve \ref{lemmaapprox2nd}, we solve the equations $c_{pN}=c_{p(N-1)}=\ldots=c_p=0$ inductively. They, from top to bottom, are
\begin{equation}
\begin{split}
a_{pN}+b_N^p & = 0 \\
&\vdots\\
a_{p(N-i)}+b_{N-i}^p-b_{p(N-i)} & =0\\
&\vdots\\
a_p+b_1^p-b_p & =0,
\end{split}
\end{equation}
where $b_{p(N-i)}=0$ for $p(N-i)>N$. Moreover, one can easily checks that $b_{N-i}$ has valuation at least $p(r_0(N-i)-\delta)$. Hence, $\nu(c_{N-i})$ is at least $p(r_0(N-i)-\delta)$, proving \ref{lemmaapprox1st}
\end{proof}
\begin{remark}
The proof above also shows that there are only finitely many solutions for the $b_j$'s and that they vary analytically as the $a_i$'s do.
\end{remark}

\begin{proposition}
\label{propdetect}

Assume conditions \ref{firstconddetect},\ref{secondconddetect}, and \ref{thirdconddetect} hold. Choose $s \in (0,pr_0) \cap \mathbb{Q}$ and $N \in \mathbb{N}$ such that
\begin{equation}
\label{eqndetectbound}
    pN \ge \frac{\delta}{r_0-s}.
\end{equation}
\noindent Let $b_1, \ldots, b_N$ be as in Lemma \ref{lemmaapprox}. Define $\lambda_m(\chi) \in [0,pr_0]$ by
\[ \lambda_m(\chi):=\max(\{r \in (0,pr_0] \mid \sw_{\chi}(r,\infty)>-m \} \cup \{0 \}) \]
Set
\[ \mu_m(\chi):=\max\bigg(\bigg\{\frac{\nu(c_m)-\nu(c_k)}{m-k} \mid 1 \le k <m \bigg\} \cup \{0\} \}\bigg). \]
Then the following hold.
\begin{enumerate}[label=\alph*.]
    \item \label{propdetectitem1} For all $r \in (0,s] \cap \mathbb{Q}$, we have
    \[ [F+a^p-a]_r \not\in \kappa^p_r. \]
    \noindent Therefore,
    \[\delta_{\chi}(r)=- \frac{\nu_r(F+a^p-a)}{p}, \text{ and } \sw_{\chi}(r,\infty)=-\ord_{\infty}[F+a^p-a]_r. \]
    \item \label{propdetectitem2} We have
    \[ \lambda_m(\chi) <s \iff \mu_m(\chi)<s.\]
    \item \label{propdetectitem3} If $\lambda_m(\chi)<s$, then $\lambda_m(\chi)=\mu_m(\chi)$.
\end{enumerate}
\end{proposition}

\begin{remark}
\label{remarkdetect}
Note that, if $\lambda_m(\chi) \neq r$, then Proposition \ref{propdetect} implies that $\lambda_m(\chi)$ is the largest value on the interval $(0,pr_0]$ where the piece-wise linear function $\delta_{\chi}$ has a ``kink''.
\end{remark}

\begin{proof}
Fix $r \in (0, s] \cap \mathbb{Q}$ and set $M:=\ord_{\infty}[F+a^p-a]_r$. Apply Lemma \ref{lemmaapprox}, we obtain the following inequality
\begin{equation}
\label{eqnproofpropdetect1}
    \nu_r(F+a^p-a)=\nu(c_M)-prM \ge p(M(r_0-r)-\delta) \ge p(M(r_0-s)-\delta). 
\end{equation}
\noindent On the other hand, as condition \ref{thirdconddetect} forces $\delta_{\chi}(r)>0$, Proposition \ref{propcomputeordp} shows that 
\begin{equation}
\label{eqnproofpropdetect2}
 \nu_r(F+a^p-a)<0.
\end{equation}
\noindent Hence, it must be true that $\delta/(r_0-s)>M$. It then follows from (\ref{eqnproofpropdetect1}),(\ref{eqnproofpropdetect2}), and the choice of $N$ that
\begin{equation}
\label{eqnconditionN}
    M< \frac{\delta}{r_0-s} \le Np.
\end{equation}
\noindent If $M$ was divisible by $p$ then (\ref{eqnconditionN}) and Lemma \ref{lemmaapprox} part \ref{lemmaapprox2nd} would show that $c_M=0$, which contradicts the definition of $M$. Therefore, $M$ is prime to $p$, and Part \ref{propdetectitem1} of the proposition follows from Proposition \ref{propcomputeordp} and Remark \ref{remarkboundaryswan}.

The rest of the proof is exactly the same as in \cite[Proposition 5.19]{MR3194815}.
\end{proof}

The proof of the following corollary is straightforward.

\begin{corollary}
\label{cordetectlowerbound}
In the notation of Proposition \ref{propdetect}, set
\begin{align*}
    \lambda_{m,l}(\chi): &= \max(\lambda_m(\chi),l), \\
    \mu_{m,l}(\chi): &= \max(\mu_m(\chi),l).
\end{align*}
\noindent Then Proposition \ref{propdetect} still holds when replacing $\lambda_{m}(\chi)$ (resp. $\mu_{m}(\chi)$) by $\lambda_{m,l}(\chi)$ (resp. $ \mu_{m,l}(\chi)$).
\end{corollary}

\subsection{Refined Swan conductor of \texorpdfstring{$\mathbb{Z}/p^n$}{Zpn}-extensions}
\label{sectionrefinedSwanASWleal}

In this section, we discuss a technique for calculating the refined Swan conductors of a fixed character $\chi:=\mathfrak{K}_n(\underline{a})$, where $\underline{a} \in W_n(\mathbb{K})$. We mostly follow the settings from \cite{MR3726102} and \cite{kato_leal_saito_2019}. Different perspectives can be found in \cite{MR1465067} \cite{MR2052867}.

Write $\mathbb{K}=\kappa((\pi))$ for some $\pi \in \mathcal{O}_{\mathbb{K}}$, where $\kappa$ is the residue of $\mathbb{K}$. Let $\{b_{\lambda}\}_{\lambda \in \Lambda}$ be a lift of a $p$-basis of $\kappa$ to $\mathcal{O}_{\mathbb{K}}$. Then $\Omega_{\mathcal{O}_{\mathbb{K}}}^1(\log)$ is the $\mathcal{O}_{\mathbb{K}}$-module with basis $\{db_{\lambda}, d\log \pi: \lambda \in \Lambda \}$.

We first define valuations on $\Omega^1_{\mathbb{K}}$ and $W_n({\mathbb{K}})$ as follows.

\begin{definition}
If $\omega \in \Omega^1_{\mathbb{K}}$ and $\underline{a}=(a^1, \ldots, a^n) \in W_n({\mathbb{K}})$, let
\[  \nu^{\log} \omega=\sup \big\{ i \mid \omega \in \pi^i \otimes_{\mathcal{O}_{\mathbb{K}}} \Omega_{\mathcal{O}_{\mathbb{K}}}^1(\log) \big\}, \]
and
\begin{equation}
\label{eqnwittvaluation}
     \nu(\underline{a})=\max_i \{ -p^{n-1-i} \nu(a^i)\}=\min_i \{p^{n-1-i} \nu(a^i) \}. 
\end{equation}
If $\omega=f(x)dx \in \Omega^1_\mathbb{K}$, then we set $[\omega]=[f(x)]dx$.
\end{definition}

These valuations define increasing filtrations of $\Omega^1_{\mathbb{K}}$ and $W_n({\mathbb{K}})$ by the subgroups 
\begin{equation}
    \begin{split}
        F_s \Omega^1_{\mathbb{K}}&=\{ \omega \in \Omega^1_{\mathbb{K}} \mid \nu^{\log} \omega \ge -s \} \text{, and } \\
        F_sW_n({\mathbb{K}})&= \{\underline{a} \in W_n({\mathbb{K}}) \mid \nu(\underline{a}) \ge -s \},
    \end{split}
\end{equation}

\noindent respectively, where $n \in \mathbb{Z}_{\ge 0}$.

\begin{remark}
\label{remarkswanfiltrationdefn}
Kato defined in \cite{MR991978} the filtration $F_s \textrm{H}^1_{p^n}(\mathbb{K})$ as the image of $F_s W_n({\mathbb{K}})$ under the ASW map (\ref{eqnaswmap}). Note that, for $\chi=\mathfrak{R}(\underline{a}) \in \textrm{H}^1_{p^n}(\mathbb{K})$, the Swan conductor $\sw (\chi)$ is defined to be the smallest $s$ such that $\underline{a} \in F_s \textrm{H}^1_{p^n}(\mathbb{K})$.
\end{remark}

We shall now define what it means for a Witt vector $\underline{a} \in W_n({\mathbb{K}})$ to be ``best'' (which is a generalization of ``reducible'' in \cite[Definition 3.14]{2020arXiv200203719D}).

\begin{definition}
\label{defnbest}
Let $\underline{a} \in W_n({\mathbb{K}})$, and $s$ be the smallest non-negative integer such that $\underline{a} \in F_s W_n({\mathbb{K}})$. We say that $a$ is \textit{best} if there is no $\underline{a}' \in W_n({\mathbb{K}})$ that maps to the same element as $\underline{a}$ in $\textrm{H}^1_{p^n}(\mathbb{K})$ such that $\underline{a}' \in F_{s'} W_n({\mathbb{K}})$ for some non-negative integer $s'<s$.
\end{definition}

When $\nu(\underline{a}) \ge 0$, $\underline{a}$ is clearly best. When $\nu(\underline{a})<0$, $\underline{a}$ is best if and only if there are no $\underline{a}',\underline{b} \in W_n({\mathbb{K}})$ satisfying
\[ \underline{a}=\underline{a}'+\wp(\underline{b}) \]
\noindent and $\nu(\underline{a})<\nu(\underline{a}')$.
When $n=1$, the following result, which immediately follows from Proposition \ref{propcomputeordp}, characterizes when $a \in \mathbb{K}$ is best.

\begin{proposition}
$a \in {\mathbb{K}} \setminus \mathcal{O}_{\mathbb{K}}$ is best if and only if $a=\pi w$, where $\nu(\pi)<0$ and $w \in \mathcal{O}_{\mathbb{K}}$ with $\overline{w} \in \kappa \setminus \kappa^p.$ 
\end{proposition}

\begin{definition}[{\cite[Definition 2.3]{MR3726102}}]
\label{defnrelevance}
We say that the $i$-th position of a Witt vector $\underline{a} \in W_n(\mathbb{K})$ is \textit{relevant} if $\nu(\underline{a})=p^{n-1-i}\nu(a^i)$. Let $j=\min \{i \mid \nu(a)=p^{n-1-i} \nu(a^i) \}$. We call $n-j+1$ the \textit{relevance length} of $\underline{a}$.
\end{definition}

The following result gives us an explicit algorithm to calculate the refined Swan conductors of a character using its best Witt vector representation.

\begin{theorem}{\cite[Theorem 2.7]{MR3726102}}
\label{thmbest}
Let $\underline{a} \in W_n(\mathbb{K})$. The following conditions are equivalent:
\begin{enumerate}[label=(\arabic*)]
    \item $\underline{a}$ is best.
    \item There exists some relevant position $i$ such that $a^i$ is best in the sense of length one.
    \item $\nu(\underline{a})=\nu^{\log}(d\underline{a})$, where $d\underline{a}:=\sum_i (a^i)^{p^{n-i}-1} da^i$.
\end{enumerate}
\end{theorem}

In particular, when $\underline{a}$ is best, the differential conductor of the corresponding character can be calculated as follows.

\begin{proposition}
\label{propcalculaterefinedswanwitt}
Suppose $\underline{a} = (a^1, \ldots, a^n) \in W_n(\mathbb{K})$ is defined with relevance length $n-j+1$, and let $\chi := \mathfrak{K}_n(\underline{a}) \in \mathrm{H}^1_{p^n}(\mathbb{K})$. Then, $\mathrm{sw}(\chi) = \nu(d\underline{a})$, and if $\mathrm{sw}(\chi) > 0$, we have
$$\dsw(\chi)=[d\underline{a}]=\sum_{i=j}^n [a^i]^{p^{n-i}-1} d[a^i].$$
\end{proposition}

\begin{proof}
In \cite[\S 2]{MR3726102}, the author defines the map $\rsw: F_d \textrm{H}_{p^n}^1(\mathbb{K}) \xrightarrow{} F_d \Omega^1_{\mathbb{K}}/ \allowbreak F_{\lfloor d/p \rfloor} \Omega^1_{\mathbb{K}}$ (which is the same as $\rsw$ in \cite{MR4045428}) sending a best $\underline{a}$, which represents $\chi$, to $d\underline{a}$. In \cite[Theorem 1.5]{MR4045428}, Kato and Saito show that $\rsw$ in Leal's coincides with $\rsw^{\ab}$ defined in \cite{MR948251} (mentioned in (\ref{eqnrswab})). See also \cite[Lemma 3.7]{MR991978} and the discussion following Proposition 6.8 of the same paper. In particular, in the notation above, we have
\[ \rsw^{\ab}(\chi)=\pi^{-\nu(d\underline{a})} \otimes [d\underline{a}].  \]
The rest then follows immediately from (\ref{eqnrswab}). 
\end{proof}

\begin{example}
\label{excalculatemainex}
In this example, we calculate the refined Swan conductors of some restrictions of the deformation in Example \ref{exmain}. Recall the $\chi$ is defined by the ASW equation
\begin{equation}
\label{order4deformationrepeat}
    \wp(Y_1,Y_2)=\bigg(\frac{1}{X^2(X-t^8)}, \frac{1}{X^3(X-t^8)^2(X-t^2)^2(X-t^2(1+t^2))^2} \bigg).
\end{equation}
Over the subdisc $\mathcal{D}[2]$ associated with $\spec R[[t^{-8}X]]$, substitute $t^{-8}X$ by $X_4$ in (\ref{order4deformationrepeat}), we obtain
\begin{equation*}
    \bigg(\frac{1}{t^{24}X_4^2(X_4-1)}, \frac{1}{t^{48}X_4^3(X_4-1)^2(t^6X_4-1)^2(t^6X_4-1(1+t^2))^2} \bigg).
\end{equation*}
As $48=24 \cdot 2$ and
\[ \frac{1}{x^2(x-1)} \frac{d}{dx} \bigg( \frac{1}{x^2(x-1)} \bigg)+\frac{d}{dx} \bigg( \frac{1}{x^3(x-1)^2} \bigg) =\frac{dx}{x^3(x-1)^3} \neq 0, \]
it follows from Proposition \ref{propcalculaterefinedswanwitt} that $\delta_{\chi_2}(4)=24$ and
\[ \omega_{\chi_2}(4)=\frac{dx}{x^3(x-1)^3}, \]
\noindent which is not an exact differential form as in the $\mathbb{Z}/p$-covers case. Observe that $\mathcal{C}(\omega_{\chi_2}(4))=\omega_{\chi_1}(4)=\frac{dx}{x^2(x-1)^2}$, where $\chi_1=\chi^p$ and $\mathcal{C}$ is the Cartier operator of $\Omega^1_{k(x)}$. This phenomenon will be discussed in \S \ref{seccartierprediction}.

Iterate the above computation for $\mathcal{D}[1]$ (associated with $\spec R[[t^{-2}X]]$, we get $\delta_{\chi_2}(1)=9$ and $\omega_{\chi_2}(1)=\frac{dx}{x^6(x-1)^4}$, which is exact thus satisfies Theorem \ref{theoremCartierprediction} part \ref{theoremCartierpredictionpart2b}. 
\end{example}

\subsubsection{Conditions on the refined Swan conductors of cyclic covers}
\label{seccartierprediction}
In order to answer induction type questions like Proposition \ref{propmainonepointcover}, one would be interested in learning what the $n$-level can be when the $(n-1)$-level is known. The next theorem, which is the equal-characteristic analog of \cite[Theorem 1.2]{MR3167623}, will do exactly that for our situation.

Let $\chi_n \in \textrm{H}^1(K,\mathbb{Z}/p^n)$ be a radical character of order $p^n$, with $n \ge 1$. For $i=1, \ldots, n$, we set
$$ \chi_i:=\chi_n^{p^{n-i}}, \hspace{5mm} \delta_i:=\delta_{\chi_i}, \hspace{5mm} \omega_i:=\omega_{\chi_i}.  $$
The tupel $(\delta_i, \omega_i)_{i=1, \ldots, n}$ is called the \textit{ramification datum} associated to $\chi_n$. For $1 \le j \le n$, the pair $(\delta_j, \omega_j)$ is called $\chi_n$'s \textit{$j$th ramification datum}.
\begin{theorem}
\label{theoremCartierprediction}
Let $\chi_n$ and $(\delta_i,\omega_i)_{i=1,\ldots,n}$ be as above. Let $\mathcal{C}: \Omega^1_{\kappa} \xrightarrow{} \Omega^1_{\kappa}$ be the Cartier operator \cite[\S 3.2]{MR3167623}. Then for all $i=1, \ldots, n$ the following holds.
\begin{enumerate}[label=(\roman*)]
    \item \label{theoremCartierpredictionpart1} $\mathcal{C}(\omega_1)=0$
    \item \label{theoremCartierpredictionpart2} Suppose $i>1$. Then
    \[\delta_i \ge p \delta_{i-1}. \]
    Moreover, we have
    \begin{enumerate}
        \item \label{theoremCartierpredictionpart2a} $\delta_i=p\delta_{i-1} \Rightarrow \mathcal{C}(\omega_i)=\omega_{i-1}$,
        \item \label{theoremCartierpredictionpart2b} $\delta_i >p\delta_{i-1}  \Rightarrow \mathcal{C}(\omega_i)=0$.
    \end{enumerate}
\end{enumerate}
\end{theorem}

\begin{proof}
Part \ref{theoremCartierpredictionpart1} follows immediately from the exactness of $\omega_1$ asserted by Proposition \ref{propcomputeordp}. 

Suppose the length-$i$-Witt-vector $\underline{a}_i=(a^1, \ldots, a^i)$ that defines $\chi_i$ is best with relevance length $l$. One may assume, after an ASW operation, that $\underline{a}_{i-1}=(a^1, \ldots, a^{i-1})$ is also best. Suppose first that $l=1$. Then, it follows from the definition that $$\delta_{i}=-\nu(a^i)/p>-\max_{j<i}\{p^{i-1-j}\nu(a^j)\}/p = p\delta_{i-1}.$$ In addition, the differential form $\omega_i=[d\underline{a}_i]=d[a^i]$ is exact, proving \ref{theoremCartierpredictionpart2b}.

Let us now consider the case where the relevance length $l>1$. Thus, the relevance length of $\overline{a}_{i-1}$ is $l-1$. An easy computation shows
\[ \mathcal{C}(\omega_i)= \mathcal{C} \bigg( \sum_{m=i-l+1}^i  [a^m]^{p^{n-m}-1} d[a^m]   \bigg) =\sum_{m=i-l+1}^{i-1}  [a^m]^{p^{n-m-1}-1} d[a^m]=\omega_{i-1},   \]
confirming \ref{theoremCartierpredictionpart2a}.
\end{proof}

\begin{definition}
\label{defnextensiondegdata}
In the notation of Theorem \ref{theoremCartierprediction} \ref{theoremCartierpredictionpart2}, we say that $(\delta_i, \omega_i)$ \textit{extends} $(\delta_{i-1}, \omega_{i-1})$.
\end{definition}

\begin{definition}
We call the equation $\mathcal{C}(y)=w$ for a given $w \in \Omega^1_{\kappa}$ the \emph{Cartier operator equation}. Some solutions to this equation will be discussed in \S \ref{secsolutioncartierequation}.
\end{definition}

\begin{remark}
The above theorem can also be proved quite easily by adapting the computation from \cite{MR3167623} to the equal-characteristic case.
\end{remark}

The following result says that the inverse of Theorem \ref{theoremCartierprediction} also holds. That is the equal-characteristic version of \cite[Theorem 4.6]{MR3167623}. The theorem itself is not employed in this paper. However, it contains the key idea of the main theorem's proof.  

\begin{theorem}
\label{theoremCartiersuff}
Let $(\delta_i, \omega_i)_{i=1, \ldots, n}$ be a tuple satisfying Theorem \ref{theoremCartierprediction}. Then, after a finite extension of $K$, there exists a radical character $\chi$ of order $p^n$ on $\mathbb{K}$, such that $(\delta_i, \omega_i)_i$ is the ramification datum associated to $\chi$.
\end{theorem}

\begin{proof}
We do induction on $i$. Suppose first that we are given a dual $(\delta_1, \omega_1)$, where $\delta_1 \in \mathbb{Q}_{\ge 0}$ and $\omega_1 \in \Omega^1_{\kappa}$ such that $\mathcal{C}(\omega_1)=0$. By the previous discussion, we may assume without loss of generality that $\omega_1=dw$ for some $w \in \kappa \setminus \kappa^p$. Let $W$ be a lift of $w$ to $\mathbb{K}$. Then, it immediately follows from Proposition \ref{propcomputeordp} that the character $\chi_u \in \textrm{H}^1_p({\mathbb{K}})$ where $u:=\pi^{-p\delta_1}w$ verifies $\delta_{\chi_u}=\delta_1$ and $\omega_{\chi_u}=\omega_1$, confirming the base case.

We complete the induction process in two steps corresponding to the following two lemmas.

\begin{lemma}
\label{lemmaminimalextension}
Let $\chi_{n-1} \in \textrm{H}^1_{p^{n-1}}({\mathbb{K}})$ be a radical character of order $p^{n-1}$. Then there exists a character $\chi_{\min} \in \textrm{H}^1_{p^n}(\mathbb{K})$ extending $\chi_{n-1}$ and such that $\delta_{\chi_{\min}}=p\delta_{n-1}$.
\end{lemma}

\begin{proof}
 One may assume that $\chi_{n-1}$ is given by $\underline{f}=(f^1, \ldots, f^{n-1}) \in W_{n-1}({\mathbb{K}})$ where $\underline{f}$ is best (Definition \ref{defnbest}). It follows from Theorem \ref{thmbest} that there exists some relevant position $0 \le i \le n-2$ such that $f^i$ best in the sense of length one. Recall that it means $\nu(f)=p^{n-2-i}\nu(f^i)=-\delta_{n-1}$, and $f^i=\pi^{-\nu(f^i)} w$ where $\overline{w} \in \kappa \setminus \kappa^p$. Consider the character $\chi_{\min}$ defined by $\underline{f}'=(f^1, \ldots, f^{n-1},0) \in W_n({\mathbb{K}})$. It is immediate from the definition that $\chi_{\min}^p=\chi_{n-1}$, and $\nu(\underline{f}')=p^{n-i}\nu(f^i)$. Hence, $i$ is also a relevant position of $f'$. Therefore, it again follows from Theorem \ref{thmbest} that $f'$ is best. Thus, we have $\delta_{\chi_{\min}}=-p^{n-i-1}\nu(f^i)=p\delta_{n-1}$.
\end{proof}

\begin{lemma}
Let $\chi_{n-1} \in \textrm{H}^1_{p^{n-1}}({\mathbb{K}})$ be a radical character of order $p^{n-1}$ with $(n-1)$th ramification datum $(\delta_{n-1}, \omega_{n-1})$. Suppose $(\delta_n, \omega_n)$ extends $(\delta_{n-1}, \omega_{n-1})$ in the sense of Definition \ref{defnextensiondegdata}. Then there exists an extension $\chi_{n}$ of $\chi_{n-1}$ such that $\delta_{\chi_n}=\delta_n$ and $\omega_{\chi_n}=\omega_n$.
\end{lemma}

\begin{proof}
Let $\chi_{\min}$ be the extension of $\chi_{n-1}$ with $\delta_{\chi_{\min}}=p \delta_{n-1}$, which exists by Lemma \ref{lemmaminimalextension}. Set $\omega_{\min}:=\dsw_{\chi_{\min}}$. It follows from Theorem \ref{theoremCartierprediction} \ref{theoremCartierpredictionpart2a} that $\mathcal{C}(\omega_{\min})=\omega_{n-1}$. If $\delta_n=p\delta_{n-1}$ and $\omega_n=\omega_{\min}$, then we are done. Otherwise, set $0 \neq \eta:=\omega_{n}-\omega_{\min}$. Since $\mathcal{C}(\eta)=\omega_{n-1} -\omega_{n-1}=0$, the differential form $\eta$ is exact. We thus may write $\eta=du$ for some $u \in \kappa$. Just as in the base case, one shows that the character $\psi \in \textrm{H}^1_p({\mathbb{K}})$ defined by $\pi^{-p\delta_n} U$, where $U$ is a lift of $u$ to $\mathbb{K}$, has depth $\delta_{\psi}=\delta_n$ and differential Swan $\dsw_{\psi}=\eta$. Set $\chi_n:=\chi_{\min} \cdot \psi$. Lemma \ref{lemmacombination} asserts that $\delta_{\chi_n}=\delta_n$ and $\dsw_{\chi_n}=\omega_n$, as desired.

A similar process applies for the case $\delta_n > p \delta_{n-1}$. As before, there exists $\psi \in \textrm{H}^1_p({\mathbb{K}})$ with $\delta_{\psi}=\delta$ and $\dsw(\psi)=\omega_n$. It, once more, follows from Lemma \ref{lemmacombination} that the character $\chi_n:=\chi_{\min} \cdot \psi$ satisfies the conditions we are seeking.
\end{proof}

That completes the induction process and hence the proof of Theorem \ref{theoremCartiersuff}. 
\end{proof}

\begin{remark}
\label{remarkcartierpredictiondeformingintowers}
In the situation of Proposition \ref{propmainonepointcover}, one can apply the technique from the above proof to form a cyclic    $p^n$-cover $Y_n \xrightarrow{} C$ with {\'e}tale reduction, extending $Y_{n-1} \xrightarrow{} C$, whose completion of the reduction at $x=0$ is birationally equivalent to $k[[y_n]]/k[[x]]$. However,  it is not true in general that the cover $Y_n \xrightarrow{} C$ has a good reduction as it would require the generic fiber to have the right ramification datum (Corollary \ref{corgood}), which is not easy to achieve using only the current method.
\end{remark}

\section{Hurwitz tree}
\label{secHurwitz}

\subsection{Hurwitz tree and the deformation problem}

Let $R=k[[t]]$ be a complete discrete valuation ring of equal-characteristics, where $k$ is an algebraically closed field of characteristic $p>0$. In this section, we first introduce the notion of Hurwitz tree for cyclic covers of a curve $C$ that is {\'e}tale outside an open disc $D \subset C^{\an}$. Then, we will describe how a $\mathbb{Z}/p^n$-cover gives rise to such a tree. Finally, we present an obstruction for the refined equal-characteristic deformation problem that is parallel to the obstruction for lifting given in \cite{MR2534115}. 

The definition below is identical to \cite[\S 3.1]{MR2254623}. We repeat it here for the convenience of the reader. 

\begin{definition}
\label{defdecorated}
A \emph{decorated tree} is given by the following data

\begin{itemize}
    \item a semi-stable curve $\overline{C}$ over $k$ of genus $0$,
    \item a family $(x_b)_{b\in B}$ of pairwise distinct smooth $k$-rational points of $\overline{C}$, indexed by a finite nonempty set $B$,
    \item a distinguished smooth $k$-rational point $x_0 \in \overline{C}$, distinct from any of the point $x_b$.
\end{itemize}

\noindent We require that $\overline{C}$ is stably marked by the points $((x_b)_{b \in B},x_0)$ in the sense of \cite{MR702953}.
\end{definition}

The \emph{combinatorial tree} underlying a decorated tree $\overline{C}$ is the graph $T=(V,E)$, defined as follows. The vertex set $V$ of $T$ is the set of irreducible components  of $\overline{C}$, together with a distinguished element $e_0$. We write $\overline{C}_v$ for the component corresponding to a vertex $v \neq v_0$ and $z_e$ for the singular point corresponding to an edge $e \neq e_0$. The singular point $z_e$ associated to an edge $e$ is adjacent to the vertices corresponding to the two components which intersect at $z_e$. The edge $e_0$ is adjacent to the root $v_0$ and the vertex $v$ corresponding to the (unique) component $\overline{C}_v$ containing the distinguished point $x_0$. At any point on an edge $e$, we say the direction away from the root the \textit{positive} one.

Note that, since $(\overline{C}, (x_b), x_0)$ is stably marked of genus $0$, the components $\overline{C}_v$ have genus zero, too, and the graph $T$ is a tree. Moreover, we have $\lvert B \rvert \ge 1$. For a vertex $v \in V$, we write $\overline{U}_v \subset \overline{C}_v$ for the complement in $\overline{C}_v$ of the set of singular and marked points.

\begin{definition}
\label{defnhurwitztree}
Let $M=(e_{j,i})_{1 \le j \le r}$ be an $r \times n$ matrix in $\Omega_{e_1, \ldots, e_n}$ (see Definition \ref{defnconductormatrices}). A \textit{$G=\mathbb{Z}/p^n$-Hurwitz tree} $\mathcal{T}$ of type $M$ is defined by the following data:
\begin{itemize}
    \item A decorated tree $\overline{C}=(\overline{C},(x_b),x_0)$ with underlying combinatorial tree $T=(V,E)$.
    \item For every $v \in V$, a rational $0 \le \delta_v=\delta_{\mathcal{T}}(v)$, called the \textit{depth} of $v$, 
    \item For each $v \in V$ such that $\delta_v>0$, a differential form $\omega_v=\omega_{\mathcal{T}}(v) \in \Omega^1_{\kappa}$, called the \textit{differential conductors} at $v$.
    \item For each $v \in V$, a group $G_v \subseteq G$, called the \textit{monodromy group} of $v$.
    \item For every $e \in E$, a positive rational number $\epsilon_e$, called the \textit{thickness} of $e.$
    \item For every $e \in E$, a positive integer $d_e$, called the \textit{slope} on $e$.
    \item For every $b_j \in B=\{b_1, \ldots, b_r \}$, the positive number $h_j=e_{j,n}$, called the \textit{conductor} at $b$.
    \item For $v_0$ with $\delta_{v_0}=0$, a \textbf{reduced} length-$n$-Witt-vector $\underline{f}:=(f^1, \ldots,f^n) \in W_n(\kappa)$ with only pole $0$, called the \textit{degeneration} of the tree. The rational $f^i$ is called the \textit{$i$-th degeneration} of $\mathcal{T}$. Define $$d:=\max\{ p^{n-l} \deg_{x^{-1}} (f^{l}) \mid l=1, \ldots, n\}.$$
\end{itemize}
\noindent These data are required to satisfy all of the following conditions.

\begin{enumerate}[label=(\text{H}{{\arabic*}})]
    \item \label{c1Hurwitz} Let $v \in V$. We have $\delta_v \neq 0$ if $v \neq v_0$.
    \item \label{c2Hurwitz} For each $v \in V\setminus \{v_0\}$, the differential form $\omega_v$ does not have zeros nor poles on $\overline{U}_v \subsetneq \overline{C}_v$.
    \item \label{c3Hurwitz} For every edge $e \in E \setminus \{e_0\}$, we have the equality 
    \[ -\ord_{z_e} \omega_{t(e)}-1 = \ord_{z_e}\omega_{s(e)}+1. \]
    \item \label{c4Hurwitz} For $v_0$, we have $d = \ord_{z_{e_0}}\omega_{t(e_0)}+1$.
    \item \label{c5Hurwitz} For every edge $e \in E$, we have
    \[ d_{e}= -\ord_{z_e} \omega_{t(e)}-1 \underset{s(e) \neq v_0}{\stackrel{\ref{c2Hurwitz}}{=}}\ord_{z_e} \omega_{s(e)}+1 .\]
    \item \label{c6Hurwitz} For every edge $e \in E$, we have
    \[\delta_{s(e)}+(p-1) \epsilon_e  d_{e} = \delta_{t(e)}. \]
    \item \label{c7Hurwitz} For $b \in B$, let $\overline{C}_v$ be the component containing the point $x_b$. Then the differential $\omega_v$ has a pole at $x_b$ of order $h_b$.
    \item \label{c8Hurwitz} For each $v$, we have
    \[ G_{v'} \subseteq G_v, \]
    for every successor vertex $v'$ of $v$. Moreover, we have
    \[ \sum_{v \to v'} [G_v:G_{v'}]>1, \]
    except if $v=v_0$ is the root, in which case there exists exactly one successor $v'$ and we have $G_v=G_{v'}=G$.
\end{enumerate}
\end{definition}

\noindent For each $v \in V \setminus \{v_0\}$, we call $(\delta_v, \omega_v)$ the \textit{degeneration type} of $v$. For each $e \in E$, we call $(\delta_{s(e)}, d_{s(e)})$ (resp. $(\delta_{t(e)}, d_{t(e)})$) the \textit{initial degeneration type} (resp. the \textit{final degeneration type}) of $e$. The positive integer $\mathfrak{C}:=d+1$ is called the \emph{conductor} of the Hurwitz tree. The rational $\delta:=\delta_{v_0}$ is the depth of $\mathcal{T}$. We define the \textit{height} of the tree to be the maximal positive direction edges from its root to its leaves. We call $\sum_{b \in B} h_b \cdot [B]$ the \textit{branching divisor} of $\mathcal{T}$.

\begin{remark}
\label{remarkslopesumconductors}
With the notation as above, fix $e \in E$ and let $\overline{C}_e \subseteq \overline{C}$ be the union of all components $\overline{C}_v$ corresponding to vertices $v$ which are separated from the root $v_0$ by the edge $e$. Then
\[ d_e= \sum_{\substack{b \in B \\ x_b \in \overline{C}_e}} h_b -1 >0.\]
In particular, we have $\mathfrak{C}=d+1=\sum_{b \in B} h_b$.
\end{remark}

\begin{remark}
\label{rmkextrainfhurwitz}
Depending on the study, one may like to add or remove some information from the Hurwitz tree in Definition \ref{defnhurwitztree}. For instance, the monodromy groups can be omitted where $G=\mathbb{Z}/p$ as in \cite{2020arXiv200203719D} and \cite{2000math.....11098H}, or where $G=\mathbb{Z}/p \rtimes \mathbb{Z}/m$ as in \cite{MR2254623}. In our situation, these data are useful for deducing the Hurwitz tree of a sub-cover as discussed in \S \ref{secquotienttree}. When $G=\mathbb{Z}/p \rtimes_{\psi} \mathbb{Z}/m$ where $\psi$ is non-trivial, one would like to add one extra piece of information coming from $\psi$ to the Hurwitz tree of the $\mathbb{Z}/p$-sub-cover like in \cite[Definition 3.2]{MR2254623}.
\end{remark}

\subsection{Hurwitz trees arise from cyclic covers}
\label{seccovertotree}

Fix a cyclic group $G:=\mathbb{Z}/p^n$. Let $R=k[[t]]$. Let $K$ denote the fraction field of $R$. As usual, we may associate with $\mathcal{D} \subset C^{\an}$ the spectrum of $R\{X\}$. Suppose we are given an admissible $G$-character $\chi \in \textrm{H}^1_{p^n}(\mathbb{K})$, which gives rise to an exponent-$p^n$-cover of the closed disc $\mathcal{D}$
\[ \Phi: \spec  R\{X\} \xrightarrow{} \spec R\{Z\}.\]
Let $h$ be the conductor, and $\delta$ be the depth of this cover. As discussed in \S \ref{secsemistablemodel}, there exists a semi-stable model of $C$ corresponding to $\mathcal{D}$'s interior and $\Phi$'s branch locus. The dual graph of its special fiber $\overline{C}$ forms a decorated tree $T=(V,E)$. For each vertex $v$ in $T$, we denote by $U_v \subset C^{\an}$ the affinoid subdomain with reduction $\overline{U}_v$, which can be thought of as a punctured disc. Following the exact procedure from \cite[\S 4.2]{2020arXiv200203719D}, one can construct a Hurwitz tree $\mathcal{T}$ from the dual graph and the refined Swan conductors. More precisely, the data of $\mathcal{T}$ are as in Table \ref{tab:covertotree}. See also \cite{MR2254623} \cite{MR2534115}.

\begin{table}[ht]
    \centering
\begin{tabular}{ |p{5.4cm}|p{9.6cm}|  }
\hline
Data on $\mathcal{T}$ & Degeneration data of the cover \\
\hline
\hline
The decorated tree & The branching geometry of $\Phi$ (Definition \ref{defngeometryofbranchpoints})\\
\hline
The depth of $v$ & The depth of the restriction of $\Phi$ to $U_v$.   \\
\hline
The differential conductor at $v$ & The differential conductor of the restriction of $\Phi$ to $U_v$ \\
\hline
The thickness of an edge $e$ & The thickness of the corresponding annulus divided by $p$ \\
\hline
The conductor at a leaf $b$ & The conductor of $\Phi$ at the branch point associated to $b$ \\
\hline
The reduced degeneration at $v_0$ & The reduced degeneration of $\Phi$ \\
\hline
The monodromy group at $v$ & The largest inertia group of the leaves succeeding $v$ \\
\hline
\end{tabular}
\vspace{3mm}
    \caption{Assigning a Hurwitz tree to an admissible cover}
    \label{tab:covertotree}
\end{table}

\begin{example}
\label{exmainhurwitz}
Let us calculate the Hurwitz tree for the $\mathbb{Z}/4$-cover from Example \ref{exmain}. Figure \ref{figlevel1tree} (resp. \ref{figlevel2tree}) is the one associated with the subcover $\chi_1$ (resp. $\chi_2$), which we call $\mathcal{T}_1$ (resp. $\mathcal{T}_2$). The monodromy group at each vertex of $\mathcal{T}_1$ is $\mathbb{Z}/2$. The reduction type at $v_1$ is $\big(12, \frac{dx}{x^2(x-1)^2} \big)$ and at $v_0$ is $\big(0, \frac{1}{x^3} \big)$. 

\tikzstyle{level 1}=[level distance=2.2cm, sibling distance=5cm]
\tikzstyle{level 2}=[level distance=1.2cm, sibling distance=2.3cm]
\tikzstyle{level 3}=[level distance=1.2cm, sibling distance=1cm]
\tikzstyle{level 4}=[level distance=1.2cm, sibling distance=1cm]
\tikzstyle{bag} = [text width=11.5em, text centered]
\tikzstyle{end} = [circle, minimum width=3pt,fill, inner sep=0pt]
\begin{figure}[ht]
\begin{subfigure}{.43\textwidth}
  \centering
  \begin{tikzpicture}[grow=up, sloped]
\node[end, label=left:{$v_0$}]{}
child{
        node[end, label=above:{$v_1$}]{}
    child {
                node[end, label=above:
                    {$2[t^8]$}] {}
                edge from parent
            }
    child {
                node[end, label=above:
                    {$2[0]$}] {}
                edge from parent
            }
    edge from parent
    node[above] {$e_0$}
    node[below] {$4$}
    };
\end{tikzpicture}
  \caption{$\mathcal{T}_1$}
  \label{figlevel1tree}
\end{subfigure}%
\tikzstyle{level 1}=[level distance=1.2cm, sibling distance=2cm]
\tikzstyle{level 2}=[level distance=1.2cm, sibling distance=3.0cm]
\tikzstyle{level 3}=[level distance=1.2cm, sibling distance=1.5cm]
\tikzstyle{level 4}=[level distance=1.2cm, sibling distance=2cm]
\tikzstyle{bag} = [text width=11.5em, text centered]
\tikzstyle{end} = [circle, minimum width=3pt,fill, inner sep=0pt]
\begin{subfigure}{.43\textwidth}
  \centering
\begin{tikzpicture}[grow=up, sloped]
\node[end, label=left:{$v_0$}]{}
child{
        node[end, label=above:{$v'_1$}]{}
    child {
        node(a)[end, label=above:{$v_2$}]{}        
            child {
                node(d)[end, label=above:
                    {$2[t^2(1+t^2)]$}] {}
                edge from parent
            }
            child {
                node(c)[end, label=above:
                    {$2[t^2]$}] {}
                edge from parent
            }
            edge from parent 
            node[above] {$e_{2}$}
            node[below]  {$1$}
    }
    child {
        node(a)[end, label=above:{$v_1$}]{}        
            child {
                node(d)[end, label=above:
                    {$3[t^8]$}] {}
                edge from parent
            }
            child {
                node(c)[end, label=above:
                    {$3[0]$}] {}
                edge from parent
            }
            edge from parent 
            node[above] {$e_{0,2}$}
            node[below]  {$3$}
    }
    edge from parent
    node[above] {$e_{0,1}$}
    node[below] {$1$}
    };
\end{tikzpicture}
  \caption{$\mathcal{T}_2$}
  \label{figlevel2tree}
\end{subfigure}%
\caption{The Hurwitz trees for Example \ref{exmain}}
\label{figexmainhurwitz}
\end{figure}
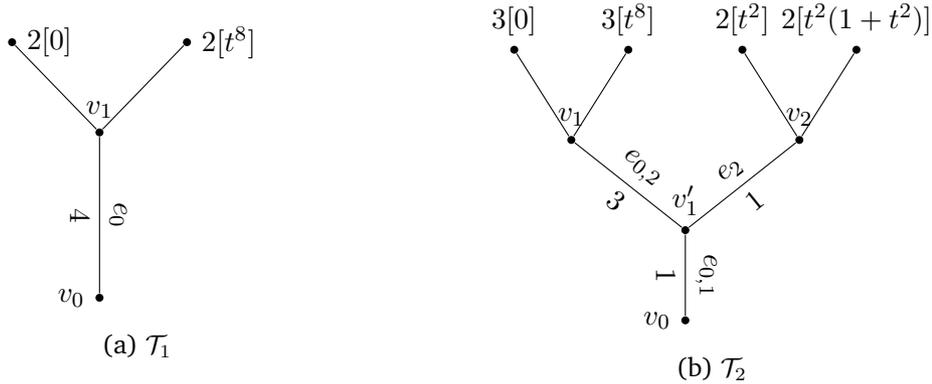

Regarding $\mathcal{T}_2$, it follows from the calculation in Example \ref{excalculatemainex} that the degeneration type at $v_1, v'_1, v_2$, and $v_0$ are $\Big(24, \frac{dx}{x^3(x-1)^3} \Big)$, $\Big(9, \frac{dx}{x^6(x-1)^4} \Big)$, $\Big(12, \allowbreak \frac{dx}{x^2(x-1)^2} \Big)$, and $\Big(0, \Big(\frac{1}{x^3}, \frac{1}{x^9} \Big) \Big)$, respectively. The monodromy groups at all vertices but $v_2$, and the two leaves $[t^2]$ and $[t^2(1+t^2)]$ are $\mathbb{Z}/4$. The monodromy groups at those vertices are $\mathbb{Z}/2$.
\end{example}

\begin{definition}
\label{defnextendhurwitz}
Suppose we are given a $\mathbb{Z}/p^{n-1}$-tree $\mathcal{T}_{n-1}$ and a $\mathbb{Z}/p^{n}$-tree $\mathcal{T}_{n}$. We say that $\mathcal{T}_n$ \textit{extends} $\mathcal{T}_{n-1}$, denoted
by $\mathcal{T}_{n-1} \prec \mathcal{T}_n$, if the followings hold.
\begin{enumerate}[label=(\arabic*)]
    \item The decorated tree of $\mathcal{T}_{n}$ is a refinement of one for $\mathcal{T}_{n-1}$.
    \item At each vertex $v$ of $\mathcal{T}_{n-1}$, the depth conductor and the differential conductor verify the conditions of Theorem \ref{theoremCartierprediction}.
    \item At each vertex $v$ (resp. each leaf $b$) of $\mathcal{T}_{n-1}$, if the monodromy group is $\mathbb{Z}/p^i$ ($i \le n-1$), then the monodromy group at the corresponding one on $\mathcal{T}_n$ is $\mathbb{Z}/p^{i+1}$.
    \item At each vertex $v$ (resp. each leaf $b$) of $\mathcal{T}_n$ that is not one of $\mathcal{T}_{n-1}$, the monodromy group is exactly $\mathbb{Z}/p$.
    \item Suppose $\delta_{\mathcal{T}_n}(v_0)=\delta_{\mathcal{T}_{n-1}}(v_0)=0$. Then the reduced degeneration of $\mathcal{T}_n$ is a length-$n$ Witt vector $(f^1, \ldots, f^{n-1}, f^n)$ such that $(f^1, \ldots, f^{n-1})$ is the reduced degeneration of $\mathcal{T}_{n-1}$.
\end{enumerate}
We say $\mathcal{T}_n$ extends a $\mathbb{Z}/p^i$-tree $\mathcal{T}_i$ if there exist a sequence of consecutive $n-i+1$ extending trees $\mathcal{T}_i \prec \mathcal{T}_{i+1} \prec \ldots \prec \mathcal{T}_{n-1} \prec \mathcal{T}_n$.
\end{definition}


\begin{definition}
\label{defsumcondsprec}
Suppose $r$ is a rational place on an edge $e$ of a Hurwitz tree $\mathcal{T}$. Set
\[ \mathfrak{C}_{\mathcal{T}}(e):= \sum_{b \in \mathbb{B}_{\mathcal{T}}(e)} h_b, \]
which is equal to $d_e+1$ by Remark \ref{remarkslopesumconductors}. We say a $\mathbb{Z}/p^n$-Hurwitz tree $\mathcal{T}$ is \textit{{\'e}tale} if $\delta_{\mathcal{T}}=0$, and is \textit{radical} otherwise. We define, for an edge $e$ and $r \in [s(e),t(e)]\cap \mathbb{Q}$, the depth of $\mathcal{T}$ at the place $r$ of $e$ as follows
\[ \delta_{\mathcal{T}}(r,e):=\delta_{\mathcal{T}}(s(e))+d_e(r-s(e))= \delta_{\mathcal{T}}(t(e))-d_e(t(e)-r). \]
\end{definition}

The existence of a Hurwitz tree $\mathcal{T}$ with particular conditions is necessary for a cyclic extension to deform non-trivially.

\begin{corollary}
\label{corminimalitydepth}
Suppose $\mathcal{T}$ is an {\'e}tale $\mathbb{Z}/p^n$-tree that extends $\mathcal{T}_{n-1}$, and $\chi \in \textrm{H}^1_{p^n}(\mathbb{K})$ is a character extends $\chi_{n-1}$. Suppose, moreover, that the tree associated with $\chi$ has the same shape as $\mathcal{T}$ and its branch points' conductors agree with that of $\mathcal{T}$. Then, on each edge $e$ of $\mathcal{T}$, 
\[ \delta_{\chi}(r,z_e) \ge \delta_{\mathcal{T}}(r,e)\]
\noindent holds for all $r \in [s(e), t(e)] \cap \mathbb{Q}_{\ge 0}$. The equality happens everywhere if $\chi$ is {\'e}tale (Definition \ref{defnetalegoodreduction}).
\end{corollary}

\begin{proof}
The result follows immediately of Remark \ref{remarkmindepth}.
\end{proof}

\begin{theorem}
\label{theoremgoodhurwitz}
The  $\mathbb{Z}/p^n$-cover $Y_n \xrightarrow{\phi_n} C$ verifying \ref{propmainonepointcover1} and \ref{propmainonepointcover2} in Proposition \ref{propmainonepointcover} is a deformation of $k[[y_n]]/k[[x]]$ over $R$ if it gives rise to an {\'e}tale tree $\mathcal{T}_n$ and its reduction is in the same ASW class of one that defines $k[[y_n]]/k[[x]]$.
\end{theorem}

\begin{proof}
The assumption about the degeneration of $\mathcal{T}_n$ implies that the cover has {\'e}tale reduction and the special fiber is birational equivalence to $k[[y_n]]/ \allowbreak k[[x]]$. In addition, suppose the $n$-th conductor of $k[[y_n]]/k[[x]]$ is $\iota_n$, which is equal to $d+1$ by definition. Then, by Remark \ref{remarkslopesumconductors}, the sum of the conductors of $\phi_n$'s branch points is also $\iota_n$. The theorem follows immediately from the previous corollary and Corollary \ref{corgood} (\ref{corgooditem1}).
\end{proof}

Suppose we are given a $\mathbb{Z}/p^n$-extension $\phi_n$ of $k[[x]]$ that is defined by $\underline{f}:=(f^1, \ldots, f^n) \in W_n(k[[x]])$, and a $\mathbb{Z}/p^{n-1}$-deformation $\chi_{n-1}$ of its sub-cover. The above theorem implies that, one can extend $\chi_{n-1}$ to a deformation of $\phi_n$ only if the Hurwitz tree it gives rise to can be extended. We call that obligation the \emph{Hurwitz tree obstruction} for the deforming in towers problem.

When $n=1$, it is known that every $\mathbb{Z}/p$-tree arises from some $\mathbb{Z}/p$-covers of the rigid disc.

\begin{theorem}[{\cite[Theorem 1.1]    {2020arXiv200203719D}}]
\label{thminversedegreep}
Suppose $\mathcal{T}$ is an {\'e}tale $\mathbb{Z}/p$-tree. Then there exists a $\mathbb{Z}/p$-covers whose associated tree coincides with $\mathcal{T}$.
\end{theorem}

\begin{remark}
\label{remarkAScase}
The proof of Theorem \ref{thminversedegreep} utilizes the formal patching technique for $G\cong \mathbb{Z}/p$-torsor. Recall that a $\mathbb{Z}/p$-cover of the boundary of a disc is determined by its depth and its boundary Swan conductor as mentioned in Remark \ref{remarkboundaryorderp}. Hence, one can ``glue'' the covers of the discs and the annuli which partition a unit disc, along their boundaries in a $G$-equivariant way to get a $\mathbb{Z}/p$-cover of the whole disc. That is the technique first used by Henrio in \cite{2000math.....11098H} to construct a $\mathbb{Z}/p$-lift from a Hurwitz tree in mixed characteristic. Unfortunately, we cannot apply the same technique for $\mathbb{Z}/p^n$-covers (where $n>1$), since it is no longer true that the depth and the boundary conductor determine the action on the boundary as discussed in the same remark. We can, however, generalize that approach to an arbitrary $\mathbb{Z}/p \rtimes_{\psi} \mathbb{Z}/m$-cover by adding an information coming from $\psi$ to the Hurwitz tree of the sub-$\mathbb{Z}/p$-cover as explained in Remark \ref{rmkextrainfhurwitz}. We thus expect an analog of Theorem \ref{thminversedegreep} for $\mathbb{Z}/p \rtimes_{\psi} \mathbb{Z}/m$-covers, which is also parallel to the main result of \cite{MR2254623}.
\end{remark}

\begin{definition}
Suppose $\mathcal{T}_n$ is a $\mathbb{Z}/p^n$-Hurwitz tree, and $v \neq v_0$ is one of its vertices with $m$ branches branching out $e_1, \ldots, e_m$, i.e., $v=s(e_1)=\ldots =s(e_m)$. Suppose, moreover, that the differential conductor at $v$ is of the form
\[ \omega_{\mathcal{T}_n}(v)= \frac{c_vdx}{\prod_{i=1}^m (x-a_i)^{h_i}} = \bigg(\sum_{i=1}^m \sum_{j=1}^{h_i} \frac{a_{i,j}}{(x-a_i)^j}\bigg) dx=: \sum_{i=1}^m \omega_i, \]
where $a_i=[z_{e_i}]_v$. We call $\omega_i$ the $e_i$-part of $\omega_{\mathcal{T}_n}(v)$.
\end{definition}

The below example shows how one can derive the relative positions of unknown branch points from the conditions for Hurwitz trees and their extensions.

\begin{example}
\label{exhurwitztreeobstruction}
Suppose $p=2$, $\phi_2 \in {\rm H}^1_4(k(x))$ is defined by $\big(\frac{1}{x^5}, 0\big)$, hence of branching datum $[6,11]$, and $\Phi_1$ is a deformation of $\phi_1:=\phi_2^2$ of type $[6] \rightarrow [2, 2, 2]^{\intercal}$ over $R:=k[[t]]$, whose branch points are $(0, t^2, t^4)$. Then the tree $\mathcal{T}_1$ associated to $\Phi_1$ has the following form
\tikzstyle{level 1}=[level distance=4cm, sibling distance=2cm]
\tikzstyle{level 2}=[level distance=4cm, sibling distance=1cm]
\tikzstyle{bag} = [text width=6em, text centered]
\tikzstyle{end} = [circle, minimum width=3pt,fill, inner sep=0pt]
\[ 
\begin{tikzpicture}[grow=right, sloped]
\node[bag] {$\Big(0,\frac{1}{x^{5}}\Big)$}
child{
        node[bag] {$\Big(\textcolor{black}{5},\textcolor{black}{\frac{dx}{x^{4}(x-1)^2}}\Big)$}
    child {
                node[end, label=right:
                    {$2[t^2]$}] {}
                edge from parent
            }
    child {
        node[bag] {$\Big(\textcolor{black}{8},\textcolor{black}{\frac{dx}{x^2(x-1)^2}}\Big)$}        
            child {
                node[end, label=right:
                    {$2 [t^{4}]$}] {}
                edge from parent
            }
            child {
                node[end, label=right:
                    {$2 [0]$}]{}
                edge from parent
            }
            edge from parent 
            node[above] {$e_1$}
            node[below]  {$1$}
    }
    edge from parent
    node[above] {$e_0$}
    node[below] {$1$}
    };
\end{tikzpicture}
\]       
The existence of such deformation is asserted by Theorem \ref{thminversedegreep}. Suppose in addition that there is a deformation $\Phi_2$ of $\phi_2$ extending $\Phi_1$ of type
\begin{equation}
\label{eqnexhurwitztreeobstruction}
M:=\begin{bmatrix}
   2 & 2& 2 &0 \\
   3 & 3 & 3 & 2 \\
\end{bmatrix}^\intercal
\end{equation} 
Denote by $\mathcal{T}_2$ the corresponding tree. Then it must be true that $\omega_{\mathcal{T}(2)}(s(e_0))$ is equal to either
\begin{center}
    \begin{enumerate*}[label=(\roman*)]
    \item \label{exhurwitztreeobstructioncase1} $\frac{adx}{x^6(x-1)^3(x-b)^2}$, or
    \item \label{exhurwitztreeobstructioncase2} $\frac{cdx}{x^6(x-1)^5}$, or
    \item \label{exhurwitztreeobstructioncase3} $\frac{edx}{x^8(x-1)^3}$, or 
    \item \label{exhurwitztreeobstructioncase4} $\frac{fdx}{x^6(x-1)^3}$,
\end{enumerate*}
\end{center}
for some $a, b, c, e, f\neq 0$ and $b \neq 1$. Either case, $\mathcal{C}(\omega_{\mathcal{T}(2)}(s(e_0))) \neq 0$. Theorem \ref{theoremCartierprediction} \ref{theoremCartierpredictionpart2a} then forces $\delta_{\mathcal{T}_2}(s(e_0))=10$ and, respectively,
\begin{center}
    \begin{enumerate*}
    \item $d_{e_0}=10, d_{e_1}=5$, or
    \item $d_{e_0}=10, d_{e_1}=5$, or
    \item $d_{e_0}=10$, or 
    \item $d_{e_1}=5$.
\end{enumerate*}
\end{center}
A similar reasoning shows that $\delta_{\mathcal{T}_2}(s(e_1))=16$. If $\Phi_2$ has no branch point in the interior of the annulus associated to $e_1$, then $d_{e_1}=6$. That contradicts the information we got for the cases \ref{exhurwitztreeobstructioncase1}, \ref{exhurwitztreeobstructioncase2} and \ref{exhurwitztreeobstructioncase4}. Finally, a straightforward calculation shows that the branch point of ramification breaks $(0,1)$, which we call $P$, has valuation $3$ and the tree $\mathcal{T}_2$ should be of the following form.
\tikzstyle{level 1}=[level distance=3.5cm, sibling distance=1.5cm]
\tikzstyle{level 2}=[level distance=3.5cm, sibling distance=1cm]
\tikzstyle{level 4}=[level distance=3cm, sibling distance=1cm]
\tikzstyle{bag} = [text width=6em, text centered]
\tikzstyle{end} = [circle, minimum width=3pt,fill, inner sep=0pt]
\[ 
\begin{tikzpicture}[grow=right, sloped]
\node[bag] {$\Big(0,\Big(\frac{1}{x^5},0 \Big)\Big)$}
child{
        node[bag] {$\Big(\textcolor{black}{10},\textcolor{black}{\frac{dx}{x^{8}(x-1)^3}}\Big)$}
    child {
                node[end, label=right:
                    {$3[t^2]$}] {}
                edge from parent
            }
    child {
        node[bag] {$\Big(\textcolor{black}{\frac{27}{2}},\textcolor{black}{\frac{dx}{x^6(x-1)^2}}\Big)$}   
        child {
                node[end, label=right:
                    {$2 P$}]{}
                edge from parent
            }
            child {
        node[bag] {$\Big(\textcolor{black}{16},\textcolor{black}{\frac{dx}{x^3(x-1)^3}}\Big)$}    
            child {
                node[end, label=right:
                    {$3 [t^{4}]$}] {}
                edge from parent
            }
            child {
                node[end, label=right:
                    {$3 [0]$}]{}
                edge from parent
            }
            edge from parent 
            node[above] {$e_{1,1}$}
            node[below]  {$.5$}
    }
            edge from parent 
            node[above] {$e_{1,0}$}
            node[below]  {$.5$}
    }
    edge from parent
    node[above] {$e_0$}
    node[below] {$1$}
    };
\end{tikzpicture}
\]     
\end{example}

\subsubsection{Deforming Artin-Schreier-Witt covers with no essential part}
Recall from \S \ref{secessentialcover} that every one-point-$\mathbb{Z}/p^n$-cover with an essential component can deform into a cyclic cover whose branching datum has no essential component. Using the Hurwitz tree technique, we demonstrate that further non-trivial deformations of these covers are not possible. This result plays a crucial role in determining the geometry of the moduli space of cyclic covers, as detailed in \cite[Theorem 4.16]{2023arXiv230614711D}.

\begin{proposition}
    A one-point-$\mathbb{Z}/p^n$-cover $\phi_n$ whose branching datum has no essential part cannot be non-trivially deformed.
\end{proposition}

\begin{proof}
Suppose $\phi_n$ has a branching datum $[e_1, \ldots, e_n]$, where $e_1 \neq 0$. Furthermore, suppose that $\phi_n$ deforms non-trivially to $\Phi_n$. Without loss of generality, one may assume that $\Phi_n$ has the type:
\begin{equation}
\label{eqndefnoessential}
M = [e_1, e_2, \ldots, e_n] \xrightarrow{}
\begin{bmatrix}
   e_{1,1} & e_{1,2} & \ldots & e_{1,n} \\
   e_{2,1} & e_{2,2} & \ldots & e_{2,n} \\
   \vdots & \vdots & \ddots & \vdots \\
   e_{r,1} & e_{r,2} & \ldots & e_{r,n}
\end{bmatrix} = N,
\end{equation}
where $e_{1,1} \neq 0$. Suppose $e_{j,1} \neq 0$ for some $1 < j \leq n$. Let $B_i$ be the branch point of $\Phi_n$ associated with the $i$-th row of $N$. Then there exists a non-trivial Artin-Schreier deformation $\Phi_1 := (\Phi_n)^{p^{n-1}}$ of type $[e_1] \xrightarrow{} [e_{1,1}, \ldots, e_{r,1}]^{\top}$. Let $\mathcal{T}_1$ be the associated Hurwitz tree. Then Theorem \ref{theoremCartierprediction} \ref{theoremCartierpredictionpart1} asserts the existence of exact differential forms at certain leaves of $\mathcal{T}_1$ of the form:
\begin{equation}
\omega = \frac{cdx}{\prod_{i \in I}(x-a_i)^{e_{i,1}}} \in \Omega^1_{\kappa},
\end{equation}
where $c \neq 0$, $a_i \neq a_j$ for $i \neq j$, and $I \subseteq \{1, \ldots, r\}$. However, as $e_1 < p$ due to the assumption of no essential part, we have $\sum_{i \in I} e_{i,1} \leq e_1 < p + |I|$. Therefore, $\omega$ does not exist, as shown in \cite[Proposition 6.4]{2020arXiv200203719D}.

By induction, we may rewrite \eqref{eqndefnoessential} as
\begin{equation}
\label{eqndefnoessentialfinal}
M = [e_1, e_2, \ldots, e_n] \xrightarrow{}
\begin{bmatrix}
   e_1 & e_2 & \ldots &  e_{n-1} &  e_{1,n} \\
   0 & 0 & \ldots & 0 & e_{2,n} \\
   \vdots & \vdots & \ddots & \vdots & \vdots \\
   0 & 0  & \ldots & 0 & e_{r,n}
\end{bmatrix} = N.
\end{equation}
Also, since $M$ lacks an essential part, we have $pe_{n-1} \ge e_n = \sum_{i=1}^r e_{i,n}$. Therefore, we have the condition
\begin{equation*}
    pe_{n-1}-p+1 \le e_{1,n} < pe_{n-1},
\end{equation*}
where the first inequality follows from the condition on the upper jump.

Suppose $e_{1, n} = pe_{n-1} - p + 1 \equiv 1 \pmod{p}$. Let $\mathcal{T}_n$ (respectively, $\mathcal{T}_{n-1}$) represent the Hurwitz tree associated with $\Phi_n$ (respectively, $\Phi_{n-1} := (\Phi_n)^p$). Then, for each edge $e$ between the root and the leaf associated with $B_1$ in $\mathcal{T}_n$, its slope, denoted as $d_e(\mathcal{T}_n)$, can be expressed as
\begin{equation*}
    d_e(\mathcal{T}_n)=\sum_{\substack{b \in B(\mathcal{T}_n), \hspace{1mm} x_b \in \overline{C}_e}} h_b -1 > e_{1,n}-1 \ge p(e_{n-1}-1),
\end{equation*}
given that $x_{B_1} \in \overline{C}_e$ and $\lvert \overline{C}_e \rvert > 1$. On the other hand, the slope of the associated one in $\mathcal{T}_{n-1}$ is precisely $e_{n-1}-1$. As a consequence, the depth at vertex $v_1$, preceding the leaf associated with $B_1$ in $\mathcal{T}_n$, is strictly greater than $p$ times the corresponding depth in $\mathcal{T}_{n-1}$. Therefore, in accordance with Theorem \ref{theoremCartierprediction}\ref{theoremCartierpredictionpart2b}, the differential form at vertex $v_1$ within $\mathcal{T}_n$ is exact. This, however, leads to a contradiction, considering that $e_{1, n} \equiv 1 \pmod{p}$.

Suppose now that $e_{1,n} = pe_{n-1} - p + a + 1$, where $0 < a < p$. In this case, the Hurwitz tree $\mathcal{T}_n$ is {\'e}tale of type $[e_{1,n}, \ldots, e_{r,n}]$ with all the differential forms being exact. However, one can demonstrate that such a tree does not exist using an argument similar to that employed for Artin-Schreier deformation.
\end{proof}

\subsection{Reduction type}
\label{secreductiontype}
Suppose we are given a cyclic character $\chi_{n} \in \textrm{H}^1_{p^{n}} (\mathbb{K})$ of order $p^n$ that is defined by a length-$n$-Witt-vector $\underline{F}_n:=(F^1, \ldots, F^n) \in W_n(\mathbb{K})$.  Suppose, moreover, that $\delta_{\chi_n}(0)=0$. Then the reduction $\overline{\chi}_n \in \textrm{H}^1_{p^n}(\kappa)$ of $\chi_n$ is well-defined. One may further assume that $\overline{\chi}_n$ is given by a length-$n$-Witt-vector 
\[ \underline{f}_n:= (f^1 ,f^2, \ldots, f^n) \in W_n(\kappa). \]
\noindent We call $\underline{f}_n$ a \textit{reduction type} of $\chi_n$. If $\underline{f}_n$ is reduced, we say it is the \textit{reduced reduction type} of $\chi_n$, and $f^i$ is its \textit{$i$-th degeneration}. It is clear that the character $\chi_{i}:= (\chi_n)^{p^{n-i}}  \in \textrm{H}^1_{p^i}(\kappa)$ has $\underline{f}_i=(f^1, \ldots, f^i)$ as a reduction type. 


\begin{proposition}
\label{propinverse}
Suppose we are given a character $\overline{\chi}_n \in \textrm{H}^1_{p^n}(\kappa)$ of order $p^n$ defined by a length $n$-Witt vector $\underline{f}_n=(f^1, f^2, \ldots, f^n) \in W_n(\kappa)$, a complete discrete valuation ring $R$ that is finite over $k[[t]]$, and a deformation $\chi_{n-i}$ of the sub-character $\overline{\chi}_{n-i}:= \overline{\chi}^{p^{i}}_n$ over $R$ given by
\[ \wp(\underline{Y}_{n-i})=(F^1, \ldots, F^{n-i}). \]
Suppose, moreover, that $\chi_n$ is a character given by 
\[ \wp(\underline{Y}_{n})=(F^1, \ldots, F^{n-1},F^n),\]
\noindent and gives rise to an {\'e}tale tree $\mathcal{T}_n$ that extends $\mathcal{T}_{n-i}$ (in the sense of Definition \ref{defnextendhurwitz}). Then $\chi_n$ extends $\chi_{n-i}$ if it has reduction type $\underline{f}_n$.
\end{proposition}

\begin{proof}
It is clear from the definition that $\chi_n$ extends $\chi_{n-i}$. Since the Hurwitz tree $\mathcal{T}_n$ to which it gives rise to is {\'e}tale, $\chi_n$ has good reduction by Theorem \ref{theoremgoodhurwitz}. Moreover, having reduction type $\underline{f}_n$ guarantees that the special fiber of $\chi_n$ is birationally equivalent to $\overline{\chi}_n$. It then follows from the definition that $\chi_n$ is a deformation of $\overline{\chi}_n$ extending $\chi_{n-i}$.
\end{proof}

\begin{remark}
\label{remarkhurwitzdetectgoodreduction}
The above proposition implies that one can extend a $\mathbb{Z}/p^{n-1}$-deformation by finding a rational function $F^n$ that gives the right degeneration data at the root of the associated tree $\mathcal{T}_{n}$, and the right branching data to the generic fiber. In general, it is not easy to calculate the refined Swan conductor of a character, let alone controlling it to obtain the wanted reduction. To get around it, we start our construction of $F^n$ from the disc corresponds to a ``final vertex'' (see \ref{step 1}) of the extending Hurwitz tree $\mathcal{T}_n$, where the degeneration is easy to manage. We then continuously modify $F^n$ until we get to the boundary of the unit disc $\spec R[[X]]$. More details will be given in \S \ref{sectinduction}. That strategy again resembles one used by Obus and Wewers in \cite{MR3194815} to prove the Oort conjecture.
\end{remark}

\subsection{The compatibility of the differential conductors}
\label{seccompatibility}

\begin{definition}
\label{defnconstantcoeffepart}
Suppose $\mathcal{T}$ is a $\mathbb{Z}/p^n$-Hurwitz-tree arises from $\chi \in \textrm{H}^1_{p^n}(\mathbb{K})$, which has conductor $\iota_n$ and good reduction. Suppose, moreover, that $v$ is a vertex (which can be the root) of the {\'e}tale tree $\mathcal{T}$ initiating $m$ edges $e_1, \ldots, e_m$, and has a differential conductor (or $n$-th degeneration type) of the form 
\[ \omega_{\mathcal{T}}(v)= \frac{c_vdx}{\prod_{i=1}^m (x-a_i)^{h_i}}=\sum_{i=1}^m \sum_{j=1}^{h_i} \frac{c_{v,j}dx}{(x-a_i)^j} \hspace{2mm} \bigg(\text{or }   \sum_{j=1}^l \frac{d_j}{x^j} \text{, where } d_l \neq 0, p \nmid l \bigg), \]
where $c_{v, h_i} \neq 0$ for $i=1, \ldots, m$ and $a_i=[z_{e_i}]_v$. Then we call $c_v$ (or $-ld_l$) the \textit{constant coefficient at $v$}, and $c_{v,h_i}$ the \textit{constant coefficient of its $e_i$ part}.
\end{definition}

Suppose $e$ is an edge in $\mathcal{T}$ and $r \in (s(e), t(e)) \cap \mathbb{Q}$. Then it follows from the discussion in \S \ref{secvanishingcycleformula} that the differential conductor at $r$ of $\chi$ is a finite sum of the form
\[ \omega_{\chi}(r)=\sum_{j \ge l} \frac{c_j dx}{(x-[z_e]_r)^j} \]
for some $l \in \mathbb{N}$ where $c_l \neq 0$.

\begin{definition}
\label{defnconstantpartonedge}
With the notation above, we say $\omega_{\mathcal{T}}$ has coefficient $c_l$ at $r$.
\end{definition}

The following result shows that the constant coefficients of the differential conductors along the Hurwitz tree constructed in \S \ref{seccovertotree} are ``compatible'' in some senses.

\begin{theorem}
\label{theoremcompatibilitydiff}
Suppose $\mathcal{T}$ is a tree that arises from a cyclic cover with good reduction. Let $e$ be an edge in $\mathcal{T}$. If $\delta_{\mathcal{T}}(s(e))>0$, then the followings hold.
\begin{enumerate}[label=(\arabic*)]
    \item Suppose $s(e)< r <t(e)$ is a rational place on $e$. Then the constant coefficient at $r$ is equal to that at $t(e)$.
    \item The constant coefficient at $t(e)$ is equal to that of the $e$-part at $s(e)$.
\end{enumerate}

Suppose $\mathcal{T}$ is an {\'e}tale tree, and the $n$-th level reduced degeneration at its root is a polynomial in $x^{-1}$ of degree $l$ with coefficient $-ld_l$. If $l< p\iota_{n-1}-p$, then we only need the differential conductors starting from $v_1$ to be compatible. Otherwise, we have $\iota_n=l+1$, and the differential conductor at $v_1$ has the same coefficient with that at $v_0$.
\end{theorem}

We first prove the result above for the situation of \S \ref{secdetect}. The rest of the proof will be given \S \ref{secpartition}, where we have extra tool to generalize Proposition \ref{propcompatibilitydetectsituation}.

\begin{proposition}
\label{propcompatibilitydetectsituation}
Suppose $\chi \in \textrm{H}^1_{p^{n}}(\mathbb{K})$ satisfies conditions \ref{firstconddetect}, \ref{secondconddetect}, and \ref{thirdconddetect} of \S \ref{secdetect} with respect to a rational place $r_0$. Suppose at $r_0$, the differential conductor $\omega_{\chi}(r_0)$ has constant coefficient $c$ and order of infinity $\iota-2>0$. Then there exists $r \in [0, r_0) \cap \mathbb{Q}$ such that, for all $s \in (r,r_0)\cap \mathbb{Q}$, we have
\[ \delta_{\chi}(s)= \delta_{\chi}(r_0)-(\iota-1) (r_0-s), \text{ and } \omega_{\chi}(s)=\frac{cdx}{x^{\iota}}, \]
In particular, for all $r' \in [0,r_0) \cap \mathbb{Q}$ such that $\delta_{\chi}(r')=\delta_{\chi}(r_0)-(\iota-1)(r_0-r')$, we have
\[ \omega_{\chi}(r')= \frac{cdx}{x^{\iota}} + \sum_{i=1}^{\iota-1}  \frac{c_idx}{x^i}=: \frac{cdx}{x^\iota} + \omega'_{\chi}(r'), \]
where $\omega'_{\chi}(r')$ is exact, i.e., $c_i=0$ for all $i \equiv 1 \pmod{p}$. 
\end{proposition}

\begin{proof}
We first assume that $G \cong \mathbb{Z}/p$. Set $Y:=X/t^{pr_0}$ and $\delta:=\delta_{\chi}(r_0)$. By Proposition \ref{propcomputeordp}, we may assume the restriction of $\chi$ to $D[pr_0]$ is represented by an Artin-Schreier class $f_{\chi}=t^{-p\delta}F(Y)$, where the reduction (modulo $t$) of $F$ is $\overline{F}:=f \not\in \kappa^p$. In addition, assumption \ref{firstconddetect} allows us to write
\[  df=\bigg(  \sum_{j \ge \iota} \frac{\overline{b}_j}{x^j}  \bigg) dx \in \Omega_{\kappa}^{\ast}, \]
where $\overline{b}_j \in k$ and $\overline{b}_{\iota}=c$. Set $I:=\{i \in \mathbb{N} \mid  \overline{b}_i \neq 0 \}$. It then follows from the discussion in \S \ref{secdetect} that
\begin{equation}
\label{eqninitialAS}
   f_{\chi}=\sum_{i=0}^{\infty}  a_i Y^{-i} =\frac{1}{t^{p\delta}} \Big( \sum_{j \ge \iota-1} \frac{d_j}{Y^j} + \sum_{l \not\in I} \frac{e_l}{Y^l}   \Big), 
\end{equation}
where $d_j \in R$, $\nu(d_j)=0$, and $\nu(e_l)>0$. In addition, using the assumption about the constant coefficient and Proposition \ref{propcomputeordp}, we learn that $c=(1-\iota)\overline{d}_{\iota-1}$. Let $Z:=Yt^{ps}$. We thus have
\begin{equation}
    \label{eqnlaterAs}
    f_{\chi}= \sum_{i=0}^{\infty}  a_it^{psi} Z^{-i} =\frac{1}{t^{p\delta}} \Big( \sum_{j \ge \iota-1} \frac{d_jt^{psj}}{Y^j} + \sum_{l \not\in I} \frac{e_lt^{psl} }{Y^l}   \Big).
\end{equation}
Therefore, there exists $v \in (0,r_0) \cap \mathbb{Q}$ such that, for all $w \in [v,r_0) \cap \mathbb{Q}$, the value  $\nu(d_{\iota-1})+p(r_0-w)(\iota-1)$ is strictly smaller than the valuations of other terms. The first part then immediately follows from Proposition \ref{propcomputeordp}. Let $r' \ge v$ be the largest kink of $\delta_{\chi}$ on $(0, r_0]$. Apply the process from \S \ref{secdetect} for some $r \in (0, r') \cap \mathbb{Q}$, one may replace $f_{\chi}$ of (\ref{eqninitialAS}) by another one in the same Artin-Schreier class, which may assume to have the form of (\ref{eqnlaterAs}), such that Proposition \ref{propcomputeordp} applies immediately for any rational place in $(0,v)$. In addition, that procedure does not change the  ${\iota-1}$ coefficient of $f_{\chi}$ and, for $s:=r_0-r$, $\nu(d_{\iota-1} t^{ps(\iota-1)}) \le \nu(e_l t^{psl})$ only when $l< \iota-1$ and $l \not\equiv 0 \pmod p$. The rest once more follows from an application of Proposition \ref{propcomputeordp} to the place $r'$.

We now consider the case $G \cong \mathbb{Z}/p^n$ ($n>1$). Suppose the length-$n$-Witt vector $\underline{f}=(f^1, \ldots, \allowbreak f^{n-1}, f^n)$ that represent $\chi$ is best at $r_0$, and $n-j+1$ is its relevance length (see Definition \ref{defnbest}). Recall from Proposition \ref{propcalculaterefinedswanwitt} that its differential Swan conductor at $r_0$ is 
\[ \dsw_{\chi}(r_0)= \sum_{i\ge j}^n [f^i]_{r_0}^{p^{n-i}-1}d[f^i]_{r_0}= \bigg[\sum_{i=1}^n(f^i)^{p^{n-i}-1} df^i\bigg]_{r_0}. \]
The first assertion then follows from the application of the base case's argument to $\sum_{i=1}^n(f^i)^{p^{n-i}-1} \allowbreak df^i$. Let us consider the case $\delta_{\chi}(r_0)> p\delta_{\chi^p}(r_0)$. Then $\underline{f}$'s relevance length (Definition \ref{defnrelevance}) is $1$ at $r_0$ and $\dsw_{\chi}(r_0)=d[f^n]_{r_0}$.
One then may prove the second assertion for this case using the same reasoning for the case $n=1$, where $f^n$ plays the role of $f_{\chi}$. Suppose now that $\delta_{\chi}(r_0)= p\delta_{\chi^p}(r_0)$, i.e., the relevance length of $\underline{f}$ is greater than $1$. Then, as before, there exists a smallest $r' \in [0, r_0) \cap \mathbb{Q}$ such that $\delta_{\chi}(s)$ is linear on $[r', r_0)$ with slope $\iota-1$. In addition, for all $s \in [r, r_0) \cap \mathbb{Q}$, it must be true that
\[ \omega_{\chi}(s)=\frac{c_sdx}{x^{\iota}} \text{, and } \omega_{\chi}(r')=\frac{c_{r'}dx}{x^{\iota}}+ \sum_{i=1}^{\iota-1} \frac{c_{r',i} dx}{x^i}. \]
Note that $c_s=c$ for $s \in (r,r_0) \cap \mathbb{Q}$. Let us consider two separate cases
\begin{itemize}
    \item If $\delta_{\chi}(s')>p\delta_{\chi^p}(s')$ for $s' \in (r', r_0)$, then $\delta_{\chi}(s)>p\delta_{\chi^p}(s)$ for all $s \in [r', r_0) \cap \mathbb{Q}$. Therefore, it must be true that $\underline{f}$ is best of relevance length $1$ on the interval $(r',r_0)$. We are in a situation similar to one where $\delta_{\chi}(r_0)> p\delta_{\chi^p}(r_0)$, where $s'$ plays the role of $r_0$.
    \item  If $\delta_{\chi}(s')=p\delta_{\chi^p}(s')$ for $s' \in (r, r_0)$, then, by Theorem \ref{theoremCartierprediction}, it must be true that $\delta_{\chi}(s)=p\delta_{\chi^p}(s)$ for all $s \in [r,r_0]$. Note that, if $\delta_{\chi^p}( \-- )$ concaves down at a place $s$ on $(0, r_0)$, then $\ord_{\infty} \omega_{\chi^p}(s)+1> -\ord_{0} \omega_{\chi^p}(s)-1$ (Proposition \ref{propdeltasw}), which implies $\omega_{\chi^p}(s)$ has a pole outside $0$, contradicting Corollary \ref{corswgood}. Therefore, if $\omega_{\chi^p}(s')=mdx/x^l$, then $\omega_{\chi^p}(s)=mdx/x^l$ for all $s \in (r',r_0) \cap \mathbb{Q}$ by induction, hence $\omega_{\chi}(s)=m^pdx/x^{pl+1}$. The rest easily follows.
\end{itemize}
That completes the proof of the proposition.
\end{proof}

\begin{definition}
\label{defncompatibility}
With the above notation, we say the conductor at $t(e)$ is \emph{compatible} with one at $s(e)$ if the constant coefficient at $t(e)$ is equal to that of the $e$ part at $s(e)$. Suppose $v \prec v'$, then we say $v'$ is compatible with $v$ if, given any pair of adjacent vertices $v \prec v_1 \prec v_2 \prec v'$, $v_1$ is compatible with $v_2$.
\end{definition}

\subsubsection{Partition of a tree by its edges and vertices}
\label{secsubtree}

Suppose we are given a Hurwitz tree $\mathcal{T}$ and $e$ is one of its edges. We define $\mathcal{T}(e)$ to be a sub-(non Hurwitz)-tree of $\mathcal{T}$ whose
\begin{itemize}
    \item edges, vertices and information on them coincide with those succeeding $s(e)$, and whose
    \item differential datum at $s(e)$ is equal to the $e$-part of $\omega_{\mathcal{T}}(s(e))$.
\end{itemize}
These tree will play a critical role in the extending of Hurwitz tree in \S \ref{secvanishingobstruction}, and the ``partition process'' \S \ref{secpartition}. More precisely, \S \ref{secproofmain} will construct a cover that gives rise to a whole tree inductively from its sub-tree. Note that $\mathcal{T}(e)$ verifies all of the Hurwitz tree's axioms but ones that involve the root and the trunk.

\subsubsection{Deriving a \texorpdfstring{$\mathbb{Z}/p^{n-1}$}{zpn-1}-tree from a \texorpdfstring{$\mathbb{Z}/p^{n}$}{zpn}-tree}
\label{secquotienttree}
Suppose $\chi_n$ is a $\mathbb{Z}/p^n$-character that gives rise to a Hurwitz tree $\mathcal{T}_n$. By following the below process, one can derive the tree $\mathcal{T}_{n-1}$ for level $n-1$ just from the level-$n$-tree, and the branching data of the generic fiber.
\begin{enumerate}[label=(\arabic*)]
    \item Erasing the $\mathbb{Z}/p$-leaves and the edges $e$ of $\mathcal{T}_n$ such that the monodromy group at $t(e)$ is $\mathbb{Z}/p$.
    \item Set the conductors at the leaves according to the $(n-1)$-branching-data of the generic fiber.
    \item If the monodromy group at a vertex $v$ is $G_{\mathcal{T}_n}(v) =\mathbb{Z}/p^m$, where $m>1$, then set $G_{\mathcal{T}_{n-1}}(v)= \mathbb{Z}/p^{m-1}$. 
    \item Starting from the root, assigning the degeneration data at the succeeding vertices inductively so that the differential conductors are compatible and verify Theorem \ref{theoremCartierprediction}.
\end{enumerate}

One thus can easily derive the degeneration data of the lower (than $(n-1)$) levels from the highest one by reiterating the above process. The readers can see that from the examples in \S \ref{secvanishingobstruction}.

\subsection{The vanishing of the differential Hurwitz tree obstruction}
\label{secvanishingobstruction}

In the below proposition, we discuss the vanishing of the Hurwitz tree obstruction for the refined deformation problem. In addition, those extending Hurwitz trees are specially designed to be used as the ``skeletons'' to build the lifts in \S \ref{secproofmain}. The proof is postponed to \S \ref{secconstructtree} as they are quite technical. However, we will provide some generalized enough examples and discuss some of their features to help the readers follow \S \ref{secproofmain}.

\begin{proposition}
\label{propextendtree}
Suppose $k[[y_n]]/k[[x]]$ is cyclic Galois of order $p^n$ and with conductor $\iota_n$. Suppose, moreover, that $Y_{n-1} \xrightarrow[]{\phi_{n-1}} C$ is an admissible deformation of the $\mathbb{Z}/p^{n-1}$-subextension $k[[y_{n-1}]]/\allowbreak k[[x]]$ over a finite extension $R$ of $k[[t]]$, and $\mathcal{T}_{n-1}$ is its associated Hurwitz tree. Then there exists a good $\mathbb{Z}/p^n$-Hurwitz tree $\mathcal{T}_{n}$ of conductor $\iota_n$ that extends $\mathcal{T}_{n-1}$.
\end{proposition}

Next, we will give three examples of extending a $\mathbb{Z}/p$-tree of conductor $\iota_1$ to certain $\mathbb{Z}/p^2$-trees, one for minimal second upper jump (i.e., $\iota_2=p\iota_{1}-p$) and two for non-minimal second upper jump (i.e., $\iota_2=p\iota_{1}-p+a_2$, where $p \nmid a_2$).

\begin{example}
\label{exminimaljump}
Suppose $\overline{\chi}$ is a character in $\textrm{H}^1_{7^2}(\kappa)$ that is given by the equation
\[ \wp(y_1,y_2)=\bigg(\frac{1}{x^{29}}, \frac{2}{x^{135}}+\frac{1}{x} \bigg). \]
\noindent One can think of $\overline{\chi}$ as a $\mathbb{Z}/{49}$-extension of the complete discrete valuation ring $k[[x]]$. By Theorem \ref{theoremcaljumpirred}, the extension $\overline{\chi}$ has upper jumps $u_1=29$ and $u_2= \max \{29 \cdot 7, 135\}=203$. That means the second jump is minimal. The $\mathbb{Z}/7$-subextension $\overline{\chi}_1$ of $\overline{\chi}$ is given by
\[ y_1^7-y_1=\frac{1}{x^{29}}. \]
\noindent Consider the $\mathbb{Z}/7$-Hurwitz tree $\mathcal{T}_1$ in Figure \ref{figextendtreeminsf1}, whose degeneration data are on Table \ref{tab:degenerationdata} .The number $a$ can be either $2+5\sqrt{6}$ or $2-5\sqrt{6}$. Note that these choices of $a$ make $dx/(x^{11}(x-1)^5(x-a)^3)$ exact. Note also that the constant coefficient $1/a^3$ of the differential conductor at $t(e_2)$ is necessary for it to be compatible with one at $t(e_1)$. Recall that one can construct from $\mathcal{T}_1$ a $\mathbb{Z}/7$-character $\chi_1$ of type $[5,6,5,3,4,7]^{\top}$ that deforms $\overline{\chi}_1$ (Theorem \ref{thminversedegreep}). The tree $\mathcal{T}_2^{\min}$ in Figure \ref{figextendtreeminsf2} with degeneration data in the third column of Table \ref{tab:degenerationdata} and with monodromy group $\mathbb{Z}/49$ at each vertex is a $\mathbb{Z}/49$-tree that extends $\mathcal{T}_1$ and has conductor $204$. One can show $\delta_{\mathcal{T}_2}(v)=7\delta_{\mathcal{T}_1}(v)$ and $\mathcal{C}(\omega_{\mathcal{T}^{\min}_2}(v))=\omega_{\mathcal{T}_1}(v)$ for each $v \in V_{\mathcal{T}^{\min}_2} \setminus \{v_0\}$. It follows from the data in Table \ref{tab:degenerationdata} that $\mathcal{T}_2^{\min}$ can represent a deformation of type
\[ \begin{bmatrix}
30 & 204 \\
\end{bmatrix} \xrightarrow{}  \begin{bmatrix}
   5 & 6 & 5 & 3 & 4 & 7 \\
   35 & 36 & 35 & 21 & 28 & 49 \\
\end{bmatrix}^\intercal. \] 
\tikzstyle{level 1}=[level distance=1.2cm, sibling distance=3.9cm]
\tikzstyle{level 2}=[level distance=1.2cm, sibling distance=2.1cm]
\tikzstyle{level 3}=[level distance=1.2cm, sibling distance=0.9cm]
\tikzstyle{level 4}=[level distance=1.2cm, sibling distance=1cm]
\tikzstyle{bag} = [text width=11.5em, text centered]
\tikzstyle{end} = [circle, minimum width=3pt,fill, inner sep=0pt]
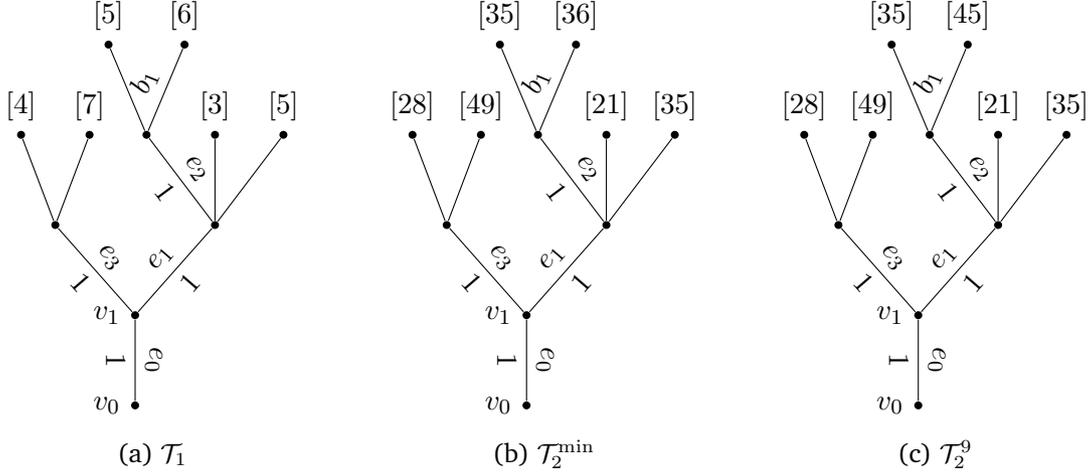
\begin{figure}[ht]
\begin{subfigure}{.32\textwidth}
  \centering
  \begin{tikzpicture}[grow=up, sloped]
\node[end, label=left:{$v_0$}]{}
child{
    node[end, label=left:{$v_1$}]{}
    child {
        node[end]{}        
            child {
                node(d)[end, label=above:
                    {$[5]$}] {}
                edge from parent
            }
            child {
                node[end, label=above:
                    {$[3]$}] {}
                edge from parent
            }
            child {
                node[end] {}
                child {
                    node(d)[end, label=above:
                    {$[6]$}] {}
                    edge from parent
                    node[above] {$b_1$}
                    }
                child {
                    node(d)[end,label=above:
                    {$[5]$}] {}
                    edge from parent
                    }
                edge from parent
                node[above] {$e_2$}
                node[below] {$1$}
            }
            edge from parent 
            node[above] {$e_1$}
            node[below]  {$1$}
    }
        child {
                node[end] {}
                child {
                    node(d)[end, label=above:
                    {$[7]$}] {}
                    edge from parent
                    }
                child {
                    node(d)[end, label=above:
                    {$[4]$}] {}
                    edge from parent
                    }
                edge from parent
                node[above] {$e_3$}
                node[below] {$1$}
            }
    edge from parent
    node[above] {$e_0$}
    node[below] {$1$}
    };
\end{tikzpicture}
  \caption{$\mathcal{T}_1$}
  \label{figextendtreeminsf1}
\end{subfigure}%
\begin{subfigure}{.32\textwidth}
  \centering
\begin{tikzpicture}[grow=up, sloped]
\node[end, label=left:{$v_0$}]{}
child{
    node[end, label=left:{$v_1$}]{}
    child {
        node[end]{}        
            child {
                node(d)[end, label=above:
                    {$[35]$}] {}
                edge from parent
            }
            child {
                node[end, label=above:
                    {$[21]$}] {}
                edge from parent
            }
            child {
                node[end] {}
                child {
                    node(d)[end, label=above:
                    {$[36]$}] {}
                    edge from parent
                    node[above] {$b_1$}
                    }
                child {
                    node(d)[end,label=above:
                    {$[35]$}] {}
                    edge from parent
                    }
                edge from parent
                node[above] {$e_2$}
                node[below] {$1$}
            }
            edge from parent 
            node[above] {$e_1$}
            node[below]  {$1$}
    }
        child {
                node[end] {}
                child {
                    node(d)[end, label=above:
                    {$[49]$}] {}
                    edge from parent
                    }
                child {
                    node(d)[end, label=above:
                    {$[28]$}] {}
                    edge from parent
                    }
                edge from parent
                node[above] {$e_3$}
                node[below] {$1$}
            }
    edge from parent
    node[above] {$e_0$}
    node[below] {$1$}
    };
\end{tikzpicture}
  \caption{$\mathcal{T}^{\min}_2$}
  \label{figextendtreeminsf2}
\end{subfigure}%
\begin{subfigure}{.32\textwidth}
  \centering
  \begin{tikzpicture}[grow=up, sloped]
\node[end, label=left:{$v_0$}]{}
child{
    node[end, label=left:{$v_1$}]{}
    child {
        node[end]{}        
            child {
                node(d)[end, label=above:
                    {$[35]$}] {}
                edge from parent
            }
            child {
                node[end, label=above:
                    {$[21]$}] {}
                edge from parent
            }
            child {
                node[end] {}
                child {
                    node(d)[end, label=above:
                    {$[45]$}] {}
                    edge from parent
                    node[above] {$b_1$}
                    }
                child {
                    node(d)[end,label=above:
                    {$[35]$}] {}
                    edge from parent
                    }
                edge from parent
                node[above] {$e_2$}
                node[below] {$1$}
            }
            edge from parent 
            node[above] {$e_1$}
            node[below]  {$1$}
    }
        child {
                node[end] {}
                child {
                    node(d)[end, label=above:
                    {$[49]$}] {}
                    edge from parent
                    }
                child {
                    node(d)[end, label=above:
                    {$[28]$}] {}
                    edge from parent
                    }
                edge from parent
                node[above] {$e_3$}
                node[below] {$1$}
            }
    edge from parent
    node[above] {$e_0$}
    node[below] {$1$}
    };
\end{tikzpicture}
  \caption{$\mathcal{T}^{9}_2$}
  \label{figextendtreeminsf3}
\end{subfigure}
\caption{Extending trees}
\label{figextendtreemin}
\end{figure}
\end{example}
\begin{table}[ht]
    \centering
\begin{tabular}{ |p{1.8cm}|p{3.7cm}|p{3.7cm}|p{3.7cm}|  }
\hline
Vertices & $\mathcal{T}_1$ & $\mathcal{T}^{\min}_2$ & $\mathcal{T}^9_2$ \\
\hline
\hline
$v_0$ & $\big( 0, \frac{1}{x^{29}} \big)$ & $\big( 0, \big(\frac{1}{x^{29}}, \frac{2}{x^{135}}+\frac{1}{x} \big)$ & $\big( 0, \big( \frac{1}{x^{29}}, \frac{3}{x^{212}}+\frac{1}{x^3} \big)$ \\
\hline
$v_1$ & $\big( 29, \frac{dx}{x^{19}(x-1)^{11}} \big)$ & $\big( 203, \frac{dx}{x^{127}(x-1)^{77}} \big)$ & $\big( 212, \frac{3dx}{x^{136}(x-1)^{77}} \big)$ \\
\hline
$t(e_3)$ & $\big( 39, \frac{dx}{x^{7}(x-1)^{4}} \big)$ & $\big( 280, \frac{dx}{x^{49}(x-1)^{28}} \big)$ & $\big( 288, \frac{3dx}{x^{49}(x-1)^{28}} \big)$ \\
\hline
$t(e_1)$ & $\big( 47, \frac{-dx}{x^{11}(x-1)^{5}(x-a)^3} \big)$ & $\big( 329, \frac{-dx}{x^{71}(x-1)^{35}(x-a)^{21}} \big)$ & $\big( 347, \frac{-3dx}{x^{80}(x-1)^{35}(x-a)^{21}} \big)$ \\
\hline
$t(e_2)$ & $\big( 57, \frac{-dx}{a^3x^{5}(x-1)^{6}} \big)$ & $\big( 399, \frac{-dx}{a^{21}x^{35}(x-1)^{36}} \big)$ & $\big( 426, \frac{-3dx}{a^{21}x^{35}(x-1)^{52}} \big)$ \\
\hline
\end{tabular}
    \caption{Degeneration data of the trees in Figure \ref{figextendtreemin}}
    \label{tab:degenerationdata}
\end{table}

\begin{example}
\label{exminimumplusone}
If we replace the Witt vector that defined $\overline{\chi}$ in the previous example by
\[ \wp(y_1, y_2)=\bigg(\frac{1}{x^{29}}, \frac{3}{x^{212}} +\frac{1}{x^3} \bigg), \]
then the upper jumps of the new character is $(30, 213)$. Suppose the deformation on the $\mathbb{Z}/7$-level is the same. Hence, its corresponding tree $\mathcal{T}_1$ is also identical to one in the previous example. Figure \ref{figextendtreeminsf3} is a tree $\mathcal{T}^9_2$ that extends $\mathcal{T}_1$ and has conductor $213$. Note that $\delta_{\mathcal{T}_2}(v)>7\delta_{\mathcal{T}_1}(v)$ and $\mathcal{C}(\omega_{\mathcal{T}^9_2}(v))=0$ for each $v \in V_{\mathcal{T}^9_2} \setminus \{v_0\}$.
\end{example}

Note that the extension tree in Example \ref{exextensionnonessential} is not unique! The following example illustrates an alternative construction where the extending tree $\mathcal{T}^{\ness}_2$ has no essential jumps from level $(n-1)$ to level $n$ (\emph{ness} stands for non-essential). The general construction will be given in Proposition \ref{propconstructreenodiff}.

\begin{example}
\label{exextensionnonessential}
Suppose the character $\overline{\chi}$ from the previous example corresponds to
\[ \wp(y_1,y_2)= \bigg( \frac{1}{x^{29}}, \frac{2}{x^{211} }+\frac{1}{x}\bigg) \bigg(\text{resp. } \bigg( \frac{1}{x^{29}}, \frac{3}{x^{212} }+\frac{1}{x^3}\bigg) \bigg), \]
and the $\mathbb{Z}/7$-deformation is the same. Then the diagram in Figure \ref{figextendtreenoessentialsf2} (resp. Figure \ref{figextendtreenoessentialsf3}) gives an extending tree of $\mathcal{T}_1$ with the right reduction type. Its degeneration data are in Table \ref{tab:degenerationdataexnonessential}. The monodromy groups at all vertices but the ends of the two leaves of conductors $7$ are isomorphic to $\mathbb{Z}/49$. The monodromy groups at those two vertices are $\mathbb{Z}/7$. Observe that, in both Figure \ref{figextendtreenoessentialsf2} and \ref{figextendtreenoessentialsf3}, we have $\delta_{\mathcal{T}_2}(v_1)=7 \delta_{\mathcal{T}_1}(v_1)$. 
\tikzstyle{level 1}=[level distance=1.2cm, sibling distance=5cm]
\tikzstyle{level 2}=[level distance=1.2cm, sibling distance=2.3cm]
\tikzstyle{level 3}=[level distance=1.2cm, sibling distance=1cm]
\tikzstyle{bag} = [text width=11.5em, text centered]
\tikzstyle{end} = [circle, minimum width=3pt,fill, inner sep=0pt]
\begin{figure}[ht]
\begin{subfigure}{.3\textwidth}
  \centering
  \begin{tikzpicture}[grow=up, sloped]
\node[end, label=below:{$v_0$}]{}
child{
node[end, label=left:{$v_1$}]{}
    child {
        node[end]{}        
            child {
                node(d)[end, label=above:
                    {$[5]$}] {}
                edge from parent
            }
            child {
                node[end, label=above:
                    {$[3]$}] {}
                edge from parent
            }
            child {
                node[end] {}
                child {
                    node(d)[end, label=above:
                    {$[6]$}] {}
                    edge from parent
                    node[above] {$b_1$}
                    }
                child {
                    node(d)[end,label=above:
                    {$[5]$}] {}
                    edge from parent
                    }
                edge from parent
                node[above] {$e_2$}
                node[below] {$1$}
            }
            edge from parent 
            node[above] {$e_1$}
            node[below]  {$1$}
    }
        child {
                node[end] {}
                child {
                    node(d)[end, label=above:
                    {$[7]$}] {}
                    edge from parent
                    }
                child {
                    node(d)[end, label=above:
                    {$[4]$}] {}
                    edge from parent
                    }
                edge from parent
                node[above] {$e_3$}
                node[below] {$1$}
            }
    edge from parent
    node[above] {$e_{0}$}
    node[below] {$1$}
    };
\end{tikzpicture}
  \caption{$\mathcal{T}_1$}
  \label{figextendtreenoessentialsf1}
\end{subfigure}%
\tikzstyle{level 1}=[level distance=0.6cm, sibling distance=4.1cm]
\tikzstyle{level 2}=[level distance=0.6cm, sibling distance=1cm]
\tikzstyle{level 3}=[level distance=1.2cm, sibling distance=2cm]
\tikzstyle{level 4}=[level distance=1.2cm, sibling distance=0.9cm]
\tikzstyle{bag} = [text width=11.5em, text centered]
\tikzstyle{end} = [circle, minimum width=3pt,fill, inner sep=0pt]
\begin{subfigure}{.32\textwidth}
  \centering
\begin{tikzpicture}[grow=up, sloped]
\node[end, label=below:{$v_0$}]{}
child{
node[end]{}
    child{
    node[end, label=below:{$[7]$}]{}
    edge from parent
    }
    child{
    node[end]{}
    child {
        node[end]{}        
            child {
                node(d)[end, label=above:
                    {$[35]$}] {}
                edge from parent
            }
            child {
                node[end, label=above:
                    {$[21]$}] {}
                edge from parent
            }
            child {
                node[end] {}
                child {
                    node(d)[end, label=above:
                    {$[36]$}] {}
                    edge from parent
                    node[above] {$b_1$}
                    }
                child {
                    node(d)[end,label=above:
                    {$[35]$}] {}
                    edge from parent
                    }
                edge from parent
                node[above] {$e_2$}
                node[below] {$1$}
            }
            edge from parent 
            node[above] {$e_1$}
            node[below]  {$1$}
    }
        child {
                node[end] {}
                child {
                    node(d)[end, label=above:
                    {$[43]$}] {}
                    edge from parent
                    }
                child {
                    node(d)[end, label=above:
                    {$[28]$}] {}
                    edge from parent
                    }
                edge from parent
                node[above] {$e_3$}
                node[below] {$1$}
            }
    edge from parent
    }
    child{
    node[end, label=below:{$[7]$}]{}
    edge from parent
    }
    edge from parent
    node[above] {$e_{0,1}$}
    node[below] {$\frac{3}{7}$}
    };
\end{tikzpicture}
  \caption{$\mathcal{T}^{8 \ness}_2$}
  \label{figextendtreenoessentialsf2}
\end{subfigure}%
\begin{subfigure}{.32\textwidth}
  \centering
  \begin{tikzpicture}[grow=up, sloped]
\node[end, label=below:{$v_0$}]{}
child{
node[end]{}
    child{
    node[end,label=below:{$[7]$}]{}
    edge from parent
    }
    child{
    node[end]{}
    child {
        node[end]{}        
            child {
                node(d)[end, label=above:
                    {$[35]$}] {}
                edge from parent
            }
            child {
                node[end, label=above:
                    {$[21]$}] {}
                edge from parent
            }
            child {
                node[end] {}
                child {
                    node(d)[end, label=above:
                    {$[37]$}] {}
                    edge from parent
                    node[above] {$b_1$}
                    }
                child {
                    node(d)[end,label=above:
                    {$[35]$}] {}
                    edge from parent
                    }
                edge from parent
                node[above] {$e_2$}
                node[below] {$1$}
            }
            edge from parent 
            node[above] {$e_1$}
            node[below]  {$1$}
    }
        child {
                node[end] {}
                child {
                    node(d)[end, label=above:
                    {$[43]$}] {}
                    edge from parent
                    }
                child {
                    node(d)[end, label=above:
                    {$[28]$}] {}
                    edge from parent
                    }
                edge from parent
                node[above] {$e_3$}
                node[below] {$1$}
            }
    edge from parent
    }
    child{
    node[end,label=below:{$[7]$}]{}
    edge from parent
    }
    edge from parent
    node[above] {$e_{0,3}$}
    node[below] {$\frac{5}{14}$}
    };
\end{tikzpicture}
  \caption{$\mathcal{T}^{9 \ness}_2$}
  \label{figextendtreenoessentialsf3}
\end{subfigure}
\caption{Extending trees with no essential jumps from level $1$ to $2$}
\label{figextendtreenoessential}
\end{figure}
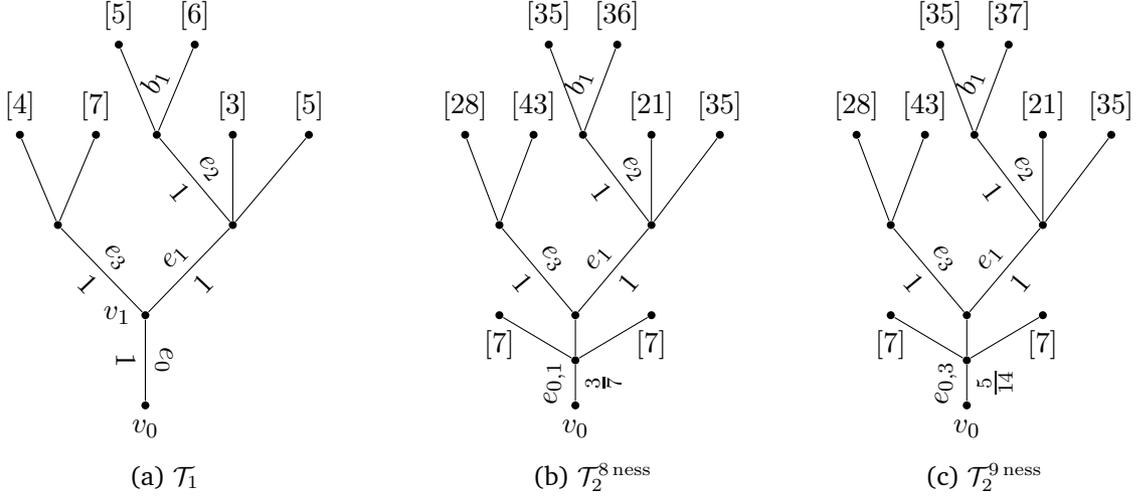
\begin{table}[ht]
    \centering
\begin{tabular}{ |p{1.8cm}|p{3.7cm}|p{3.7cm}|p{3.7cm}|  }
\hline
Vertices & $\mathcal{T}_1$ & $\mathcal{T}^{8\ness}_2$ & $\mathcal{T}^{9\ness}_2$ \\
\hline
\hline
$v_0$ & $\big( 0, \frac{1}{x^{29}} \big)$ & $\big( 0, \big(\frac{1}{x^{29}}, \frac{2}{x^{211}}+\frac{1}{x} \big)$ & $\big( 0, \big(\frac{1}{x^{29}}, \frac{3}{x^{212}}+\frac{1}{x^3} \big)$ \\
\hline
$t(e_{0,1})$ & $\big( \frac{89}{7}, \frac{dx}{x^{30}} \big)$ & $\big( \frac{633}{7}, \frac{2dx}{x^{198}(x-1)^{7}(x-2)^7} \big)$ & \\
\hline
$t(e_{0,3})$ & $\big( \frac{145}{14}, \frac{dx}{x^{30}} \big)$ & & $\big( \frac{530}{7}, \frac{3dx}{x^{199}(x-1)^{7}(x-3)^7} \big)$ \\
\hline
$v_1$ & $\big( 29, \frac{dx}{x^{19}(x-1)^{11}} \big)$ & $\big( 203, \frac{dx}{x^{127}(x-1)^{71}} \big)$ & $\big( 203, \frac{dx}{x^{128}(x-1)^{71}} \big)$ \\
\hline
$t(e_3)$ & $\big( 39, \frac{dx}{x^{7}(x-1)^{4}} \big)$ & $\big( 273, \frac{dx}{x^{43}(x-1)^{28}} \big)$ & $\big( 273, \frac{dx}{x^{43}(x-1)^{28}} \big)$ \\
\hline
$t(e_1)$ & $\big( 47, \frac{-dx}{x^{11}(x-1)^{5}(x-a)^3} \big)$ & $\big( 329, \frac{-dx}{x^{71}(x-1)^{35}(x-a)^{21}} \big)$ & $\big( 330, \frac{-dx}{x^{72}(x-1)^{35}(x-a)^{21}} \big)$ \\
\hline
$t(e_2)$ & $\big( 57, \frac{-dx}{a^3x^{5}(x-1)^{6}} \big)$ & $\big( 399, \frac{-dx}{a^{21}x^{35}(x-1)^{36}} \big)$ & $\big( 401, \frac{-dx}{a^{21}x^{35}(x-1)^{37}} \big)$ \\
\hline
\end{tabular}
    \caption{Degeneration data of the trees in Figure \ref{figextendtreenoessential}}
    \label{tab:degenerationdataexnonessential}
\end{table}
\end{example}

\begin{remark}
\label{remarkatmost2nonp}
Suppose $r$ is an arbitrary rational point on the tree $\mathcal{T}_{2}$. Then one can observe from the previous examples that there are at most two succeeding leaves whose conductors are not divisible by $p$. That also holds for all the trees $\mathcal{T}_n$ that we will construct to prove Proposition \ref{propextendtree}. This fact is essential for the proof of Lemma \ref{lemmacontroljunk}, which, in turn, is a key ingredient of Proposition \ref{propmainonepointcover}'s proof.
\end{remark}

\begin{remark}
\label{rmksubtree}

In the tree $\mathcal{T}_2^{\min}$ of Figure \ref{figextendtreeminsf2}, there is one leaf $b_1$ with conductor $l_{1,2}=7l_{1,1}-7+1$. The other leaves $b_i$ have conductor $l_{i,2}=7l_{i,1}$. Furthermore, at each vertex $v$ adjacent to the edges $e_2, e_1$, and $e_0$, $\delta_{\mathcal{T}_{2}}(v)=7\delta_{\mathcal{T}_{1}}(v)$. We say the tree $\mathcal{T}^{\min}_{2}$ extends $\mathcal{T}_{1}$ \textit{minimally}. Note that $\mathcal{T}^{\min}_{2}(e_1)$ (resp. $\mathcal{T}^{\min}_{2}(e_2)$) also extends $\mathcal{T}_{1}(e_1)$ (resp. $\mathcal{T}_{1}(e_2)$) minimally. In addition, the tree $\mathcal{T}^9_{2}(e_1)$ of Figure \ref{figextendtreeminsf3} has one leaf $b_1$ with conductor $l_{1,2}=7l_{1,1}-7+9+1$ and the other leaves $b_i$ has conductor $l_{i,2}=7l_{i,1}$. We say $\mathcal{T}^9_{2}$ extends $\mathcal{T}_{1}$ by \textit{$9$-additively}. Similarly, the tree $\mathcal{T}^{\min}_{2}(e_3)$ extends $\mathcal{T}_{1}(e_3)$ $(7-1)$-additively.

At the vertex $v_1$ of the $\mathbb{Z}/49$-trees in Figure \ref{figextendtreenoessentialsf2} and \ref{figextendtreenoessentialsf3}, all but two sub-trees starting from it extends the corresponding ones of $\mathcal{T}_1$ $(7-1)$-additively. The tree $\mathcal{T}^{9\ness}_{2}(e_1)$ extends $\mathcal{T}_{1}(e_1)$ $1$-additively, $\mathcal{T}^{8\ness}_{2}(e_3)$ and  $\mathcal{T}^{9\ness}_{2}(e_3)$) extend $\mathcal{T}_{1}(e_3)$ minimally. Formal definitions of the above conventions will be given in \S \ref{secconstructtree}.
\end{remark}

\section{Proof of Proposition \ref{propmainonepointcover}}
\label{secproofmain}

In this section, we give the final step (\ref{step4main}) of the proof of Proposition \ref{propmainonepointcover}, hence Theorem \ref{thminductionmain}, hence Theorem \ref{thmdeformtowers}. From now on, we denote $\mathbb{K}:=K(X)$, and $\kappa:=k(x)$ its residue field.

The inverse process of \S \ref{seccovertotree}, constructing a cover from a Hurwitz tree, was utilized by Henrio in \cite{2000math.....11098H} and by Bouw and Wewers in \cite{MR2254623} to solve the lifting problem for $\mathbb{Z}/p$-covers and $\mathbb{Z}/p \rtimes \mathbb{Z}/m$-covers, respectively. However, the main technique in these papers is the gluing method, which no longer works in our situation (as discussed in Remark \ref{remarkboundaryorderp} and Remark \ref{remarkAScase}). We instead generalize Obus and Wewers' approach in \cite{MR3194815}.

\subsection{Geometric set up}
\label{secgeomsetup}
In the rest of the paper, we set $C \cong \mathbb{P}_K^1=\Proj K[X,\allowbreak V]$. Recall that $\mathcal{D}$ is the closed unit disc inside $(\mathbb{A}^1_K)^{\an} \subset C^{\an}$ centered at $X=0$. We are given a character $\overline{\chi}_{n} \in \textrm{H}^1_{p^{n}}(\kappa)$ corresponding to a cyclic cover of $\mathcal{D}$ over $k$ of order exactly $p^{n}$, branching only at $0$ with upper ramification breaks $(m_1, \ldots, m_{n-1}, m_n)$. Set $(\iota_1, \ldots, \iota_n):=(m_1+1, \ldots, m_n+1)$. For $i=1, \ldots,n-1$, we set $\overline{\chi}_i:=  (\overline{\chi}_n)^{p^{n-i}} \in \textrm{H}^1_{p^i}(\kappa)$. Suppose there is a character $\chi_{n-1} \in \textrm{H}^1_{p^{n-1}}(\mathbb{K})$ deforming $\overline{\chi}_{n-1}$ (in the sense of Definition \ref{defnreductiondeformation}). In order to prove the refined local deformation problem, we must show that there exists a character $\chi_n \in \textrm{H}^1_{p^n}(\mathbb{K})$ that smoothly deforms $\overline{\chi}_n$ and verifies $\chi_{n}^p=\chi_{n-1}$.

We may assume that $\chi_{n-1}$ corresponds to an extension of $\mathbb{K}$ given by a length-$(n-1)$ Witt vector $\underline{G}_{n-1}=(G^1, \ldots, G^{n-1})$. Then any $\chi_n \in \textrm{H}^1_{p^n} (\mathbb{K})$ such that $\chi_n^p=\chi_{n-1}$ is given by a length-$n$ Witt vector $\underline{G}_n=(G^1, \ldots, G^{n-1}, \allowbreak G^n) \in W_n(\mathbb{K})$. After a finite extension of $K$, one may assume that $\underline{G}_{n-1}$ is best (Definition \ref{defnbest}). Recall that, by Proposition \ref{propextendtree}, one can construct a Hurwitz tree $\mathcal{T}_n$ with conductor $\iota_n= m_n+1$ that extends $\mathcal{T}_{n-1}$ (Definition \ref{defnextendhurwitz}). Suppose the type of such $\mathcal{T}_n$ is $(\iota_{n,1}, \ldots, \iota_{n,r} )$, where $\sum_{i} \iota_{n,i}=\iota_n$. We would then like $G^n$ to be of the form
\[ G^n=\sum_{j=1}^r \sum_{i=1}^{\iota_{n,j}-1} \frac{a_{j,i}}{(X-b_j)^i} \in \mathbb{K},  \]
\noindent where $a_{i,j} \in k((t))$, $b_i$'s are pairwise distinct $K$-points of the open unit disc $D$, and the geometry of the poles (Definition \ref{defngeometryofbranchpoints}) of $\underline{G}_{n}$ agrees with that of $\mathcal{T}_n$. We say that the polynomial $G^n$ \textit{gives rise} to the character $\chi_n$ (from $\chi_{n-1}$).

\subsection{The base case}

The case $n=1$ is trivial as one can always deform an Artin-Schreier cover trivially over $k[[t]]$ (using the same Artin-Schreier equation that defines the special fiber to represent the generic fiber). Recall that Theorem \ref{thminversedegreep} asserts that the existence of a deformation of Artin-Schreier covers with a given type is fully determined by the presence of $\mathbb{Z}/p$-differential-Hurwitz trees of the same type. One may utilize this fact to study the deformation ring $R_{\mathbb{Z}/p}$ (\S \ref{seclocalglobal}).

\subsection{Partition of the extending Hurwitz tree}
\label{secpartition}

In this section, we discuss a strategy to translate our situation to one that can be solved using the tools from \cite{MR3194815}. Suppose that $\chi_{n-1}$ branches at $s$ points $b_1, \ldots, b_s$ where $b_i \in \mathfrak{m} R$. We first show that its corresponding Witt vector $\underline{G}_{n-1}=(G^1, \ldots, G^{n-1})$ can be partitioned with respect to a rational place on the associated Hurwitz tree.

\begin{proposition}
\label{proppartitionWitt}
With the notation as in the geometric setup, suppose $A$ and $B$ partition the branch locus $\mathbb{B}(\chi_{n-1})=\{b_1, \ldots, b_s\}$. Then there exist two length-$(n-1)$ Witt vectors
\[ (G^1_{A}, \ldots, G^{n-1}_{A}), \hspace{2mm}  \text{ and } \hspace{2mm}   (G^1_{B}, \ldots, G^{n-1}_{B})\]
\noindent such that the former vector only has poles in $A$, the latter one only has poles in $B$, and $$(G^1, \ldots, \allowbreak G^{n-1})\allowbreak=(G^1_{A}, \ldots, G^{n-1}_{A})+(G^1_{B}, \ldots, G^{n-1}_{B}).$$
\end{proposition}

\begin{proof} We prove by induction on $n$. Let us first assume that $n=2$. We then can write $G^1$ as the sum of $G^1_A$ and $G^1_B$ simply by splitting its partial fraction decomposition. When $n>2$, we have
\begin{equation*}
    \begin{split}
        (G^1,G^2, \ldots, G^{n-1})-(G^1_A,0, \ldots, 0)&=(G^1_B, G^2_{1}, \ldots, G^{n-1}_{1})\\ &=(G^1_B, 0, \ldots, 0)+(0, G^2_{1}, \ldots, G^{n-1}_{1}).\\
    \end{split}
\end{equation*}
The second equation comes from the ``decomposition'' property of Witt vector \cite[\S 26, (32)]{MR2371763}. The second term is $\mathbb{V}(G^2_{1}, \ldots, G^{n-1}_{1})$, where $\mathbb{V}$ is the Verschiebung operation for Witt-vectors. By induction, we may write $(G^2_{1}, \ldots, G^{n-1}_{1})=(G^2_{1,A}, \ldots, G^{n-1}_{1,A})+(G^2_{1,B}, \ldots, G^{n-1}_{1,B})$, using the obvious notation. Therefore,
\begin{equation*}
    (G^1, \ldots, G^{n-1}) =(G^1_A,G^2_{1,A}, \ldots, G^{n-1}_{1,A})+ \allowbreak (G^1_B,G^2_{1,B}, \ldots, G^{n-1}_{1,B})
\end{equation*} as we want, completing the proof. \end{proof}

\begin{definition}
Let $s$ be a rational place on an edge $e$ of a Hurwitz tree $\mathcal{T}_{n-1}$, and $\overline{z} \in k$ or $\overline{z}=\infty$. We define $\underline{G}_{ n-1,s,\overline{z}}$ to be a partition of $\underline{G}_{n-1}$ whose poles specify to $\overline{z}$ on the canonical reduction of $\mathcal{D}[s,z_e]$. For any $s(e)<r<t(e)$, set $$\underline{G}_{n-1,r}:=\underline{G}_{ n-1,r,0}, \hspace{3mm} \text{ and } \hspace{3mm} \underline{G}_{n-1,s(e)}=\underline{G}_{n-1,t(e)}:=\underline{G}_{n-1,r}.$$ If $v=t(e)$ is a vertex, we set $\underline{G}_v:=\underline{G}_{t(e)}$. Define $$\chi_{n-1,s}:=\mathfrak{K}_{n-1}(\underline{G}_{n-1,s}), \hspace{3mm}  \text{ and }  \hspace{3mm} \chi_{n-1,s,\infty}:=\mathfrak{K}_{n-1}(\underline{G}_{n-1,s,\infty}).$$ 
\end{definition}

\begin{remark}
Observe that $\underline{G}_{n-1,s_1}=\underline{G}_{n-1,s_2}$ for all $s(e)<s_1 \le s_2 <t(e)$. In addition, if $s(e)<s<t(e)$ then $$\underline{G}_{n-1,t(e)}=\underline{G}_{n-1, s}+\underline{G}_{n-1, s, \infty}.$$
Finally, if $r=t(e)=s(e_1)=\ldots =s(e_l)$, then $\underline{G}_{n-1,t(e)}=\sum_{j=1}^l \underline{G}_{n-1, s(e_j)}$.
\end{remark}

The following result is simple but will turn out to be critical for our approach.

\begin{proposition}
\label{proppartition}
Suppose we are given a character $\chi_{n-1} \in \textrm{H}^1_{p^{n-1}}(\mathbb{K})$, which is defined by $ \underline{G}_{n-1}:=(G^1, \ldots, G^{n-1}) \in W_{n-1}(\mathbb{K})$, has good reduction, and gives rise to a $\mathbb{Z}/p^{n-1}$-Hurwitz tree $\mathcal{T}_{n-1}$. Then the followings hold.
\begin{enumerate}
    \item \label{proppartition1} Suppose $s(e)<s<t(e) \in \mathbb{Q}_{\ge 0}$, where $e$ is an edge of $\mathcal{T}_{n-1}$, and $\underline{G}_{n-1}=\underline{G}_{{n-1},s} + \underline{G}_{{n-1},s,\infty}$ as in Proposition \ref{proppartitionWitt}. Then:
\[ \delta_{\chi_{n-1}}(s)=\delta_{\chi_{n-1,s}}(s), \hspace{3mm} \text { and } \hspace{3mm} \omega_{\chi_{n-1}}(s)=\omega_{\chi_{n-1,s}}(s).\]
    \item \label{proppartition2} Suppose $s=t(e)=s(e_1)=\ldots=s(e_l)$. Then $$\delta_{\chi_{n-1}}(s)=\delta_{\chi_{n-1,s(e_i)}}(s), \hspace{3mm} \text{ and } \hspace{3mm} \omega_{\chi_{n-1}}(s)=\sum_{i=1}^l \omega_{\chi_{n-1,s(e_i)}}(s).$$ In particular, the constant coefficient (\S \ref{seccompatibility}) of the $e_i$-part of $\omega_{\chi_{n-1}}(s)$ is equal to the constant coefficient of $\omega_{\chi_{n-1,s(e_i)}}(s)$.
\end{enumerate}
\end{proposition}

\begin{proof}
Suppose $\delta_{\chi_{n-1}}(r) \neq \delta_{\chi_{n-1,s}}(r)$ at a rational place $s(e)<s \le r <t(e)$ in $e$. Since $\chi_{n-1}$ has good reduction, it follows from Proposition \ref{propvanishingcycle} and Remark \ref{remarkmindepth} that $\delta_{\chi_{n-1}}(r)$ is the minimal depth possible at $r$ with the given positions and conductors of the branch points whose valuations (with respect to $z_e$) are greater than $r$. Note that the restrictions of $\chi_{n-1,s}$ and $\chi_n$ to $D[pr,z_e]$ have the same branching data. Hence $\delta_{\chi_{n-1}}(r) < \delta_{\chi_{n-1,s}}(r)$. Therefore, it must be true that $\delta_{\chi_{n-1,s}}(r)=\delta_{\chi_{n-1,s,\infty}}(r)$, else $\delta_{\chi_{n-1}}(r)=\max(\delta_{\chi_{n-1,s}}(r),\delta_{\chi_{n-1,s,\infty}}(r))$ by Proposition \ref{lemmacombination}, which contradicts the previous sentence. On the other hand, by applying Proposition \ref{propvanishingcycle}, we obtain
\[ \ord_{0}(\omega_{\chi_{n-1,s,\infty}}(w)) \ge 0 \text{, and } \ord_{\infty}(\omega_{\chi_{n-1,s}}(w)) \ge 0, \]
for any $w \in \mathbb{Q}\cap [s(e),t(e)]$. That means, on the direction from $s(e)$ to $t(e)$ in $e$, $\delta_{\chi_{n-1,s}}$ is stricly increasing, and $\delta_{\chi_{n-1,s,\infty}}$ is strictly decreasing (recall Corollary \ref{corleftrightderivative} which says the differential determines the rate of change of the depth). Thus, it must be true that $\delta_{\chi_{n-1,s,\infty}}(w) \le \delta_{\chi_{n-1,s}}(w)$, and the equality can only happen when $w=s(e)$. That is a contradiction. The claim about the differential conductor then follows immediately from Lemma \ref{lemmacombination}.

The rest is straightforward, following a similar line of reasoning.
\end{proof}
\begin{remark}
\label{remarkpartition}
Suppose $\chi_{n-1}:=\mathfrak{K}_{n-1}(\underline{G}_{n-1})$ is a $\mathbb{Z}/p^{n-1}$-cover that has $\mathcal{T}_{n-1}$ as the corresponding tree, and $\mathcal{T}_n$ is a $\mathbb{Z}/p$-extension of $\mathcal{T}_{n-1}$. Consider a vertex $v=t(e)$ on $\mathcal{T}$ with $m$ branches $e_1, e_2, \ldots, e_m$ on the directions away from the root. As before we can partition $\underline{G}_{n-1}$ into a sum as below
\[ \underline{G}_{n-1}= \underline{G}_{n-1, v,\infty} +\sum_{i=1}^m \underline{G}_{n-1,v,i}. \]
Each $\underline{G}_{n-1,v,i}$ here is a length-$(n-1)$-Witt-vector whose roots are the leaves of $\mathcal{T}_{n-1}$ succeeding $e_i$. The poles of $\underline{G}_{n-1,v,\infty}$ are the leaves of $\mathcal{T}_{n-1}$ that are not ones of the $\mathcal{T}_{n-1}(e_i)$'s. Set $\chi_{n-1,v,i}:=\mathfrak{K}_{n-1}(\underline{G}_{n-1,v,i})\in \textrm{H}^1_{p^{n-1}}(\mathbb{K})$. For any rational place $r$ on $\mathcal{T}_{n}(e_i)$, we have $\delta_{\chi_{n-1,v,i}}(r)=\delta_{\chi_{n-1}}(r)$ by Proposition \ref{proppartition}. Hence, we obtain full information about the depth and the differential conductor of $\chi_{n-1,v,i}$. Moreover, all of the branch points of $\chi_{n-1,v,i}$ lie in the closed sub-disc $\mathcal{D}[pv,z_e]$. That is similar to the situation considered in \cite{MR3194815}, thus allows us to apply many technique from that paper. More precisely, we will extend each of the $\chi_{n-1,v,i}$'s at a time. Their sum then gives an extension of $\chi_{n-1,v}$ with tree $\mathcal{T}_{n}(v)$. 
\end{remark}

\begin{definition}
\label{defnrightbranchdata}
Suppose we are given a tree $\mathcal{T}_n$ that extends $\mathcal{T}_{n-1}$, and $r$ is a rational place on an edge $e$ of $\mathcal{T}_n$. When $s(e)<r\le t(e)$, we define $\mathcal{H}_r$ to be the collection of all $H \in \mathbb{K}$ that gives rise to $\mathfrak{K}_1(H) \in \textrm{H}^1_{p}(\mathbb{K})$ whose branching divisor is everywhere smaller than or equal to that of $\mathcal{T}_n(s(e))=\mathcal{T}_{n}(r)$. We define $\mathcal{H}_{s(e)}=\mathcal{H}_r$ for a rational $s(e)<r<t(e)$. Hence $\mathcal{H}_r=\mathcal{H}_s$ for all $s(e) \le s \le r \le t(e)$. As before, if $v=t(e)$ is a vertex, then we set $\mathcal{H}_v:=\mathcal{H}_{t(e)}$.
\end{definition}

\begin{corollary}
\label{corHreductiontype}
With the notation as in Definition \ref{defnrightbranchdata}, a polynomial $H \in \mathcal{H}_{s(e_0)}$ extends $\chi_{n-1}$ to a character $\chi_n$ that deforms $\overline{\chi}_{n}$ if and only if its reduction type coincides with $\overline{\chi}_{n}$'s reduced representation. 
\end{corollary}

\begin{proof}
``$\Leftarrow$'' Since $H \in \mathcal{H}_{s(e_0)}$ and $\underline{G}_{n-1}$ is best (as we assume in the geometric setup), it follows from Theorem \ref{theoremcaljumpirred} that $\mathfrak{C}(\chi_n, e_0, 0) \le \iota_n$. Moreover, as the different of the generic fiber is at least that of the special fiber by Proposition \ref{propdifferentcriterion}, the assumption on $\chi_n$'s reduction type implies that the conductor of the special fiber is $\iota_n$, hence enforcing the equality. That information, combining with Proposition \ref{propinverse}, proves ``$\Leftarrow$''.

The ``$\Rightarrow$'' direction is immediate from the definition.
\end{proof}

\begin{definition}
\label{defnrightbranchdepth}
Given a $\mathbb{Z}/p^n$-tree $\mathcal{T}_n$ that extends $\mathcal{T}_{n-1}$, and a vertex $v$ of $\mathcal{T}_n$. We define $\mathcal{G}_v \subseteq \mathcal{H}_v$ to be the collection of $G \in \mathbb{K}$ such that the branching datum of the character $\chi_{n,G} \in \textrm{H}^1_{p^n}(\mathbb{K})$ to which it gives rise fits into $\mathcal{T}_n(v)$, and such that 
\[\delta_{\chi_{n,G}}(v)=\delta_{\mathcal{T}_n}(v).\]
We denote by $\mathcal{G}^{\ast}_v \subseteq \mathcal{G}_v$ the collection of $G$ such that $\omega_{\chi_{n,G}}(v)$ is a nonzero multiple of $\omega_{\mathcal{T}_n}(v)$.
\end{definition}

\begin{definition}
Suppose we are given a rational $G \in \mathcal{G}_v$, and $v'$ is a vertex succeeding $v$ in $\mathcal{T}_{n-1}$. Then, as before, one may right $G=G_{v'}+G_{v',\infty}$. We set $G\lvert_{v'}:=G_{v'}$.
\end{definition}

\begin{proposition}
\label{propvertexpartition}
Suppose $\mathcal{T}_n$ is a good $\mathbb{Z}/p^n$-Hurwitz tree that extends $\mathcal{T}_{n-1}$, and $G \in \mathcal{G}_v$ for some vertex $v$ of $\mathcal{T}_n$. Let $\chi$ be the extension of $\chi_{n-1}$ by $G$. Then 
\[ \delta_{\chi}(r)=\delta_{\mathcal{T}_n}(r) \]
\noindent for every rational place $r>v$. In particular, at every vertex $v'$ succeeding $v$, we have $G\lvert_{v'} \in \mathcal{G}_{v'}$.
\end{proposition}

\begin{proof}
Suppose $\delta_{\chi}(r) \neq \delta_{\mathcal{T}_n}(r)$ for some $r>v$. Then, it follows from Corollary \ref{corminimalitydepth} that $\delta_{\chi}(r)>\delta_{\mathcal{T}_n}(r)$. On the other hand, the right derivative of $\delta_{\chi}$ at any rational place $r$ is $-\ord_{0}(\omega_{\chi}(r))-1$ by Corollary \ref{corleftrightderivative}, which is at most $\mathfrak{C}(\chi,r,\overline{0})-1$ by Proposition \ref{propvanishingcycle}. Hence, as $\delta_{\chi}(v)=\delta_{\mathcal{T}_n}(v)$, it must be true that $\delta_{\chi}(r) \le \delta_{\mathcal{T}_n}(r)$, which negates the assumption. The last assertion follows easily from Proposition \ref{proppartition} (\ref{proppartition2}).
\end{proof}

\begin{definition}
At a vertex $v=t(e)$ with $r$ direct successor edges $e_1, e_2, \ldots, e_l$, we define 
\[ \sum_{i=1}^r \mathcal{G}_{s(e_i)}:=\bigg\{\sum_{i=1}^l G_{s(e_i)} \mid G_{s(e_i)} \in \mathcal{G}_{s(e_i)} \bigg\}. \]
\end{definition}

\begin{corollary}
\label{corollarypartitionsum}
With the notation as in the above definition, we have  $\mathcal{G}_v=\mathcal{G}_{t(e)}=\sum_{i=1}^l \mathcal{G}_{s(e_i)}$.
\end{corollary}

\begin{proof}
``$\supseteq$'' is immediate from the definition and Lemma \ref{lemmacombination} (\ref{lemmacombination2}). Suppose $G \in \mathcal{G}_{t(e)}$. Then, as discussed in Remark \ref{remarkpartition}, we may write $G=\sum_{i=1}^r G_{s(e_i)}$, where $G_{s(e_i)} \in \mathcal{H}_{s(e_i)}$. Finally, as $\delta_{\chi_{s(e_i)}}(r)=\delta_{\chi}(r)=\delta_{\mathcal{T}_n}(r)$ for all $s(e)<r$ by Proposition \ref{propvertexpartition}, Proposition \ref{propdeltasw} asserts that $\delta_{\chi_{s(e_i)}}(s(e_i))=\delta_{\chi}(s(e_i))=\delta_{\mathcal{T}_n}(s(e_i))$. Therefore, $G_{s(e_i)} \in \mathcal{G}_{s(e_i)}$.
\end{proof}

Note that, we also have $\mathcal{H}_{t(e)} = \sum_{i=1}^l \mathcal{H}_{s(e_i)}$.

\subsubsection{Proof of Theorem \ref{theoremcompatibilitydiff}}

We now have the necessary tools to complete the proof of the compatibility theorem. Recall that the situation of \S \ref{secdetect} has been solved by Proposition \ref{propcompatibilitydetectsituation}. Let us now consider an arbitrary rational place $r \in (s(e),t(e)]$. We may again write $\underline{G}=\underline{G}_r+\underline{G}_{r, \infty}$. Applying Proposition \ref{proppartition}, with $\chi_r:=\mathfrak{K}(\underline{G}_r)$, one obtains
\[ \delta_{\chi}(r)=\delta_{\chi_r}(r) \text{, and } \omega_{\chi}(r)=\omega_{\chi_r}(r). \]
Hence, as the differential conductor of $\chi_{r}$ follow the compatibility rules, so is $\chi$.
\qed

\begin{remark}
\label{rmkparallelmixedcharacteristic}
Most of the results in \S \ref{secHurwitz} and \S \ref{secproofmain} until this point easily translate to the case where $R$ is of mixed characteristic. The most important among them is Proposition \ref{proppartition} (\ref{proppartition1}). The only problematic part is Proposition \ref{proppartition} (\ref{proppartition2}), hence so is Theorem \ref{theoremcompatibilitydiff}, since we not yet have an exact analog of Theorem \ref{thmbest} in mixed characteristic.
\end{remark}

\subsection{The induction process}
\label{sectinduction}
Following is the outline of the induction process to construct a character $\chi_n$ that deforms $\overline{\chi}_n$, extends $\chi_{n-1}$, and gives rise to $\mathcal{T}_n$. One may assume that the  length-$(n-1)$ Witt-vector that defines $\chi_{n-1}$ is best. We would like to find a rational function $G^n \in \mathbb{K}$ that extends $\chi_{n-1}$ to a deformation $\chi_n$ of $\overline{\chi}_n$. Below is the induction (ind) process.

\begin{enumerate}[label=Ind \arabic*]
\item \label{step 1} We start at a vertex $v=t(e)$ of $\mathcal{T}_n$ that only precedes leaves. We call $v$ a \textit{final vertex} of $\mathcal{T}_n$. In \S \ref{secfinalqcqs}, we consider the affinoid $\mathcal{G}_{t(e)}$. It is a collection of $G^n_{t(e)} \in \mathbb{K}$ that gives rise to a cover $\chi_{n,t(e)}$ with desired depth and (generic) branching datum at $t(e)$. We then show that the subset $\mathcal{G}^{\ast}_{t(e)} \subsetneq \mathcal{G}_{t(e)}$, whose elements give rise to the characters with the right differential conductor at $t(e)$, is non-empty. Recall that $\omega_{\mathcal{T}_n}(t(e))$ is either exact or $\mathcal{C}(\omega_{\mathcal{T}_n}(t(e)))=\omega_{\chi_{n-1}}(t(e))$ (Theorem \ref{theoremCartierprediction}). The case $\omega_{\mathcal{T}_n}(v)$ is exact is considered in \S \ref{secfinalexact}, the non-exact case is examined in \S \ref{secfinalnonexact}. 
\item \label{step 2} This step is carried out in \S \ref{seccontroledge} and \S \ref{secminimal}. Let $e$ be an edge of $\mathcal{T}_n$ that is the direct predecessor of a final vertex. Then $\underline{G}_{n-1,t(e)}=\underline{G}_{n-1,s(e)}$ and there are no leaves in $\mathcal{T}_n$ with valuation (with respect to $z_e$) greater than $s(e)$ besides one at $t(e)$. In order for the depth of $\chi_{n,s(e)}$ at $s(e)$ to be equal to $\delta_{\mathcal{T}_n}(s(e))$, we continuously control its differential conductor along $e$ by modifying $G^n_{t(e)}$. The key ingredient is Proposition \ref{propreducekink}, which says that one can always replace $G^n_{t(e)}$ by another element of $\mathcal{G}_{t(e)}$ that ``fits'' better into $\mathcal{T}_n$. Furthermore, we prove that the collection of $G^n_{t(e)}$ whose depth at $s(e)$ is $\delta_{\mathcal{T}_n}(s(e))$, i.e., the quasi-compac quasi-separated (qcqs) set $\mathcal{G}_{s(e)}$, is a non-empty subset of $\mathcal{G}^{\ast}_{t(e)}$ (Proposition \ref{propminimal}). That shows one can secure the wanted ramification data over the edge $e$, completing the base case.
\item \label{step 3} Suppose $e$ is an edge of $\mathcal{T}_n$ where $t(e)$ is \emph{not} final, and $e_1, \ldots, e_l$ are the direct successors of $t(e)$. As in Remark \ref{remarkpartition}, we may write $\underline{G}_{n-1}$ as
\[ \underline{G}_{n-1}= \sum_{i=1}^l \underline{G}_{n-1,s(e_i)} + \underline{G}_{n-1,t(e),\infty}.\]
By induction, for each $i$, there exists a qcqs $\mathcal{G}_{s(e_i)}$ such that for each $G^n_{s(e_i)} \in \mathcal{G}_{s(e_i)}$ and for each $i=1, \ldots, l$, the following holds.
\begin{enumerate}
    \item The branching datum of $\chi_{n,s(e_i)}=\mathfrak{K}_n\big(\underline{G}_{n-1,s(e_i)}, G^n_{s(e_i)}\big)$ coincides with $\mathcal{T}_{n}(e_i)$. 
    \item The depth of $\chi_{n,s(e_i)}$ at $t(e)$ is equal to $\delta_{\mathcal{T}_n}(t(e)\big)$.
\end{enumerate}
Moreover, it follows from Corollary \ref{corollarypartitionsum} that $\mathcal{G}_{t(e)}=\sum_{i=1}^l \mathcal{G}_{s(e_i)}$. Hence, $\mathcal{G}_{t(e)}$ is also qcqs. We call $t(e)$ an \textit{exact vertex} if $\delta_{\mathcal{T}_n}(t(e))>p\delta_{\mathcal{T}_{n-1}}(t(e))$. Otherwise, $\delta_{\mathcal{T}_n}(t(e))=p\delta_{\mathcal{T}_{n-1}}(t(e))$, and we name $t(e)$ a \textit{Cartier vertex}. The controlling of exact vertices (resp. Cartier vertices) to achieve $\omega_{\mathcal{T}_n}(t(e))=\omega_{\chi_{n,t(e)}}(t(e))$ is similar to \ref{step 2}. We again obtain a non-empty open subset $\mathcal{G}^{\ast}_{t(e)} \subset \mathcal{G}_{t(e)}$ whose elements gives the desired differential conductor at $t(e)$.

\item \label{step 4} Let $e$ be the edge from the previous step. We can simply repeat the process from \ref{step 2} to achieve the right degeneration data over $e$. That completes the inductive step.
\item \label{step 5} Iterating the previous two steps, we acquire a nonempty qcqs $\mathcal{G}_{s(e_0)}$ (recall that $e_0$ is the trunk of $\mathcal{T}_n$). Finally, \S \ref{seccontrolroot} shows that $\mathcal{G}_{s(e_0)}$ contains a non-empty subset $\mathcal{G}^{\ast}_{s(e_0)}$ whose elements give rise to characters with desired reduction type. That result, together with Corollary  \ref{corHreductiontype}, complete the proof of the main theorem.
\end{enumerate}

\subsection{Some natural degree \texorpdfstring{$p$}{p} characters}
\label{seccontrollingchars}

\begin{definition}
\label{defextraexact}
Suppose $\mathcal{T}_n$ is a tree that extends $\mathcal{T}_{n-1}$ in Proposition \ref{propextendtree}, and $r$ is a rational place on an edge $e$ of $\mathcal{T}_n$. We may assume that $z_e=0$. Suppose, moreover, that $b_1, \ldots, b_m$ are the successor leaves of $e$ on $\mathcal{T}_n$, $l_j$ is the conductor at $b_j$, and set $l:=\sum_{j=1}^m l_j$. Define $\mathcal{W}_r$ to be the collection of exact differentials forms
\begin{equation}
    \label{eqnjunkpart}
    \omega_r=\sum_{i=1}^{l} \frac{c_idx}{x^i},
\end{equation}
\noindent where $c_i \in k$. Note that, as $\omega_r$ is exact, and $c_{qp-1}=0$ for all $q \in \mathbb{Z}$.
\end{definition}

\begin{lemma}
\label{lemmacontroljunk}
Suppose $\mathcal{T}_n$ is a tree extending $\mathcal{T}_{n-1}$ in Proposition \ref{propextendtree}, and $r$ is a rational place on an edge $e$ of $\mathcal{T}_n$. Suppose, moreover, that we are given $\delta \in \mathbb{Q}_{> 0}$ and $\omega$ is an exact differential form in $\mathcal{W}_{r}$. Then there exists $H_r \in \mathcal{H}_r$ (defined in Definition \ref{defnrightbranchdata}) such that, for $\psi:=\mathfrak{K}_1(H_r)$, we have $\delta_{\psi}(r)=\delta$ and $\omega_{\psi}(r)=\omega$. 
\end{lemma}

\begin{proof}
Suppose $\omega$ has the form as in (\ref{eqnjunkpart}). Suppose, for each $i$, there exists a rational function $H_{r,i} \in \mathcal{H}_r$ so that, for $\psi_i:=\mathfrak{K}_1(H_{r,i})$, we have $\delta_{\psi_i}(r)=\delta$ and $\omega_{\psi_i}(r)=\frac{c_idx}{x^i}$. Then, by Lemma \ref{lemmacombination} \ref{lemmacombination2}, the character $\psi:=\sum_{i=1}^r \psi_i$ works. Therefore, we may further assume that 
\[ \omega=\frac{c_idx}{x^i}, \]
\noindent where $c_i \in k$, $p \nmid (i-1)$, and $i \le l$. By the construction of $\mathcal{T}_n$, all but at most two of the $l_i$'s are divisible by $p$ (mentioned in Remark \ref{remarkatmost2nonp}). Without loss of generality, we may assume that they are $l_1$ and $l_2$. We will treat each case separately.

Suppose first that $p \mid l_2$. That is the situation on most of the places of $\mathcal{T}^{\min}_2$ in Example \ref{exminimaljump}. Consider the rational function
\[H:= \frac{x^{l-i}}{t^{p (\delta-r(i-1))}(x-b_1)^{l_1-1}\prod_{j=2}^m(x-b_j)^{l_j}} \in \mathbb{K}. \]
\noindent Then it is straightforward to check that $H \in \mathcal{H}_r$, $[H]_r=\frac{1}{x_r^{i-1}} \notin (k(x_r))^p$, and $\nu_r(H)=-p \delta$. Hence, by Proposition \ref{propcomputeordp}, it is exactly what we are seeking.

Suppose $p\nmid l_1$ and $p \nmid l_2$, for instance, where $r$ is a place on $e_{0,1}$ of $\mathcal{T}^{8\ness}_{2}$ in Example \ref{exextensionnonessential} and $\mathcal{T}_n=\mathcal{T}^{8\ness}_{2}(r)$. We first regard the case $l_1=m_1p+1$ and $l_2= m_2p+1$. Then, as $\omega \in \mathcal{W}_r$ and $l=pq+2$, we may assume $i \le l-2$. One can show that
\[H:= \frac{x^{l-1-i}}{t^{p (\delta-r(i-1))}(x-b_1)^{mp_1}(x-b_2)^{mp_2}\prod_{j=3}^m(x-b_j)^{l_j}}  \]
\noindent works as before.

Finally, let us consider the case $l_2=m_2p+1$ and $1<\overline{l}_1<p$ as on the edge $e_{0,1}$ of $\mathcal{T}^{9\ness}_2$ in Example \ref{exextensionnonessential}. Then, similar to the other cases, the rational function that does the job is
\[H:=\frac{x^{l-1-i}}{t^{p (\delta-r(i-1))}(x-b_1)^{l_1-1}(x-b_2)^{m_2p}\prod_{j=3}^m(x-b_j)^{l_j}}.   \] We hence have solutions for all the cases. \end{proof}

\begin{corollary}
\label{corutilizeroot}
Suppose $\iota_n$ is the conductor of $\mathcal{T}_n$, $g=\sum_{j=1}^{m_n} c_j /x^j$ is a polynomial in $k[x^{-1}]$ that consists of only terms of prime-to-$p$ degree. Then, after a possible finite extension of $\mathbb{K}$, there exists $G'_{s(e_0)} \in \mathcal{H}_{s(e_0)}$ such that the character $\psi:= \mathfrak{K}_1(G'_{s(e_0)})$ satisfies the followings
\begin{itemize}
    \item $\delta_{\psi}(0)=0$, and
    \item its reduction $\overline{\psi}$, considered as a character in $\textrm{H}^1_p(\kappa)$, corresponds to the extension $y^p-y=g$.
\end{itemize}
\end{corollary}

\begin{proof}
The proof is almost identical to one for Lemma \ref{lemmacontroljunk}. In particular, $c_j/x^j$ in this case plays the role of $-jc_jdx/x^{j+1}$ in the proof of that lemma.
\end{proof}

\subsection{Controlling a final vertex}
\label{seccontrolfinal}
Suppose $v=t(e)$ is a final vertex of $\mathcal{T}_n$ with differential conductor
\[ \omega_{\mathcal{T}_n}(v)=\frac{cdx}{\prod_{j=1}^r(x-[b_j]_v)^{l_j}}, \]
\noindent where $l_j$ is the $n$-th conductor of the branch point associated to the leaf $b_j$ on $\mathcal{T}_n$. It follows from the rules asserted in Theorem \ref{theoremCartierprediction} that either $\mathcal{C}(\omega_{\mathcal{T}_n}(v))=\omega_{\mathcal{T}_{n-1}}(v)=\omega_{\chi_{n-1}}(v)$, or $\mathcal{C}(\omega_{\mathcal{T}_n}(v))=0$. We discuss the former case in \S \ref{secfinalexact} and the latter one in \S \ref{secfinalnonexact}.

Recall from \S \ref{secpartition} that $\chi_{n-1}$ can be represented by a best length-$(n-1)$ Witt vector $(G^1, \ldots, \allowbreak G^{n-1})\allowbreak=(G^1_{v}, \ldots, G^{n-1}_{v})+(G^{1}_{\infty}, \ldots, G^{n-1}_{\infty})$. The discussion in Remark \ref{remarkpartition} allow us to reduce our study to the case where $\underline{G}_{n-1}$ is equal to $(G^1_{v}, \ldots, G^{n-1}_{v})$ and is also best. We first set $G^n=0$. Then, it follows from Theorem \ref{thmbest} and Theorem \ref{theoremCartierprediction} that the character $\chi_{n}$ it gives rise to verifies
\[ \delta_{\chi_n}(v)=p \delta_{\mathcal{T}_{n-1}}(v), \text{ and }  \mathcal{C}(\omega_{\chi_n}(v))=\omega_{\mathcal{T}_{n-1}}(v).  \]
\noindent Furthermore, Theorem \ref{theoremcaljumpirred} shows that the $n$-th conductor of $\chi_n$ at $b_j$ is $\iota_{j,n}=p \iota_{j,n-1}$ for all $j$. Therefore, $0 \in \mathcal{H}_{v}$ and $\delta_{\chi_n}(v) \le \delta_{\mathcal{T}_n}(v)$. In the rest of this subsection, we show that one can replace $0$ by some $G^n \in \mathcal{H}_v$ such that $\delta_{\chi_n}(v)=\delta_{\mathcal{T}_n}(v)$ and $\omega_{\chi_n}(v)=\omega_{\mathcal{T}_n}(v)$.

\subsubsection{The final qcqs}
\label{secfinalqcqs}
Suppose there are $m$ leaves $b_1, \ldots, b_m$ with $n$-th conductors $l_1, \ldots, l_m$, respectively, that are attached to $v$. As discussed in Remark \ref{remarkatmost2nonp}, one may assume that $p$ divides $l_i$ for all $i$'s but maybe $l_1$ and $l_2$.
Consider the subset $\mathcal{H}'_v$ of $\mathcal{H}_v$ that consists of the elements of the form
\begin{equation}
\label{eqngeneralfinalqcqs}
    G=\frac{1}{t^{p\delta_{\mathcal{T}_n}(v)}} \bigg(\sum_{j=1}^m \sum_{i=1}^{l_j-1} \frac{a_{j,i}}{(x_v-(b_j)_v)^i}\bigg),
\end{equation}
where $a_{j,i} \in k((t))$, $ \min_{i \not\equiv 0 \pmod{p}}  (\nu(a_{j,i})) = 0$, and ${a}_{j,i}=0$ where $i \equiv 0 \pmod{p}$ for each $j$. Note that $\mathcal{H}'_v$ can be identified with an affinoid subset of $\mathbb{A}_{\mathbb{K}}^{\sum_{j=1}^m (l_j-1)}$ via the $a_{j,i}$'s. We first show that this collection is exactly the $\mathcal{G}_v$ defined in Definition \ref{defnrightbranchdepth}. 

\begin{proposition}
The subset $\mathcal{H}'_v$ coincides with $\mathcal{G}_v$. In particular, $\mathcal{G}_v$ is an affinoid (hence qcqs).
\end{proposition}

\begin{proof}
Suppse $G \in \mathcal{H}'_v$. It is immediate from (\ref{eqngeneralfinalqcqs}) that $[G]_v \in \kappa_v \setminus \kappa_v^p$. Hence, we have $\delta_G(v)=\delta_{\mathcal{T}_n}(v)$ by Proposition \ref{propcomputeordp}. Therefore, $G \in \mathcal{G}_v$, hence $\mathcal{H}'_v \subseteq \mathcal{G}_v$. Conversely, each $H \in \mathcal{G}_v$ is in the same Artin-Schreier class of one of the form
$$\Tilde{H}= \frac{1}{t^{p\delta_{\mathcal{T}_n(v)}}} (H')_v =\frac{1}{t^{p\delta_{\mathcal{T}_n(v)}}} \bigg(\sum_{j=1}^m \sum_{i=1}^{l_j-1}\bigg( \frac{b_{j,i}}{(x-b_j)^i}\bigg)\bigg)_v  , $$
where $[H']_v \not\in \kappa^p$, again by Proposition \ref{propcomputeordp}. It is straight forward to realize that $\Tilde{H}$ must be of the form (\ref{eqngeneralfinalqcqs}).
\end{proof}

\subsubsection{The exact case}
\label{secfinalexact}
Suppose first that $\omega_{\mathcal{T}_n}(v)$ is exact. Then, by the conditions on an extending tree (Definition \ref{defnextendhurwitz}), we have
$$\delta_{\mathcal{T}_n}(v)>p \delta_{\mathcal{T}_{n-1}}(v) \text{, and } \omega_{\mathcal{T}_n}(v)=\frac{cdx}{\prod_{j=1}^m(x-[b_j]_v)^{l_j}} =:d \overline{h}_v. $$
for some $\overline{h}_v \in k(x)\setminus k(x)^p$. We may further assume that
\[ \overline{h}_v= \sum_{j=1}^m \sum_{i=1}^{l_j-1} \frac{b_{j,i}}{(x-[b_j]_v)^i} \]
where $b_{j,i} \in k$, and $b_{j,l_j-1} \neq 0$ for each $j$. One can consider $\overline{h}_v$, after replacing $x$ by $x_v$, as an element of $k[[t]]](x_v)$. Hence, the rational function
\[ H_v:=\frac{1}{t^{p\delta_{\mathcal{T}_{n}}(v)}} \bigg(\sum_{j=1}^m \sum_{i=1}^{l_j-1} \frac{b_{j,i}}{(x_v-(b_j)_v)^i} \bigg) \in \mathbb{K} \]
is an element of $\mathcal{H}_v$ that verifies, for $\psi:=\mathfrak{K}_1(H_v)$,
\[ \delta_{\psi}(v)=\delta_{\mathcal{T}_{n}}(v), \text{ and } \omega_{\psi}(v)=\omega_{\mathcal{T}_{n}}(v). \]
It follows from Lemma \ref{lemmacombination} \ref{lemmacombination1} that, adding $H_v$ to $G^n=0$ makes
\[ \delta_{\chi_n}=\delta_{\mathcal{T}_{n}}(v), \text{ and } \omega_{\chi_n}=\omega_{\mathcal{T}_{n}}(v). \]
\noindent We thus achieve the desired branching data at $v$. Define
\[ \mathcal{G}^{\ast}_{t(e)}:=\big\{ G \in \mathcal{G}_{t(e)} \mid \omega_{\chi_{n}}(v)=\omega_{\mathcal{T}_{n}}(v) \big\}, \]
which is non-empty as the above construction suggests.

\subsubsection{The non-exact case}
\label{secfinalnonexact}
Let us now consider the case $\mathcal{C}(\omega_{\mathcal{T}_n}(v))=\omega_{\mathcal{T}_{n-1}}(v)=\omega_{\chi_{n-1}}(v)$. Recall that, we start with $G^n=0$. Then, as $\underline{G}_{n-1}$ is best, it follows from Theorem \ref{thmbest} that $$ \delta_{\chi_n}(v)=p \delta_{\chi_{n-1}}(v)=\delta_{\mathcal{T}_n}(v).$$
Thus, $\mathcal{C}(\omega_{\mathcal{T}_n}(v)-\omega_{\chi_n}(v))=\omega_{\mathcal{T}_{n-1}}(v)-\omega_{\mathcal{T}_{n-1}}(v)=0$, i.e., the difference between the two differential forms is exact. Therefore, we may assume
\[\omega_{\mathcal{T}_n}(v)-\omega_{\chi_n}(v) = d\overline{h}_v  \]
\noindent for some $\overline{h}_v \in k(x)\setminus k(x)^p$. In addition, Theorem \ref{theoremcaljumpirred} again tells us that the $n$-th conductor at the branch point $b_j$ is $\iota_{j,n}=p \iota_{j,n-1} \le l_j$ for $1 \le j \le m$. It then follows from Proposition \ref{propvanishingcycle} that
\[ \ord_{[b_j]_v} (d\overline{h}_v) \ge -l_j   \]
\noindent for all $j$. One thus may assume that
\[ \overline{h}_v= \sum_{j=1}^m \sum_{i=1}^{l_j-1} \frac{b_{j,i}}{(x-[b_j]^i)}. \]
\noindent Hence, $d\overline{h}_v \in \mathcal{W}_v$. By the same argument as before, there exists an $H_v \in \mathcal{H}_v$ such that, for $\psi:=\mathfrak{K}_1(H_v)$ ,
\[ \delta_{\psi}(v)=\delta_{\mathcal{T}_n}(v) \text{ and } \omega_{\psi}(v)=d\overline{h}_v .\]
Again, replacing $G^n$ by $G^n+H_v$ results to
\[ \omega_{\chi_n}(v)=\omega_{\mathcal{T}_n}(v).\]
Thus, the set $\mathcal{G}^{\ast}_{t(e)}$ defined in the previous subsection is not empty, completing \ref{step 1}. 

\subsection{Controlling an edge}
\label{seccontroledge}
Suppose $e$ is an edge where $t(e)$ is final. We again may assume that $z_e=0$. Recall that $\mathcal{G}_{t(e)}^{\ast}$, which is the collection of $G \in \mathbb{K}$ that gives rise to a character $\chi_n$ verifying
$$ \delta_{\chi_n}(t(e))=\delta_{\mathcal{T}_n}(t(e)) \text{, and } \omega_{\chi_n}(t(e))=\omega_{\mathcal{T}_n}(t(e)),$$  
is non-empty by \ref{step 1}. Suppose, moreover, that the order of $\infty$ of $\omega_{\mathcal{T}_n}(t(e))$ is $l-2$. Then, by Corollary \ref{corleftrightderivative}, the left derivative of $\delta_{\chi_n}$ at $t(e)$ is $l-1$. Therefore, as $\delta_{\chi_n}$ is piece-wise linear by Proposition \ref{propdeltasw}, there exists a rational $s(e) \le \lambda <t(e)$ such that
\[ \delta_{\chi_n}(s)=\delta_{\mathcal{T}_n}(t(e))-(l-1)(t(e)-s)=\delta_{\mathcal{T}_n}(s)\]
\noindent for all $s \in [\lambda,t(e)]$. Let $\lambda_e(G)$ be the minimal value of $\lambda$ with this property. That means, $\lambda_e(G)$ is the largest ``kink'' of the function $\delta_{\chi_n}$ on $(s(e),t(e))$ (or is $s(e)$ if $\delta_{\chi_n}$ is linear on $[s(e),t(e))$). Our goal is proving that there exists a $G \in  \mathcal{G}^{\ast}_{t(e)}$ satisfying $\lambda_e(G)=s(e)$. First, we will show that one can always reduce the kink on $e$, i.e., making $\lambda_e(G)$ strictly smaller (if it is greater than $s(e)$) by replacing it with another element of $\mathcal{G}^{\ast}_{t(e)}$.

\begin{proposition}
\label{propreducekink}
Suppose $G \in \mathcal{G}^{\ast}_{t(e)}$ with $\lambda_e(G)>s(e)$. Then there exists $G' \in \mathcal{G}^{\ast}_{t(e)}$ with $\lambda_e(G')<\lambda_e(G)$.
\end{proposition}

\begin{proof}
Suppose $\lambda:=\lambda_e(G)$ is the largest kink of $\chi_n$, which is the extension of $\chi_{n-1}$ by $G$, on $e$. We know from Proposition \ref{propvanishingcycle} that $\ord_{\infty}(\omega_{\chi_n}(\lambda))$ is equal to the left-derivative-of-$\delta_{\chi_n}$-plus-one, which is at most $l-2$. Moreover, it follows from the same corollary that
\[ \ord_0(\omega_{\chi_n}(\lambda))=-l, \text{ and } \ord_{\overline{z}}(\omega_{\chi_n}(\lambda)) \ge 0, \hspace{2mm} \forall  \overline{z} \neq 0, \infty. \]
Let us first consider the case $\delta_{\chi_n}(\lambda)>p\delta_{\chi_{n-1}}(\lambda)$. Then $\omega_{\chi_n}(s)$ is exact and $(l-1)$ is prime to $p$. Therefore, we may assume that $\omega_{\chi}(\lambda)$ is of the form
\[ \omega_{\chi}(\lambda)=\frac{cdx}{x^l} + d\overline{F} \]
\noindent where $0 \neq d\overline{F} \in \mathcal{W}_{\lambda}$. Applying Lemma \ref{lemmacontroljunk}, we obtain $F \in \mathcal{H}_{\lambda}$, such that, for $\psi:=\mathfrak{K}_1(F)$,
\[ \delta_{\psi}(\lambda)=\delta_{\chi_n}(\lambda), \text{ and } \omega_{\psi}(\lambda)=-d\overline{F}. \]
\noindent One may check that $G':=G+F$ lies in $\mathcal{G}_{t(e)}$. Finally, it follows from Lemma \ref{lemmacombination} \ref{lemmacombination2} that replacing $G$ by $G'$ changes $\chi_n$ as follows
\[ \delta_{\chi_n}(\lambda)=\delta_{\mathcal{T}_n}(\lambda), \text{ and } \omega_{\chi_{n}}(\lambda)= \frac{cdx}{x^l}.  \]
\noindent Thus, $\chi_n$ has no kink at $\lambda$, whence $\lambda_e(G')<\lambda_e(G)$.

Let us now consider the case $\delta_{\chi_n}(\lambda)=p\delta_{\chi_{n-1}}(\lambda)$. The same line of reason as in \S \ref{secfinalnonexact} gives
\[ \omega_{\chi_n}(\lambda)-\omega_{\mathcal{T}_n(\lambda)}=d\overline{F}, \]
\noindent where $\overline{F}$ is as in the previous case. Repeating the above process, we obtain a $G' \in \mathcal{G}^{\ast}_{t(e)}$ that verifies $\lambda_e(G')<\lambda$.
\end{proof}

\begin{remark}
\label{remarkstrategymin}
If $\lambda_e(G')$ is once more strictly greater than $s(e)$, we can repeat the procedure from the above proof to obtain an element of $\mathcal{G}^{\ast}_{t(e)}$ with strictly smaller kink (on $e$). Thus, if we can show that $\lambda_e$ achieves a minimum value then it must be equal to $s(e)$. That is the goal of the next section.
\end{remark}

\subsection{The minimal depth Swan conductor}
\label{secminimal}
This section adapts \cite[\S 6.4]{MR3194815}. We would like show that the function $\lambda_{e}: \mathcal{G}_{t(e)} \rightarrow \mathbb{Q}_{\ge 0}$ defined in \S \ref{seccontroledge} takes the value $s(e)$ for some $G_{t(e)} \in \mathcal{G}^{\ast}_{t(e)}$. As discussed in Remark \ref{remarkstrategymin}, it suffices to prove the following.

\begin{proposition}
\label{propminimal}
The function $G_{t(e)} \mapsto \lambda_{e}(G_{t(e)})$ takes a minimal value on $\mathcal{G}_{t(e)}$ (i.e. $\mathcal{G}_{s(e)}$ is nonempty). Moreover, $\mathcal{G}_{s(e)}$ can be identified with a qcqs set. 
\end{proposition}

\subsubsection{A lemma from rigid analysis}
The following lemma will be a crucial ingredient in the proof of Proposition \ref{propminimal}.

\begin{lemma}[c.f. {\cite[Lemma 6.16]{MR3194815}}]
\label{lemmamax}
Let $X$ be qcqs over $K$ and $f^1, \ldots, f^n \allowbreak \in A$ be analytic functions on $X$. Then the function
\[ \phi: X \rightarrow \mathbb{R}, \hspace{1cm} x \mapsto \max_{1 \le i \le n} \sqrt[i]{\lvert f^i(x) \rvert} \]
takes a minimal value on $X$. Equivalently, the function
\[x \mapsto \tau(x):= \min_{1 \le i \le n} \frac{\nu(f^i(x))}{i} \]
takes a maximal value on $X$. Furthermore, the subset of $X$ on which the minimum (maximum) is attained is qcqs.
\end{lemma}

\begin{proof}
The proof adapts \cite[Lemma 6.16]{MR3194815} and \cite[Lemma 4.19]{MR3552986}. It suffices to show that, if the following inequality holds $$ \gamma:=\sup_{a \in X} \bigg(\min_{1 \le i \le n} \frac{\nu(f^i(a))}{i}\bigg) \ge 0, $$
then $\gamma$ is achieved on a qcqs subset of $X$. We may assume that $X$ is an affinoid. We set $g_i(x):=\nu(f^i(x))/i$, for $1 \le i \le n$. For each $i$, let $X_i \subseteq X$ be the rational sub-domain where $g_i$ is minimal among all the $g_j$'s. Then the restriction of $\tau$ to $X_i$ is simply equal to $g_i$, which attains its maximum on $X_i$ by the maximum modulus principle. Furthermore, the subspace of $X_i$ where this maximum is attained is a Weierstrass domain in $X_i$, hence qcqs. Thus, the subspace of $X$ where $\gamma$ is attained is a union of finitely many qcqs spaces, whence also qcqs, completing the proof.
\end{proof}

\subsubsection{A quasi-compact quasi-separated contained in \texorpdfstring{$\mathcal{G}_{t(e)}$}{Gte}}
\label{secinitialqcqs}

Consider an edge $e$ of the tree $\mathcal{T}_n$. Assume for a moment that we got the desired refined Swan conductor at every $r \ge t(e)$. Suppose, moreover, that there are $m$ edges $e_1, e_2, \ldots, e_m$ where $s(e_i)=t(e)$ for each $i$. Set $I_{e_j}$ to be an indexing set of leaves succeeding $e_j$. It follows from the previous constructions that an element $G_{t(e)} \in \mathcal{G}_{t(e)}$ can be represented by
\begin{equation}
    \label{eqninitialvertex}
    G_{t(e)} =
\sum_{j=1}^m G_{s(e_j)}=\sum_{j=1}^m \Bigg( \sum_{h\in \cup I_{e_j}}\bigg( \sum_{i=1}^{l_h-1} \frac{a_{h,i}}{(x-b_h)^i}\bigg)\Bigg),
\end{equation}
where $G_{s(e_j)} \in \mathcal{G}_{s(e_j)}$, and $a_{h,i}=0$ for $i \equiv 0 \pmod{p}$. Set $m_{t(e)}:=\sum_{j=1}^m \sum_{h\in \cup I_{e_j}} (l_h-1-\lfloor l_h/p \rfloor)$. As before, we can think of $\mathcal{G}_{t(e)}$ as a subset of an affine $m_{t(e)}$-space over $K$ via the coordinates $a_{h,i}$ and $G_{t(e)}$ as one of its $K$-rational point. Moreover, by the induction hypothesis in \S \ref{sectinduction}, the subset $\mathcal{G}_{s(e_j)}$ is qcqs for $j=1, \ldots, l$. Hence, $\mathcal{G}_{t(e)}=\sum_{j=1}^m \mathcal{G}_{s(e_j)} $ is also a qcqs subset of $(\mathbb{A}_K^{m_{t(e)}})^{\an}$. 

It is a consequence of \S \ref{sectcontrolvertex}  that $\emptyset \neq \mathcal{G}^{\ast}_{t(e)} \subsetneq \mathcal{G}_{t(e)}$ where
\[ \mathcal{G}^{\ast}_{t(e)}= \{ G \in \mathcal{G}_{t(e)} \mid \omega_{\chi_n(G)}(t(e))=\omega_{\mathcal{T}_n}(t(e)) \}. \]
\noindent As a rigid analytic space, $\mathcal{G}^{\ast}_{t(e)}$ is an open subset of $\mathcal{G}_{t(e)}$. The goal of this section is to show that $\lambda_e(\mathcal{G}_{t(e)})$ takes the minimal value on $\mathcal{G}_{t(e)}$ and that the points where that minimum is achieved form a qcqs subset of $\mathcal{G}^{\ast}_{t(e)}$.

\subsubsection{Analysis}
\label{subrestriction}
Let $\phi_{n-1}:Y_{n-1} \rightarrow X$ be the Galois cover corresponding to the character $\chi_{n-1}$. By our induction hypothesis, it has good reduction and is totally ramified above $x=z_e$ which we may assume to be $0$. As usual, $D \subset X^{\an}$ is an open unit disc center at $0$. It follows that the rigid analytic subspace $C:=\phi_{n-1}^{-1}(D) \subseteq Y_{n-1}$ is another unit disc and contains the unique point $y_{n-1} \in Y_{n-1}$ above $x=0$. We choose a parameter $\widetilde{x}$ for the disc $C$ such that $\widetilde{x}(y_{n-1})=0$. Then
\[ x= \widetilde{x}^{p^{n-1}} u(\widetilde{x}), \hspace{2mm} \text{ with } u(\widetilde{x}) \in R[[\widetilde{x}]]^{\times}. \]
\noindent One sees that for $r>0$, the inverse image of the closed disc $D[pr] \subset D$ defined by the condition $\nu(x) \ge r$ is the closed disc $C[\widetilde{r}]$ defined by $\nu(\widetilde{x})\ge \widetilde{r}:=p^{-n+1} r$. Set $\widetilde{r}_{n-1}:=p^{-n+1} r_{n-1}$. Let $\mathbb{K}_{n-1}$ denote the function field of $Y_{n-1}$.

Let us fix a character $\chi \in \textrm{H}^1_{p^n}(\mathbb{K})$ that is given rise from $\chi_{n-1}$ by a rational $G \in \mathcal{G}_{t(e)}$. Let $\widetilde{\chi}:=\chi \lvert_{\mathbb{K}_{n-1}} \in \textrm{H}^1_p(\mathbb{K}_{n-1})$ denote the restriction of $\chi$ to $\mathbb{K}_{n-1}$. If $\chi$ corresponds to a cover $Y_n \rightarrow X$, then $\widetilde{\chi}$ correspond to the cover $Y_n \rightarrow Y_{n-1}$. In analogy to $\lambda_e(\chi)$, we define the function $\lambda_e(\widetilde{\chi}): \textrm{H}^1_{p}(\mathbb{K}_{n-1}) \xrightarrow[]{} \mathbb{Q}$ as follows
\[\lambda_e(\widetilde{\chi})= \begin{cases} 
       \min(\widetilde{s(e)} \le \widetilde{r} <\widetilde{t(e)} \mid \text{$\delta_{\widetilde{\chi}}$ linear on $[\widetilde{s(e)}, \widetilde{t(e)}]$}), & \chi \in \mathcal{G}^{\ast}_{t(e)} \\
      \widetilde{t(e)}, & \chi \in \mathcal{G}_{t(e)} \setminus \mathcal{G}^{\ast}_{t(e)}.
   \end{cases}
\]
Let $m_i$ (resp. $m$) be the left derivative (on $e$) of $\delta_{\chi_i}$ (resp. $\delta_{\chi}$) at $t(e)$. The following result suggests that one may apply the tools from \S \ref{secdetect} to determine the kinks of $\chi$ on $e$.

\begin{lemma}
\label{lemmarestriction}
 \begin{enumerate}
     \item \label{lemmarestriction1} For $\chi \in \mathcal{G}^{\ast}_{t(e)}$, we have $\lambda_e(\widetilde{\chi})=p^{-n+1}\lambda_e(\chi)$.
     \item \label{lemmarestriction2} Let
     \[ \widetilde{m}=p^{n-1} m-\sum_{i=1}^{n-1}m_i(p-1)p^{i-1} \]
     \noindent Then $p \nmid \widetilde{m}$ and the character $\widetilde{\chi} \in \textrm{H}^1_p(\mathbb{K}_{n-1})$ satisfies the conditions of (a),(b), and (c) of \S \ref{secdetect}  with respect to $\widetilde{m}$, the open disc $C \in Y^{\an}_{n-1}$ and the family of subdisc $C[\widetilde{r}]$ for $\widetilde{r} \in [\widetilde{s(e)},\widetilde{t(e)}]$.
     \item \label{lemmarestriction3} If $\lambda_{\widetilde{m},\widetilde{s(e)}}(\widetilde{\chi})$ is defined as in Corollary \ref{cordetectlowerbound}, and if we set $r_0$ in Proposition \ref{propdetect} to be equal to $\widetilde{t(e)}$, then $\lambda_e(\widetilde{\chi})=\lambda_{\widetilde{m}, \widetilde{s(e)}}(\widetilde{\chi})$ for all $\chi \in \mathcal{G}_{t(e)}$.
 \end{enumerate}
\end{lemma}

\begin{proof}

For $r>s(e)$, we systematically use the notation $\widetilde{r}:=p^{-n+1}r$. Let the valuation $\nu_{\widetilde{r}}$ of $\mathbb{K}_{n-1}$, which corresponds to the Gauss valuation on $C[\widetilde{r}]$ (\S \ref{secdiscannuli}), be the unique extension of $\nu_r$. By \cite[\S 7.1]{MR3167623} we have the following equation
\begin{equation} 
\begin{split}
\delta_{\widetilde{\chi}}(\widetilde{r}) & = \psi_{\mathbb{K}_{n-1}/\mathbb{K}}(\delta_{\chi}(r))  \\
 & = \delta_{\chi}(r)-\bigg( \delta_1(r) \frac{p-1}{p^{n-1}} +\ldots + \delta_{n-1}(r) \frac{p-1}{p} \bigg),
\end{split}
\end{equation}
\noindent where $\psi$ is the inverse Herbrand function \cite[IV, \S 3]{MR554237}. Since all the characters $\chi_i:=\chi^{p^{n-i}} (1 \le i <n)$ have good reduction by the hypothesis and their branch points are all contained in $\mathcal{D}[t(e)]$, each $\delta_i:=\delta_{\chi_i} (1 \le i <n)$ is linear of slope $m_i$ on the interval $[s(e),t(e)]$. Therefore, the left and the right derivative of $\delta_i(r)$ is equal to $m_i$ for all $r \in (s(e),t(e))$. Let $c$ be the left slope of $\delta_{\chi}$ at $r$. Then the left slope of $\delta_{\widetilde{\chi}}$ at $\widetilde{r}$ is equal to 
\[ p^{n-1}c -\sum_{i=1}^{n-1}m_i(p-1)p^{i-1}= p^{n-1} c +\widetilde{m} -p^{n-1}m.\]
\noindent Part (\ref{lemmarestriction1}) then follows immediately. Part (\ref{lemmarestriction2}) follows from the fact that $c \le m$. Part (\ref{lemmarestriction3}) also follows from this property, along with Proposition \ref{propdeltasw} and the fact that $\sw_{\chi}(t(e),\infty)=m$ if and only if $\chi \in \mathcal{G}^{\ast}_{t(e)}$.
\end{proof}

One may assume that the character $\widetilde{\chi}$ is the Artin-Schreier class of the rational function
\[ \widetilde{G}:=F^{n}(y_1, \ldots, y_{n-1})+G^n \in \mathbb{K}^{\times}_{n-1},  \]
\noindent where $F^n$ is a polynomial over $\mathbb{F}_p$ in $(n-1)$ variables as in \cite[Proposition 6.5]{MR3051249}. We thus may write $\widetilde{G}$ as a power series in the parameters $\widetilde{x}$ as follows
\[ \widetilde{G}= \sum_{l=1}^{\infty}a_l \widetilde{x}^{-l}. \]
Since $G^1, \ldots, G^{n-1}$ are fixed, $\widetilde{G}$ is determined by the choice of $G^n$. So, we may think of the $a_l$'s as analytic functions on the space $\mathcal{G}_{t(e)}$. In fact, $a_l$ is a polynomial in the coordinate $a_{k,i}$ (as in (\ref{eqninitialvertex})) with coefficients in $R$.

\subsubsection{}
\label{secanalyticcover}
We continue with the process of proving Proposition \ref{propminimal}. Suppose $G_0$ is an arbitrary element of $\mathcal{G}^{\ast}_{t(e)}$ and $\chi_0$ is the character it gives rise to. Then it is immediate that $\lambda_e(G_0)<t(e)$. We therefore may choose a rational number $s \in (\lambda(G_0), t(e))$. With the notation from Section \ref{subrestriction}, recall that $\widetilde{\chi}$ is the restriction of $\chi$ to the function field $\mathbb{K}_{n-1}$ of $Y_{n-1}$. It is an Artin-Schreier cover defined by a rational function $\widetilde{G} \in \mathbb{K}_{n-1}^{\times}$. By Lemma \ref{lemmarestriction}, we have
\[ \lambda(\widetilde{G}) < \widetilde{s}:=p^{1-n}s < \widetilde{t(e)}. \]
\noindent Set $\widetilde{\delta}:=\delta_{\widetilde{\chi}} (\widetilde{t(e)})$. Let $N$ be an integer such that
\[ Np \ge \frac{\widetilde{\delta}}{\widetilde{t(e)}-\widetilde{s}}. \]
We hence arrive at the situation of (\ref{eqndetectbound}). The following is parallel to \cite[Lemma 6.8]{MR3194815}
\begin{lemma}
\label{lemmadetectapp}
There exists a finite cover $\mathcal{G}'_{t(e)} \rightarrow \mathcal{G}_{t(e)}$ and analytic functions $b_1, \ldots, b_N$ on $\mathcal{G}'_{t(e)}$ with the following property. Set
\[ d:= \sum_{j=0}^N b_j \widetilde{x}_1^{-j} \]
\noindent and write
\[ F+d^p-d=\sum_{l=1}^{\infty}c_l \widetilde{x}_1^{-l}, \]
\noindent where the $c_l$ are now analytic functions on $\mathcal{G}'_{t(e)}$. Then
\begin{enumerate}
    \item for all $l \ge 1$ and all points $y \in \mathcal{G}'_{t(e)}$, we have $\nu(c_l(y)) \ge p( \widetilde{t(e)} l -\widetilde{\delta})$;
    \item we have $c_{pl}=0$ for $l \le N$.
\end{enumerate}
\end{lemma}

\begin{proof}
The lemma follows immediately from Proposition \ref{propdetect} and Remark \ref{remarkdetect}.
\end{proof}

\subsubsection{Proof of Proposition \ref{propminimal}}
We can now complete the proof of Proposition \ref{propminimal}. Let $\widetilde{m}$ be as in part (\ref{lemmarestriction2}) of Lemma \ref{lemmarestriction}. Define the function $\mu_{\widetilde{m},\widetilde{s(e)}}:\mathcal{G}'_{t(e)} \rightarrow \mathbb{R}$ as follows
\[ \mu_{\widetilde{m},\widetilde{s(e)}}(x):=\max\bigg(\bigg\{\frac{\nu(c_{\widetilde{m}}(x))-\nu(c_l(x))}{\widetilde{m}-l} \mid 1 \le l <\widetilde{m}\bigg\} \cup \{\widetilde{s(e)} \}\bigg). \]
\noindent Let $\chi \in \mathcal{G}_{t(e)}$, write $\widetilde{\chi}:=\chi\lvert_{\mathbb{K}_{n-1}}$ for its restriction to the function field of $Y_{n-1}$, and let $x \in \mathcal{G}'_{t(e)}$ be an arbitrary point over $\chi$. Thanks to Lemma \ref{lemmadetectapp}, we can apply Proposition \ref{propdetect} to compare $\mu_{\widetilde{m},\widetilde{s(e)}}(x)$ to $\lambda_{\widetilde{m},\widetilde{s(e)}}(\widetilde{\chi})$, which, in turn, is equal to $\lambda_e(\widetilde{\chi})$ by Lemma \ref{lemmarestriction} (\ref{lemmarestriction3}). We thus conclude that $\mu_{\widetilde{m},\widetilde{s(e)}}(x)< \widetilde{s}$ if and only if $\lambda_e(\widetilde{\chi}) < \widetilde{s}$. Moreover, if this is the case, then we have $\mu_{\widetilde{m},\widetilde{s(e)}}(x)=\lambda_e(\widetilde{\chi})$. In any case, by Lemma \ref{lemmarestriction} (\ref{lemmarestriction1}), we have $\lambda_e(\chi)=p^{n-1} \lambda_e(\widetilde{\chi})$ when $\chi \in \mathcal{G}^{\ast}_{t(e)}$.

The rest of the proof is identical to one in \cite[\S 6.4.5]{MR3194815}. Let $G_0 \in \mathcal{G}^{\ast}_{t(e)}$ be the rational function at the beginning of \S \ref{secanalyticcover}. Let $G'_0 \in \mathcal{G}'_{t(e)}$ be a point above $G_0$. Because $\lambda_e(G_0)<s$, we have $\mu_{\widetilde{m},\widetilde{s(e)}}(G'_0)<\widetilde{s}$ as discussed above. 
 
It follows from Lemma \ref{lemmamax} that the function $\mu_{\widetilde{m},\widetilde{s(e)}}$ takes a minimum on $\mathcal{G}'_{t(e)}$. Let $w \in \mathcal{G}'_{t(e)}$ be a point where this minimum is achieved, and let $W \in \mathcal{G}_{t(e)}$ be its image under $\mathcal{G}'_{t(e)} \rightarrow \mathcal{G}_{t(e)}$. We have $\mu_{\widetilde{m},\widetilde{s(e)}}(w) \le \mu_{\widetilde{m},\widetilde{s(e)}} (G'_0) < \widetilde{s}$. Since
\[ \lambda_{\widetilde{m},\widetilde{s(e)}} (\widetilde{W})=\mu_{\widetilde{m},\widetilde{s(e)}} (w)< \widetilde{s} <\widetilde{t(e)}, \]
\noindent we see that $W \in \mathcal{G}^{\ast}_{t(e)}$. Applying the above arguments a second time, we conclude that $\lambda_e(W)=p^{n-1}\mu_{\widetilde{m},\widetilde{s(e)}} (w)$, and this is actually the minimal value of the function $\lambda_e: \mathcal{G}^{\ast}_{t(e)} \rightarrow \mathbb{R}$. Moreover, the subset on which the minimum is attained is qcqs by the Lemma \ref{lemmamax}. This completes the proof of Proposition \ref{propminimal}.\qed



\subsection{Controlling a non-final vertex}
\label{sectcontrolvertex}
Suppose $t(e)=s(e_1)=\ldots=s(e_m)$ is a non-final-vertex of $\mathcal{T}_n$. By the induction process, we obtain a bunch of qcqs sets $\mathcal{G}_{s(e_1)}, \ldots, \mathcal{G}_{s(e_m)}$. Recall that each $G_i \in \mathcal{G}_{s(e_1)}$ has (generic) branching datum fits into $\mathcal{T}_n(e_i)$ and verifies, for $\chi_{G_i}$ be the extension of $\chi_{n-1,e_i}$ by $G_i$, the equation $\delta_{\chi_{G_i}}(s(e_i))=\delta_{\mathcal{T}_n}(s(e_i))$. Suppose, moreover, that all the branch points of $\chi_{G_i}$ specialize to $\overline{a}_i \in k$ at $t(e)$. Then the $\overline{a}_i$'s are distinct, and the differential conductor $\omega_{\chi_{G_i}}(t(e))$ is of the form
\[ \omega_{\chi_{G_i}}(t(e))=\frac{f^i(x)dx}{(x-\overline{a}_i)^{l_i}},\]
where $f^i(x) \in k[x]$ and $l_i:=\sum_{b \in \mathbb{B}(\mathcal{T}_n(e_i))} h_b$. Set $G:=\sum_{i=1}^m G_i$ and $\chi_{G}$ is the character that is given rise from $\chi_{n-1,e}$ by $G$. By Lemma \ref{lemmacombination} \ref{lemmacombination2}, we have 
$$\omega_{\chi_G}(t(e))=\sum_{i=1}^m \omega_{\chi_{G_i}}(t(e)) \text{, and } \delta_{\chi_G}(t(e))=\delta_{\chi_{G_i}}(t(e)) \hspace{2mm} \forall i=1, \ldots, m.$$ 
Hence, the collection of such $G$, which we call $\mathcal{G}_{t(e)}$, is equal to $\sum_{i=1}^m \mathcal{G}_{s(e_i)}$ and thus is qcqs. Its elements give rise to characters with the desired depth and the geometry of the branch locus at $t(e)$. Suppose the wanted differential form at $t(e)$ is
\[ \omega_{\mathcal{T}_n}(t(e))=\frac{cdx}{\prod_{i=1}^m(x-\overline{a}_i)^{l_i}} \hspace{2mm} (c \neq 0). \]
\noindent Recall that $\mathcal{C}(\omega_{\mathcal{T}_n}(t(e)))=\omega_{\mathcal{T}_{n-1}}(t(e))$ when $\delta_{\mathcal{T}_n}(t(e))=p\delta_{\mathcal{T}_{n-1}}(t(e))$, and $\mathcal{C}(\omega_{\mathcal{T}_n}(t(e)))=0$ otherwise. Either cases, as  we achieve the right depth Swan conductor at $t(e)$, we may assume that $\omega_{\chi_G}$ has the form
\[ \omega_{\chi_G} (t(e))= \frac{cdx}{\prod_{i=1}^m(x-\overline{a}_i)^{l_i}} + \sum_{i=1}^m \omega'_{\chi_G}(s(e_i)),\]
where $\omega'_{\chi_G}(s(e_i)) \in \mathcal{W}_{s(e_i)}$ is exact. Now, by repeating the process in \S \ref{seccontrolfinal} for each $i$, we replace $G$ by another one in $\mathcal{G}_{t(e)}$ where $\omega'_{\chi_G}(s(e_i))=0$ for all $i$. That shows $\mathcal{G}^{\ast}_{t(e)}$ is non-empty.

\subsection{Controlling the root}
\label{seccontrolroot} 
By the previous steps of the induction process, we get to the point where $\mathcal{G}_{s(e_0)} \neq \varnothing$. That means there exists a $G_{\min} \in \mathcal{G}_{s(e_0)}$ that gives rise to a character $\chi_{\min}$ whose braching datum fits to $\mathcal{T}_n$ and whose depth is zero at $s(e_0)$, i.e., $\chi_{\min}$ has good reduction.

Recall that $\overline{\chi}_n$ is our original character. Assume that ramification breaks of the one point cover corresponding to $\overline{\chi}_n$ is $(m_1, \ldots, m_n)$, and it is represented (upon completion at $x=0$) by a Witt vector $\underline{g}_n:=(g^1, \ldots, g^n) \in W_n(k(x))$. As discussed in \S \ref{secdeformonepointcover}, one may regard $\overline{\chi}_n$ as a one-point-cover of $\mathbb{P}^1_k$. Therefore, we further assume that each $g^i$ is a polynomial in $k[x^{-1}]$, all of whose terms have prime-to-$p$ degree. On the other hand, it follows from the previous constructions that $\overline{\chi}_{\min}$ corresponds to a Witt vector $\overline{g}_{\min}=(g^1, \ldots, g^{n-1},g_{\min})$ where $g_{\min} \in k[x^{-1}]$ has degree less than (only when $m_n=pm_{n-1}$) or equal to $m_n$ and consists of only terms of prime-to-$p$ degree. Subtracting Witt vectors yield
$$ \underline{g}_n -\underline{g}_{\min}=(0, \ldots, 0, g^n -g_{\min}). $$
\noindent We define $g:=g^n-g_{\min}$. Recall that Corollary \ref{corHreductiontype} asserting $\chi_{\min}$ deforms $\overline{\chi}_n$ if and only if $g=0$.

\begin{proposition}
There is a character $\chi_n$ that is a deformation of $\overline{\chi}_n$.
\end{proposition}

\begin{proof}
Applying Corollary \ref{corutilizeroot} for $g$, we obtain $H \in \mathcal{H}_{s(e_0)}$ such that $\psi:=\mathfrak{K}_1(H) \in \textrm{H}^1_p(\mathbb{K})$ satisfying the following
\[ \delta_{\psi}(0)=0 \text{, and } \overline{\psi} \text{ is given by } y^p-y=g.   \]
\noindent Hence, if $\psi':=\mathfrak{K}_n((0, \ldots, 0, H)) \in \textrm{H}^1_{p^n}(\mathbb{K})$, then we also have $\delta_{\psi'}(0)=0$ and the reduction $\overline{\psi}'$ corresponds to the same extension, which is represented by the Witt vector $(0, \ldots, g) \in W_n(\kappa)$.

On the other hand, recall that $\chi_{\min}$ corresponds to the length-$n$-Witt-vector
\[(G^1, G^2, \ldots, G^{n-1},G_{\min}) \in W_n(\mathbb{K})/\wp(W_n(\mathbb{K})). \]
\noindent Therefore, replacing $G_{\min}$ by $G^n:=G_{\min}+H$ equates to multiplying $\psi'$ to $\chi_{\min}$. By Lemma \ref{lemmacombination}, the result is an {\'e}tale character $\chi_n \in \textrm{H}^1_{p^n}(\mathbb{K})$ with reduction $\overline{\chi}_{\min} \cdot \overline{\psi}'$, which, in turn, is defined by the Witt vector $(g^1, \ldots, g^{n-1}, g_{\min}+g=g^n)$. Finally, since $G^n \in \mathcal{G}_{s(e_0)}$ as $H \in \mathcal{H}_{s(e_0)}$, it follows from Corollary \ref{corHreductiontype} that $\chi_n$ is a lift of $\overline{\chi}_n$.
\end{proof}

That completes the proof of the Main Theorem.  \qed

\begin{remark}
It follows from the construction and the compatibility of the differential conductors that the $\mathbb{Z}/p^n$-tree arising from $\chi_n$ coincides with the tree $\mathcal{T}_n$.
\end{remark}

\section{Proofs of Technical Results}
\label{sectechnical}

\subsection{A solution to the Cartier operator equation}
\label{secsolutioncartierequation}

\begin{proposition}
\label{propexistenceCartiersolution}
Suppose we are given a differential form as below $$\omega=\frac{cdx}{\prod_{i=1}^r(x-d_i)^{\iota_i}},$$
where $c$ and the $d_i$'s are in $k$, $d_i \neq d_j$ for $i \neq j$, $r \ge 2$, and a fixed integer $1 \le a \le p-1$. Set $\iota'_1:=p\iota_1-p+a+1$ ,$\iota'_2:=p\iota_2-p+1$, $\iota_i':=p\iota_i$ for each $i \neq 1,2$. Then the differential form $$\omega'=\frac{c^p dx}{ (-d_1-d_2)^{p-a-1} \prod_{i=1}^r(x-d_i)^{\iota'_i}} $$ verifies $\mathcal{C}(\omega')=\omega$.
\end{proposition}

\begin{proof}
The differential form $\omega'$ above can be rewritten as $$\omega'=\frac{c^p (x-d_1)^{p-a-1} dx}{(x-d_2) (-d_1-d_2)^{p-a-1} (x-d_2)^{p(\iota_2-1)}(x-d_1)^{p\iota_1}\prod_{i=3}^r(x-d_i)^{p\iota_i}}. $$ The proposition then follows from the fact that $\mathcal{C}(u^p\omega)=u\mathcal{C}(\omega)$ for all $u \in K$.
\end{proof}

\subsection{Construction of the extending Hurwitz tree}
\label{secconstructtree}

We dedicate this section for the proof of Proposition \ref{propextendtree}. In each proof, we will only demonstrate the construction of the extending Hurwitz trees. The readers can easily check that these trees satisfy the axioms in Definition \ref{defnhurwitztree} together with the conditions for an extension asserted by Theorem \ref{theoremCartierprediction} and Theorem \ref{theoremcompatibilitydiff}.

Recall the set up of Proposition \ref{propextendtree}. We are give a $\mathbb{Z}/p^{n-1}$-tree $\mathcal{T}_{n-1}$ and an $n$-th dengeneration data $(\delta_{n}, \omega_n) \in \mathbb{Q}_{>0} \times \Omega^1_{\kappa}$ (or $(0, f_n)$, where $f_n \in \kappa$) at its root. The goal is to construct a $\mathbb{Z}/p^n$-tree $\mathcal{T}_n$ that extends $\mathcal{T}_{n-1}$ and whose degeneration data at its root is the same as the given one. Note that the original result only requires $\delta_n=0$ but we will prove the general case in order to utilize an induction technique that we will soon see. 

Suppose $\mathcal{T}_{n-1}$ is a $\mathbb{Z}/p^{n-1}$-tree that has $m$ leaves $b_1, \ldots, b_m$ ($m \ge 2$), and the conductor at $b_i$ is $\iota_{i,n-1}$. As usual, $v_0$ and $e_0$ denote the root and the trunk of $\mathcal{T}_{n-1}$. Set $v_1:=t(e_0)$. Suppose, moreover, that the differential conductor (the $(n-1)$th degeneration) at $v_1$ and $v_0$ of $\mathcal{T}_{n-1}$ are 
\begin{equation}
\label{eqninitialdiffcondn-1}
\begin{split}
   \omega_{\mathcal{T}_{n-1}}(v_1)&= \frac{c_{v_1}dx}{\prod_{i=1}^r(x-[b_i]_{v_1})^{\iota_{i,n-1}}},\\   \omega_{\mathcal{T}_{n-1}}(v_0) &=\frac{c_{v_0}dx}{x^{\iota_{n-1}}} + \sum_{j < \iota_{n-1}} \frac{c_j dx}{x^j} \hspace{2mm} \bigg(g^{n-1}=\sum_{j=1}^{l} \frac{d_j}{x^j}\bigg), 
\end{split}
\end{equation}
respectively, where $l<\iota_{n-1}$. We may write $\omega_{\mathcal{T}_{n-1}}(v_0)$ in (\ref{eqninitialdiffcondn-1}) as $f(x)dx/x^{\iota_{n-1}}$. Then, when $\delta_{\mathcal{T}_{n-1}}(v_0)>0$, as the two forms are compatible (\S \ref{seccompatibility}), we have the relation $c_{v_0}=c_{v_1}=:c$. When $\delta_{\mathcal{T}_{n-1}}(v_0)=0$ and $\iota_{n-1}=p\iota_{n-2}-p+1$, where $\iota_{n-2}$ is the conductor of the sub-$\mathbb{Z}/p^{n-2}$-tree, as discussed in Theorem \ref{theoremcompatibilitydiff}, we don't have to worry about the compatibility at the root. Otherwise, the same theorem asserts that $l=\iota_{n-1}-1$ and $c_{v_1}=-ld_l$.

Suppose, moreover, that we are given a level $n$ depth $\delta_{\mathcal{T}_n}(v_0) \ge p\delta_{\mathcal{T}_{n-1}}(v_0)$ and a level $n$ differential conductor (or degeneration) at the root as follows
\begin{equation}
    \label{eqnintialdiffcondn}
    \omega_{\mathcal{T}_{n}}(v_0)= \frac{C_{v_0}dx}{x^{\iota_n}}+ \sum_{j < \iota_n} \frac{C_j dx}{x^j} \hspace{2mm} \bigg(\text{or } g^{n}=\sum_{j=1}^{L} \frac{D_j}{x^j}\bigg),
\end{equation}
where $L< \iota_{n}$. When $\delta_{\mathcal{T}_n}(v_0)>0$, it is required that $\mathcal{C}(\omega_{\mathcal{T}_{n}}(v_0))=\omega_{\mathcal{T}_{n-1}}(v_0)$ if $\delta_{\mathcal{T}_n}(v_0)=p\delta_{\mathcal{T}_{n-1}}(v_0)$, or $\omega_{\mathcal{T}_n}(v_0)$ is exact if $\delta_{\mathcal{T}_n}(v_0) > p\delta_{\mathcal{T}_{n-1}}(v_0)$. When $\delta_{\mathcal{T}_n}(v_0)= p\delta_{\mathcal{T}_{n-1}}(v_0)=0$, then $L=\iota_{n}-1$ if $\iota_n>\iota_{n-1}-p+1$.

Just as the main theme of the whole paper, we first construct the extensions of certain sub-trees of the given tree $\mathcal{T}_{n-1}$. 

\begin{definition}
With the notation as above, let $\mathcal{T}_n$ be an extension of $\mathcal{T}_{n-1}$ with the same underlying tree and such that the depth at the root of $\mathcal{T}_{n}$ is \emph{$p$ times} that of $\mathcal{T}_{n-1}$. Let $a_n \in \mathbb{Z}_{> 0}$ and $a_n=:a'_n+pu_n$, $1 \le a'_n <p$. We say
\begin{itemize}
    \item $\mathcal{T}_n$ is an \textit{additive $a_n$ extension} of $\mathcal{T}_{n-1}$ if, at all but one leaf (which may assume to be $b_1$), we have $\iota_{i,n}=p\iota_{i,n-1}$, hence $\iota_{1,n}=p\iota_{1,n-1}-p+1+a_n$, and
    \item $\mathcal{T}_n$ is a \textit{minimum extension} of $\mathcal{T}_{n-1}$ if $\iota_{1,n}=p\iota_{1,n-1}-p+1$ and $\iota_{i,n}=p\iota_{i,n-1}$ for $i \neq 2$.
\end{itemize}
\end{definition}

We first show that one can always construct an additive extension for the tree $\mathcal{T}_{n-1}$. That is the extension of $\mathcal{T}_{1}(e_3)$ in Example \ref{exminimaljump}.
\begin{proposition}
\label{propconstructadditivetree}
With the above settings, there exists a Hurwitz tree $\mathcal{T}_n$ that extends $\mathcal{T}_{n-1}$ $a_n$-additively. In particular, Proposition \ref{propextendtree} holds when $\iota_n=p\iota_{n-1}-p+1+a_n$.
\end{proposition}

\begin{proof}
First, we set the decorated tree of $\mathcal{T}_{n}$ and the thickness of the edges to be that of $\mathcal{T}_{n-1}$. We then assign to $b_1$ the conductor $\iota_{1,n}=p\iota_{i,n-1}-p+a_n+1$, to each leaf $b_i$ ($i \neq 1$) a conductor $\iota_{i,n}=p\iota_{i,n-1}$, to each edge $e$ a slope $d_e:=\sum_{h \succ e} \iota_{h,n}-1$, and to the root $v_0$ the depth $p\delta_{\mathcal{T}_{n-1}}(v_0)$. At each vertex or leaf of $\mathcal{T}_n$ where the corresponding one in $\mathcal{T}_{n-1}$ has monodromy group $\mathbb{Z}/p^i$, we equip it with the monodromy group $\mathbb{Z}/p^{i+1}$. We then assign the depth at each vertex $v \neq v_0$ inductively on the positive direction starting from $v_0$ so that \ref{c5Hurwitz} is verified. Then it is straight forward to check that $\delta_{\mathcal{T}_n}(v)>p\delta_{\mathcal{T}_{n-1}}(v)$ at all vertices $v \neq v_0$. Finally, we equip each vertex $v$ away from the root with an \emph{exact} differential form as follows
\begin{equation}
    \omega_{\mathcal{T}_{n}}(v)=\frac{c_vdx}{\prod_{b \succ v}(x-[b]_v)^{\iota_{b,n}}},
\end{equation}
where $c_v \in k^{\times}$ is determined inductively from the root so that the differential conductors at all vertices are compatible. More precisely, if $e$ is an edge of $\mathcal{T}_{n-1}$, then we set the coefficient of $\omega_{\mathcal{T}_n}(t(e))$ to be the $e$-part of $\omega_{\mathcal{T}_n}(s(e))$ (resp. the coefficient of the $n$-th degeneration) when $\delta_{\mathcal{T}_n}(s(e))>0$ (resp. when $\delta_{\mathcal{T}_n}(s(e))=0$).
\end{proof}

We then construct a minimum extension. That is the extending of $\mathcal{T}_{1}$ or $\mathcal{T}_{1}(e_1)$ in Example \ref{exminimaljump}.

\begin{proposition}
\label{propconstructminimaltree}
With the above settings, there exists a Hurwitz tree $\mathcal{T}_n$ that minimally extends $\mathcal{T}_{n-1}$. In particular, Proposition \ref{propextendtree} holds for the case $\iota_n=p\iota_{n-1}-p+1$.
\end{proposition}

\begin{proof}
Suppose first that the height of the tree $\mathcal{T}_{n-1}$ is $1$. We assign conductor $\iota_{1,n}:=p\iota_{1,n-1}-p+1$ to $b_1$ and $\iota_{i,n}:=p\iota_{i,n-1}$ to $b_i$ where $i \neq 1$, making the sum of all conductors to be $\iota_n$. We then set the depth at $v_1$ to be $p\delta_{\mathcal{T}_{n-1}}(v_1)$, the slope of $e_0$ to be $\iota_n-1$, and the differential conductor at $v_1$ to be $$    \omega_{\mathcal{T}_n}(v_1)=\frac{c^pdx}{\prod_{i=1}^r(x-[b_i]_{v_1})^{\iota_{i,n}}},$$
which verifies $\mathcal{C}(\omega_{\mathcal{T}_n}(v_1))=\omega_{\mathcal{T}_{n-1}}(v_1)$ as shown in Proposition \ref{propexistenceCartiersolution}. As before, a vertex's (or leaf's) monodromy group is cyclic of order $p$ times that of the corresponding one in $\mathcal{T}_{n-1}$. The hypothesis $\iota_n=p\iota_{n-1}-p+1$ also implies that the differential conductor (resp. the $n$-th degeneration) is of the form 
\begin{equation}
    \label{eqncondattherootminimal}
     \omega_{\mathcal{T}_n(x)}(v_0)= \frac{c^pdx}{x^{\iota_n}} +\sum_{j \le \iota_{n-1}} \frac{c_j^pdx}{x^{pj-p+1}} +dg(x) \hspace{2mm} (\text{resp. } g^n(x)),
\end{equation}
where $g(x)$ (resp. $g^n(x)$) is a polynomial in $x^{-1}$ of degree at most $\iota_n-2$. That makes $v_0$ compatible with $v_1$ in $\mathcal{T}_n$ when $\delta_{\mathcal{T}_n}(v_0)>0$, completing the base case. Recall that, in our minimal jump assumption and when $\delta_{\mathcal{T}_n}(v_0)=0$, we don't need the degenerations at $v_0$ and $v_1$ to be compatible.

Let us now consider the case where the height of $\mathcal{T}_{n-1}$ is greater than one. Suppose there are $m$ edges $e_1, \ldots, e_m$ ($m \ge 2$) starting from $v_1$. We then first assign the degeneration data at $v_0$ and $v_1$ identical to the base case. Note that the Hurwitz tree $\mathcal{T}_{n-1}(e_i)$ has height at most $h(\mathcal{T}_{n-1})-1$. We thus can extend $\mathcal{T}_{n-1}(e_1)$ to a minimum tree by induction, and $\mathcal{T}_{n-1}(e_i)$, $i \neq 1$, to an additive tree as in Proposition \ref{propconstructadditivetree}, all with depth $\delta_{\mathcal{T}_n}(v_1)$. The conductor of $\mathcal{T}_{n}(e_1)$ (resp. $\mathcal{T}_{n}(e_i)$) is $\iota_{n,e_1}:=p\mathfrak{C}_{\mathcal{T}_{n-1}}(e_1)-p+1$ (resp. $\iota_{n,e_i}:=p\mathfrak{C}_{\mathcal{T}_{n-1}}(e_i)$). Hence, the sum of conductors are $\iota_n$. That completes the construction of the tree $\mathcal{T}_n$. 
\end{proof}

We have just proved Proposition \ref{propextendtree}. One can also do the following alternative construction, which is utilized in Example \ref{exextensionnonessential}.

\begin{proposition}
\label{propconstructreenodiff}
There exists a $\mathbb{Z}/p^n$-tree $\mathcal{T}_{n}$ that extends $\mathcal{T}_{n-1}$ with no essential jumps.
\end{proposition}

\begin{proof}
Suppose there are $m$ edges $e_1, \ldots, e_m$ succeeding $v_1$ with conductors $\mathfrak{C}_1, \ldots, \mathfrak{C}_m$, respectively, on $\mathcal{T}_{n-1}$. Hence $\sum_{i=1}^m \mathfrak{C}_i=\iota_{n-1}$. Set $\mathfrak{C}'_1:=p\mathfrak{C}_1-p+a'_n$, $\mathfrak{C}'_2:=p\mathfrak{C}_2-p+1$, and $\mathfrak{C}'_i:=p\mathfrak{C}_i$. We first assign $p\delta_{\mathcal{T}_{n-1}}(v_1)$ as the depth at $v_1$ on $\mathcal{T}_n$, and the non-exact differential form $$\omega_{\mathcal{T}_n}(v_1)=\frac{C_{v_1}dx}{\prod_{i=1}^s(x-[b_i]_{v_1})^{\mathfrak{C}'_i}},$$ that verifies $\mathcal{C}(\omega_{\mathcal{T}_n}(v_1))=\omega_{\mathcal{T}_{n-1}}(v_1)$ ( which exists by Proposition \ref{propexistenceCartiersolution}) as the differential conductor at $v_1$.  We then extend $\mathcal{T}_{n-1}(e_1)$ to an additive-$a'_n$-tree, $\mathcal{T}_{n-1}(e_2)$ to a minimum tree, and each $\mathcal{T}_{n-1}(e_i)$ where $i \neq 1,2$ to an additive-$(p-1)$-tree, all with depth $p\delta_{\mathcal{T}_{n-1}}(v_1)$, and so that they are compatible with the differential conductor $\omega_{\mathcal{T}_n}(v_1)$. That makes the current sum of conductors to be $\mathfrak{C}(v_1):=p\iota_{n-1}-2p+a'_n+1$ which is strictly smaller than $p\iota_{n-1}-p+1$ and not congruent to $1$ modulo $p$. Let us now consider the trunk $e_0$ of $\mathcal{T}_{n-1}$. A simple calculation shows that one can break down the edge $e_0$ into two rational ones, called $e_{0,1}$ and $e_{0,2}$, so that $(\mathfrak{C}(v_1)-1)\epsilon_{e_{0,2}}+(\iota_n-1)\epsilon_{e_{0,1}}=p\delta_{\mathcal{T}_{n-1}}$. We then assign to $t(e_{0,1})$ on $\mathcal{T}_n$ the exact differential $$\omega_{\mathcal{T}_n}(t(e_{0,1}))=\frac{-C_0dx}{(\iota_{n}-1)x^{\mathfrak{C}(v_1)}(x-a)^{p(u_n+1)}},$$ where $a:=\big(-C_0 /((\iota_n-1)\mathfrak{C}(v_1))\big)^{1/(p(u_n+1))}$. A straightforward computation shows that the choice of $a$ makes $v_0$, $v_1$, and $t(e_{0,1})$ compatible. We then glue a leaf of conductor $p(u_n+1)$ and monodromy group $\mathbb{Z}/p^{n-1}$ to $t(e_{0,1})$. Finally, we set the remaining degeneration data at $e_{0,1}$, $e_{0,2}$, and $v_0$ in the obvious manner, completing the proof.
\end{proof}

\bibliographystyle{acm}
\bibliography{mybib}

\end{document}